\documentclass[12pt,letter,reqno]{amsart}
\usepackage[utf8]{inputenc}

\usepackage{amsthm,amssymb, amsmath, amsfonts, amssymb}
\usepackage{mathtools}
\usepackage{dsfont}
\usepackage{bm}
\usepackage{soul}
\usepackage[normalem]{ulem}

\usepackage[shortlabels]{enumitem}
\usepackage{hyperref}
\hypersetup{colorlinks = true, linkcolor = blue, citecolor = blue, urlcolor = blue}
\usepackage{cite}
\allowdisplaybreaks
\usepackage[usenames,dvipsnames]{xcolor}
\usepackage{cleveref}

\usepackage[margin=1in]{geometry}
\usepackage[textsize=small]{todonotes}

\usepackage{bookmark}

\usepackage{graphicx}

\usepackage{tikz}

\usetikzlibrary{calc}
\usetikzlibrary{positioning}

\theoremstyle{plain}
\newtheorem{theorem}{Theorem}[section]
\newtheorem*{theorem*}{Theorem}

\newtheorem{corollary}[theorem]{Corollary}
\newtheorem{lemma}[theorem]{Lemma}
\newtheorem{observation}[theorem]{Observation}
\newtheorem{proposition}[theorem]{Proposition}
\newtheorem{definition}[theorem]{Definition}
\newtheorem{claim}[theorem]{Claim}
\newtheorem*{claim*}{Claim}

\newtheorem{remark}[theorem]{Remark}

\usepackage{thmtools}

\newtheorem{innercustomlem}{Lemma}
\newenvironment{customlem}[1]
  {\renewcommand\theinnercustomlem{#1}\innercustomlem}
  {\endinnercustomlem}

\newcommand{\E}{\mathbb{E}}
\newcommand{\Var}{\mathrm{Var}}

\newcommand{\R}{\mathbb{R}}

\newcommand{\ind}[1]{\mathds{1}_{#1}}
\newcommand{\Prob}{\mathbb{P}}

\newcommand*{\cP}{\mathcal{P}}
\newcommand{\cA}{\mathcal{A}}

\def\lam{\lambda}
\def\sig{\sigma}

\def\eps{\varepsilon}

\newcommand{\tv}{\mathrm{TV}}
\def\tmix{t_{\mathrm{mix}}}

\def\gap{\mathrm{gap}}

\def\bg{{\bf G}}
\def\hbg{\hat{\bf G}}
\def\T{\mathbb T}
\def\cE{\mathcal E}

\def\hsig{\hat{\sigma}}

\def\weak{\xrightarrow[]{\mathrm{loc}}}
\def\tP{\tilde{P}}

\def\GK{\mathrm{GK}}

\def\Zising{Z^{\mathsf{Ising}}}
\def\Zanti{Z^{\mathsf{anti}}}
\def\Zfix{Z^{\mathsf{fix}}}
\def\muising{\mu^{\mathsf{Ising}}}
\def\mufix{\mu^{\mathsf{fix}}}
\def\muanti{\mu^{\mathsf{anti}}}

\date{November 20, 2025}

\title[Fixed-magnetization Ising on random graphs]{Fixed-magnetization Ising on random graphs \\ up to reconstruction}

\author{Reza Gheissari}
\address{Department of Mathematics \\ Northwestern University}
\email{gheissari@northwestern.edu}

\author{Will Perkins}
\address{School of Computer Science \\ Georgia Institute of Technology}
\email{math@willperkins.org}

\author{Corrine Yap}
\address{School of Mathematics \\ Georgia Institute of Technology}
\email{math@corrineyap.com}

\begin{document}

\begin{abstract}
    We study the fixed-magnetization ferromagnetic Ising model on random $d$-regular graphs for $d\ge 3$ and inverse temperature below the  tree reconstruction threshold. Our main result is that for each magnetization $\eta$, the free energy density of the fixed-magnetization Ising model converges to the annealed free energy density, itself   the Bethe free energy of an Ising measure on the infinite $d$-regular tree. Moreover, the fixed-magnetization Ising model exhibits local weak convergence to this tree measure. A key challenge to proving these results is that for magnetizations between the model's spinodal points, the limiting tree measure corresponds to an unstable fixed point of the belief propagation equations.

    As an application, we prove that the positive-temperature Zdeborov{\'a}--Boettcher conjecture on max-cut and min-bisection holds up to the reconstruction threshold: on the random $d$-regular graph, the expected fraction of bichromatic edges in the anti-ferromagnetic Ising model plus the expected fraction  of bichromatic edges in the zero-magnetization ferromagnetic Ising model  equals $1+o(1)$.

    A second application is completely determining the large deviation rate function for the magnetization in the Ising model on the random regular graph up to  reconstruction.

    Finally, we use the precise understanding of this rate function to show that the Glauber dynamics for the full Ising model on the random graph mixes in sub-exponential time from uniformly random initialization, well into the non-uniqueness regime where the worst-case initialization mixing time is exponentially slow. 
\end{abstract}

\maketitle

\section{Introduction}
\label{secIntro}

The ferromagnetic Ising model at inverse temperature $\beta$ (without external field) on a graph $G = (V,E)$ is the probability distribution on $\{\pm 1\}^V$ defined by 
\begin{align*}
    \mu_{G,\beta}^{\mathsf{Ising}}(\sigma) = \frac{e^{ \beta H_G(\sigma)}}{ Z^{\mathsf{Ising}}_{G,\beta} }\,, \qquad \text{where} \qquad Z^{\mathsf{Ising}}_{G,\beta} = \sum_{\sigma \in \{ \pm 1\}^V} e^{\beta H_G(\sigma)}\,
\end{align*}
is the partition function (the normalizing constant) and 
 $H(\sigma) = H_G(\sigma)$ is the number of monochromatic edges of $G$ under the $\pm 1$ coloring $\sigma$; i.e., $H_G(\sigma) = \sum_{(u,v) \in E(G)} \ind{\sigma_u=\sigma_v}$.  When $\beta \ge 0$, monochromatic edges are preferred and the model is called ``ferromagnetic."  We will also refer to Ising models with an external field $h$, where $\Zising_{G,\beta,h} = \sum_{\sigma \in \{ \pm 1\}^V} e^{h\sum_v \sigma_v + \beta H_G(\sigma)}$.

We focus in this paper on the ferromagnetic Ising models on $\mathbf{G}_d(n)$, the uniformly random $d$-regular graph, for $d\ge 3$ constant as $n\to \infty$. The typical behavior of the Ising model on $ \mathbf{G}_d(n)$ is at this point very well understood.  In particular, answers to two main questions are now known:  what is the typical (approximate) value of the partition function $\Zising_{\mathbf{G}_d(n),\beta}$ and how is a typical sample from $\muising_{\mathbf{G}_d(n),\beta}$ distributed, from the local perspective of a uniformly chosen vertex?  Using standard concentration arguments, the first question amounts to computing the \textit{free energy density}, $\lim_{n \to \infty} \frac{1}{n}\E \log \Zising_{\mathbf{G}_d(n),\beta}$, as a function of $d, \beta$.  The second question can be phrased in terms of local weak convergence of probability measures.  The answers to both questions have essential connections to translation-invariant, infinite-volume Gibbs measures (or Belief Propagation fixed points) defined on the infinite $d$-regular tree $\mathbb T_d$.  

Gerschenfeld and Montanari~\cite{GM07} showed that the free energy density of the Ising model on $G\sim \mathbf{G}_d(n)$ converges as $n\to\infty$ to that of the corresponding infinite $d$-regular tree $\mathbb T_d$. This in particular implies that it undergoes a second-order phase transition at $\beta_c = \beta_c(d) = 2\tanh^{-1}(1/(d-1))$ (asymptotic to $2/d$ as $d\to\infty$). Further work~\cite{montanari2012weak} has shown that the local weak limit of the configuration on $G$ is the (unique) Ising Gibbs measure on $\mathbb T_d$ when $\beta\le \beta_c$ and a $\frac{1}{2}$-$\frac{1}{2}$ mixture of the ``plus" and ``minus" infinite-volume Gibbs measures on $\mathbb T_d$ when $\beta>\beta_c$, and when the model is conditioned to have a positive magnetization, it converges specifically to the ``plus" measure on the infinite tree.

In this paper, we study the fixed-magnetization Ising model, which is the Ising model conditioned on the event that the magnetization $\sum_v \sigma_v$ equals some number $k\in \{-n,...,n\}$. We parameterize this by a real number $\eta \in [-1,1]$, the magnetization density.  In particular, given $\eta \in[-1,1]$ and a graph $G=(V,E)$, let $\Omega_\eta = \{  \sigma \in \{ \pm 1 \}^V :  \lfloor \frac{\sum_v \sigma_v}{n} \rfloor = \eta \}$; we call these ``configurations of  magnetization $\eta$.''  The fixed-magnetization Ising model is then the probability distribution on $\Omega_\eta$ defined by 
\begin{align*}
     \mufix_{G,\beta,\eta}(\sigma) = \frac{ e^{\beta H_G(\sigma)}}{\Zfix_{G,\beta}(\eta)}\,, \qquad \text{where} \qquad 
    \Zfix_{G,\beta}(\eta) = \sum_{\sigma \in \Omega_\eta}  e^{\beta H_G(\sigma)}\,
\end{align*}
is the fixed-magnetization Ising partition function.  Depending on the graph and the choice of parameters, the fixed-magnetization Ising model  can either behave very similarly to the Ising model or exhibit dramatically different behavior, including glassy behavior at large $\beta$~\cite{mezard1987mean,DMS} and computational hardness of approximate counting and sampling~\cite{CDKP21}.  The model (at zero magnetization) also has a combinatorial and algorithmic interpretation as a ``soft" version of the minimum bisection problem.

 We answer both of the main questions stated above for \textit{all} magnetizations when the inverse temperature is below the  \textit{tree reconstruction threshold}, 
 \begin{equation}
     \label{eq:betardef}
     \beta_r= \beta_r(d) = \log \left(\frac{\sqrt{d-1}+1}{\sqrt{d-1}-1} \right) =(2+o_d(1))d^{-1/2} \,. 
 \end{equation}
 This threshold is also known as the Kesten--Stigum threshold~\cite{kesten1966additional}; we define the tree reconstruction problem precisely in Section~\ref{secTrees}.
 We determine the free energy density for every $\eta \in[-1,1]$, which immediately gives the large deviation rate function for the magnetization in the Ising model on the random graph.  We also prove local weak convergence of the fixed-magnetization Ising model to a particular translation-invariant Ising Gibbs measure on $\mathbb T_d$ with appropriately chosen external field. The key conceptual difficulty in proving such a result is that unlike previous local weak convergence results for spin systems on $G\sim \bg_d(n)$, the tree measure to which the fixed-magnetization model converges can be an \emph{unstable} BP fixed point.  See Section~\ref{secOutline} for an overview of the techniques. As a byproduct of our proof, we in particular show that the fixed-magnetization Ising model on $G\sim \bg_d(n)$ is replica symmetric for all $\beta<\beta_r$. These results are presented in Section~\ref{subsec:freeenergyIntro}.

 Beyond its intrinsic interest, a sharp understanding of the Gibbs measure at atypical magnetizations also helps with understanding the behavior of Markov chains like Glauber dynamics at low temperatures. Specifically, a long-standing question in the study of low-temperature Markov chains is that of \emph{phase ordering}. Phase ordering asks: when $\beta>\beta_c$ and mixing is slow due to the coexistence of plus and minus phases, is low-temperature Glauber dynamics started from a balanced initialization (say uniform-at-random assignments of $\{\pm 1\}$ to $V$) able to avoid the complexity of near-zero-magnetization configurations (where the model behaves like a spin glass) and quickly pick one of the two dominant phases to equilibrate to? We make progress in this direction by combining our computation of free energies of the fixed-magnetization Ising model on $\mathbf{G}_d(n)$ with recent results of~\cite{bauerschmidt2023kawasaki} on Kawasaki dynamics for the same model. Namely, we show that Glauber dynamics for the Ising model initialized from any spin-flip symmetric distribution mixes in sub-exponential time at all $\beta< \frac{1}{4 \sqrt{d-1}}$ (well into the low-temperature regime where the worst-case initialization mixing time is exponential).  To our knowledge, this is the first setting with non-trivial geometry where Glauber dynamics are shown to mix much faster from a uniformly random start than they do from a worst-case start.
This result is presented in detail in~\Cref{secIntroMC}.

\subsection{Free energy and local weak limit of fixed magnetization Ising}\label{subsec:freeenergyIntro}

We begin by considering the typical behavior of the free energy $\frac{1}{n} \log \Zfix_{G,\beta}(\eta)$ for $G\sim \mathbf{G}_d(n)$ where $d \ge 3$, $\beta \ge 0$, and $\eta\in [-1,1]$ are fixed as $n \to \infty$. By standard concentration arguments, it suffices to understand  the \textit{free energy density}  
\begin{equation}
    \lim_{n \to \infty} \frac{1}{n} \mathbb E \log \Zfix_{\mathbf G_d(n),\beta} ( \eta)
\end{equation}
as a function of $d, \beta, \eta$.  Even the existence of this limit is highly non-trivial (the interpolation arguments of~\cite{bayati2010combinatorial} do not apply) and to the best of our knowledge is only known in cases where the value itself can be computed.

Our first main result establishes the free energy density of the fixed-magnetization Ising model for all magnetizations $\eta$ when the inverse temperature $\beta$ is below the tree reconstruction threshold $\beta_r$ defined in~\eqref{eq:betardef}.  The limit coincides with the \textit{annealed free energy density}, 
\begin{equation}
\label{eqn:firstmomentDef}
    \lim_{n \to \infty} \frac{1}{n}\log  \mathbb E  \Zfix_{\mathbf G_d(n),\beta}(\eta) = f_{d,\beta}(\eta)\,,
\end{equation}
where $f_{d,\beta}(\eta)$ has a closed-form but rather complicated expression that we write below in~\eqref{eq:f-closed-form} in Section~\ref{secTrees}. 
For the case $\eta =0$, we have the simple expression
\begin{equation}
\label{eqfdb0}
    f_{d,\beta}(0)= \log 2 + \frac{d}{2} \log \left( \frac{1+ e^{ \beta}}{2} \right ) \,.
\end{equation}
In what follows, we say an event $A_n$ holds ``with high probability'' if $\Prob(A_n) \to 1$ as $n\to \infty$.
\begin{theorem}
    \label{thmFreeEnergy}
    Fix $d \ge 3$ and suppose $0 \le \beta < \beta_r$.  Then with high probability  over $G \sim \bg_d(n)$ it holds for every $\eta \in [-1,1]$ that 
         $$\frac{1}{n}\mathbb  \log \Zfix_{G,\beta}(\eta) = f_{d,\beta}(\eta) +o(1) \,.$$
\end{theorem}

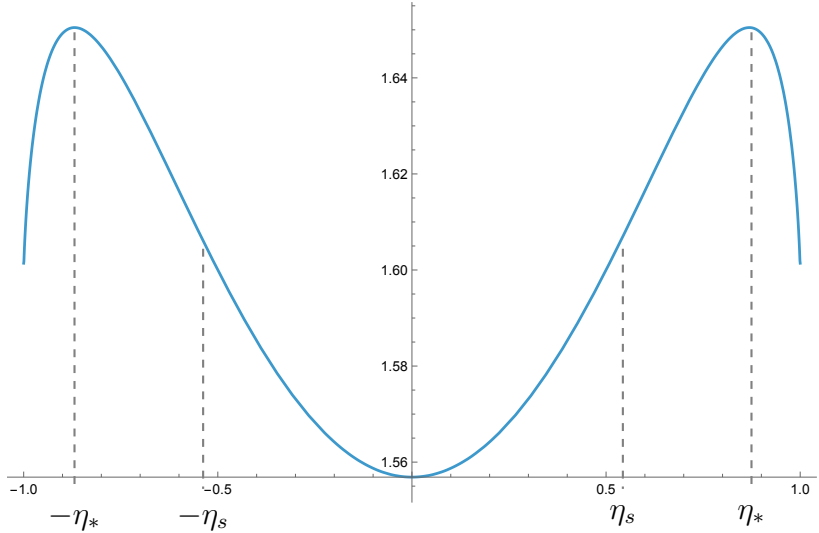
\begin{figure}
    \centering
    \begin{tikzpicture}[scale=0.9]
        \node[anchor=south, inner sep=0] (img) at (0,0)
            {\includegraphics[width=10.8cm]{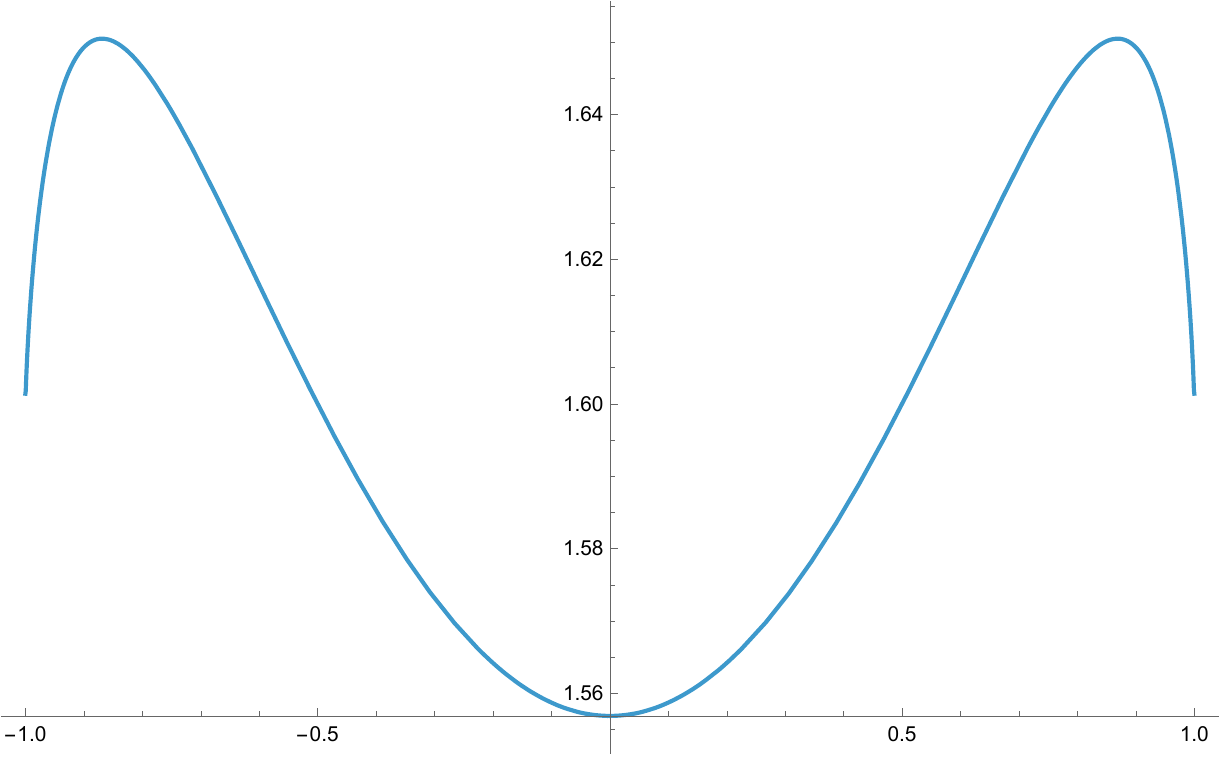}};

        \def\xMaxCm{5.0}  
        \def\yMaxCm{7.0}  
        
        \def\xInflCm{3.1} 
        \def\yInflCm{3.8} 

        \def\yMinCm{0.25}
        

        \draw [dashed, thick, gray] (\xMaxCm, \yMaxCm) -- (\xMaxCm, \yMinCm);
        \draw [dashed, thick, gray] (-\xMaxCm, \yMaxCm) -- (-\xMaxCm, \yMinCm);

        \draw [dashed, thick, gray] (\xInflCm, \yInflCm) -- (\xInflCm, \yMinCm);
        \draw [dashed, thick, gray] (-\xInflCm, \yInflCm) -- (-\xInflCm, \yMinCm);

        \node [anchor=north, yshift=-2pt] at (\xMaxCm, \yMinCm) {$\eta_*$};
        \node [anchor=north, yshift=-2pt] at (\xInflCm, \yMinCm) {$\eta_s$};
        \node [anchor=north, yshift=-2pt] at (-\xMaxCm, \yMinCm) {$-\eta_*$};
        \node [anchor=north, yshift=-2pt] at (-\xInflCm, \yMinCm) {$-\eta_s$};

    \end{tikzpicture}
    \caption{For $d=10$, $\beta=.32 \in (\beta_c,\beta_r)$, the plot of $f_{d,\beta}(\eta)$ as a function of~$\eta$. The typical magnetization is $\eta_*$ and the spinodal point $\eta_s$ is the smallest magnetization achievable by a \emph{stable} BP fixed point for some external field.}
    \label{fig:FdbZero}
\end{figure}

Theorem~\ref{thmFreeEnergy} shows that for $\beta< \beta_r$, the fixed-magnetization Ising model is \textit{replica symmetric} at all magnetizations, in the sense that one can exchange the expectation and logarithm in the formulas above, giving an equality between the free energy density and the annealed free energy density which can be interpreted as the ``Bethe free energy" of a certain Gibbs measure on $\mathbb T_d$ (see e.g.~\cite{mezard2009information} for background).  We expect that this result fails for $\eta=0$ and any $\beta>\beta_r$ (and for other values of $\eta$ at larger values of $\beta$).  The results of~\cite{DMS} strongly suggest this (at least for large $d$).  To be precise, we conjecture that for each $\beta>\beta_r$, there exists $\eps>0$ so that $\limsup \frac{1}{n} \E \log \Zfix_{\mathbf G_d(n),\beta} ( 0)  \le f_{d,\beta}(0) -\eps$.  In~\cite{coja2022ising}, such a result is proved for the anti-ferromagnetic Ising model, but this relies on the results of~\cite{coja2018information} which do not apply to the ferromagnetic Ising model.

The function $f_{d,\beta}(\eta)$ is plotted in Figure~\ref{fig:FdbZero}, for $d=10, \beta=.32$, which is between $\beta_c$ and $\beta_r$.  The features of the plot are generic to the non-uniqueness regime $\beta>\beta_c$: two (symmetric) global maxima at $\pm \eta^\ast$ and a local minimum at $0$. Between the global maxima and the local minimum there are two inflection (or spinodal) points $\pm \eta_s$; as we will describe in Section~\ref{secOutline}, it is for magnetizations between these two points that our new methods are needed.

Combined with the results of~\cite{coja2022ising}, Theorem~\ref{thmFreeEnergy} establishes the positive temperature version of the conjecture of Zdeborov{\'a} and Boettcher~\cite{zdeborova2010conjecture} on max-cut and min-bisection in random regular graphs up to reconstruction.  Their famous conjecture is that in $\mathbf G_d(n)$ the expected size of the max-cut of $G$ plus the expected size of the min-bisection is $dn/2 +o(n)$. What makes the conjecture especially intriguing is that it is not straightforwardly combinatorial: the complement of a max-cut is not close to a min-bisection. Rather the conjecture comes from the form of the non-rigorous cavity method prediction for the free energy densities in the two models.  An approximate version of the conjecture has been established  in the $d \to \infty$ limit by Dembo, Montanari, and Sen~\cite{DMS}, but it remains open for any fixed $d$.  The positive temperature version of the conjecture~\cite[Section 6]{zdeborova2010conjecture} states that for any $\beta$, the expected number of monochromatic edges in the anti-ferromagnetic Ising model plus the expected number of monochromatic edges in the zero-magnetization ferromagnetic Ising model (at the same inverse temperature $\beta$) is $dn/2 +o(n)$.  Equivalently, the sum of expected numbers of bichromatic edges is  $dn/2 +o(n)$, which in the $\beta \to \infty$ limit immediately gives their max-cut/min-bisection conjecture.  Combining the free energy density formula from Theorem~\ref{thmFreeEnergy} with the formula for the anti-ferromagnetic case from~\cite{coja2022ising} (and taking derivatives in $\beta$), we prove this  conjecture for $\beta<\beta_r$. 

Define the anti-ferromagnetic Ising model at inverse temperature $\beta$ as the measure $\muanti_{G,\beta}$ with partition function $\Zanti_{G,\beta}= \sum_{\sigma} e^{-\beta H_G(\sigma)}$, recalling that $H_G(\sigma)$ counts the number of monochromatic edges of $G$ under $\sigma$.
\begin{corollary}
    \label{corZBconj}
    For $d \ge 3$, $0\le \beta < \beta_r$, the following identity holds with high probability  over $G \sim \bg_d(n)$:
    \begin{equation}
        \E_{\muanti_{G,\beta}}  H_G +  \E_{ \mufix_{G,\beta,0}} H_G= \frac{dn}{2}(1 +o(1)) \,.
    \end{equation}
\end{corollary}

A second corollary of Theorem~\ref{thmFreeEnergy} is the complete determination of the large deviation rate function for atypical magnetizations in the  Ising model on the random regular graph.  
\begin{corollary}
    \label{CorLDrate}
    Fix $d \ge 3$, $0 \le \beta <\beta_r$, and $\eta \in [-1,1]$. Then, in probability, 
    \begin{align*}
    \lim_{n \to \infty} \frac{1}{n} \log \mu_{\mathbf G_d(n),\beta}^{\mathsf{Ising}} \left (  \sigma \in \Omega_\eta \right) = f_{d,\beta}(\eta) - \max_{\eta' \in [-1,1]} f_{d,\beta}(\eta')\,.
    \end{align*}
\end{corollary}
The formula on the right-hand side is the free energy density of the fixed-magnetization Ising model at the desired magnetization minus the free energy density of the  Ising model itself.  Such a result was known previously (in~\cite{kuchukova2024fast}, following the techniques of~\cite{Coja-Oghlan22}) when $\eta$ was sufficiently far from $0$ (as a function of $\beta$), and would follow by those arguments also for all $\eta$ for $\beta<\beta_c$ but significant new innovations are needed for the regime of $\beta \in (\beta_c,\beta_r)$ and $\eta$ between the two spinodal points $\pm \eta_s$.

Our next result determines the local weak limit of the fixed-magnetization Ising model on random graphs below the reconstruction threshold. That is, we characterize the distribution that the fixed-magnetization model induces on the local neighborhood  of a typical vertex.

To state the result, we recall the notion of local weak convergence of the Ising distribution on $\bg_d(n)$ (used in a similar contexts in e.g.~\cite{aldous2004objective,dembo2010ising,montanari2012weak}). 

For a graph $G$ and a vertex $v \in V(G)$, let $B_{G,r}(v)$ denote the depth-$r$ neighborhood of $v$.
\begin{definition}
Let $G_n$ be a random sequence of $d$-regular graphs and let $\mu_n$ be a sequence of probability measures on $\{-1,1\}^{V(G_n)}$.  Let $\mu_\infty$ be a probability measure on $\{-1,1\}^{V(\mathbb T_d)}$.  We say $\mu_n$ converges locally weakly to $\mu_\infty$, denoted $\mu_n  \weak \mu_{\infty}$, if for all $r \ge 1$ and $\eps>0$, 
\begin{align*}
&\Prob(B_{G_n,r}(v) \cong B_{\mathbb T_d,r}(v_0)) = 1-o(1)\,,
\intertext{with $\cong$ meaning graph isomorphic, and with that isomorphism,}
  &  \Prob(d_{\tv}(\mu_{n}(\sigma(B_{G_n,r}(v) \in \cdot)), \mu_{\infty}(\sigma(B_{\mathbb T_d,r}(v_0) \in \cdot))) > \eps) =  o(1) \,,
\end{align*}
    where $v_0$ denotes the root of $\mathbb T_d$. Here, the probabilities are over the choice of $G_n$ and  a uniformly random $v \in V(G_n)$. 
\end{definition}
That is, $G_n$ is locally treelike (typical depth-$r$ neighborhoods are trees) and for all fixed $r$, the distribution induced by $\mu_n$ on the depth-$r$ neighborhood of a random $v$ converges to the distribution induced by $\mu_{\infty}$ on the depth-$r$ neighborhood of the root of $\mathbb T_d$.

We also must define the limiting measures, which are Ising measures on the tree $\mathbb T_d$. Even though the fixed-magnetization model $\mu_{\bg_d(n),\beta,\eta}^{\mathsf{fix}}$ has no external field, we must allow for a field $h$ for the family of limiting tree measures that will arise. For a given choice of $\beta$ and $h$, an Ising model can be defined on an infinite graph such as $\mathbb T_d$ via finite-volume marginals and the Dobrushin--Lanford--Ruelle consistency equations (see~\cite{georgii2011gibbs}).  An Ising measure $\nu$ on $\mathbb T_d$ is \textit{translation invariant} if the single-vertex magnetization $\E_{\nu} [\sigma_v]$ is constant over $v \in V(\mathbb T_d)$.  In Section~\ref{secTrees}, we recall that for a given $d, \beta, h$ there is at least one and at most three translation-invariant Ising measures on $\mathbb T_d$: at $h=0$ the onset of non-uniqueness occurs at $\beta_c$. Let $\mathcal M_{d,\beta,h}$  be the set of all translation-invariant Ising Gibbs measures on $\mathbb T_d$ with inverse temperature $\beta$ and external field $h$, where $h$ ranges over all values in $\mathbb R$.

\begin{lemma}
\label{lem:existenceTreeMeasure}
Fix $d \ge3$, $\beta \ge 0$.  For every $\eta \in (-1,1)$ there is a unique $h_\eta$ and a unique  measure $\nu_{\eta} \in \mathcal M_{d,\beta,h_\eta}$ so that $\E_{\nu_{\eta}} [\sigma_v] = \eta$ for each vertex $v $ of $\mathbb T_d$. 
\end{lemma}

We establish that the fixed-magnetization Ising model on the random graph converges locally to this appropriately chosen measure.
\begin{theorem}
    \label{thmLocalWeak}
    Fix $d \ge 3$ and suppose $0 \le \beta < \beta_r$.  Then for every $\eta \in (-1,1)$,
    \begin{equation*}
        \mu_{\mathbf G_d(n),\beta,\eta}^{\mathsf{fix}} \weak \nu_{\eta} \,.
    \end{equation*}
\end{theorem}

Note that for the endpoints $\eta \in \{-1,1\}$, the fixed-magnetization Ising configuration is simply the all-plus (or all-minus) configuration, and converges to the all-plus (or all-minus) configuration on $\T_d$  by local weak convergence of $\bg_d(n)$ to $\T_d$. Therefore, in the proofs of Theorem~\ref{thmFreeEnergy}--\ref{thmLocalWeak}, we focus on the non-trivial cases $\eta\in (-1,1)$.

We mention in passing here (saving a precise definition for Section~\ref{secTrees}) that the tree reconstruction threshold $\beta_r$ can be defined in terms of $\nu_0$: it is the supremum over $\beta$ so that the distribution of spins  at depth $L$ induced by $\nu_0$ contains vanishing (in $L$) information about the spin at the root.

\subsection{Ising Glauber dynamics from a uniformly random initialization}
\label{secIntroMC}
We next turn to the consequence for low-temperature Markov chain behavior. The
Glauber dynamics for the Ising model $(X_t)_{t\ge 0}$ on $G$ is the following discrete-time Markov chain. Initialized from some $X_0\in \{\pm 1\}^V$, generate $(X_t)_{t\ge 0}$ iteratively as follows: given $X_{t-1}$, generate $X_t$ by 
\begin{enumerate}
    \item Picking a vertex $v$ uniformly at random from $V$; 
    \item Setting $X_{t}(w) = X_{t-1}(w)$ for $w\ne v$ and resampling 
    \begin{align*}
        X_t(v) \sim \mu_{G,\beta}(\sigma_v \in \cdot\mid (\sigma_w)_{w\ne v}= (X_{t-1}(w))_{w\ne v})\,.
    \end{align*}
\end{enumerate}
This Markov chain is reversible with respect to the Ising distribution $\mu_{G,\beta}^{\mathsf{Ising}}$. It is of significant interest both as the canonical Markov chain Monte Carlo sampler from the Ising model, and a physical model for the out-of-equilibrium evolution of thermodynamic systems. The quantity of most interest is the mixing time, which is the time it takes to be within $\varepsilon$ total variation distance of the stationary distribution, and in particular its asymptotics as a function of $n$.

The mixing time of Ising Glauber dynamics on $\mathbf{G}_d(n)$ reflects the equilibrium phase transition on $\mathbb T_d$. When $\beta<\beta_c$, the Glauber dynamics mix from worst-case initialization in optimal $O(n\log n)$ time~\cite{mossel2013exact} and in polynomial time at the critical temperature $\beta = \beta_c$~\cite{bauerschmidt2023stochastic}. However, when $\beta>\beta_c$, there is an exponential bottleneck between the plus phase where the magnetization is positive and the minus phase where the magnetization is negative~\cite{GM07,dembo2010ising,CvdHK21}. This bottleneck induces an exponentially slow mixing time (from worst-case initialization).

An important question is to understand in what sense, and for what regime of $\beta >\beta_c$, this magnetization bottleneck is the \emph{only} (at least physically relevant) obstruction to mixing. This question can be formalized by asking if spin-flip symmetric initializations (or at least nice ones like the uniform-at-random initialization) circumvent the exponential bottlenecks and equilibrate in polynomial time. Such problems have a long history going under the name of phase ordering on the lattice (see the monograph~\cite{bray1994theory}). While this can be analyzed exactly on the complete graph~\cite{LLP,DLP-censored-Glauber}, as soon as the graph has geometry, analysis of the diffusion away from zero magnetizations towards the phases becomes significantly more complicated. 

Recently, Bauerschmidt,  Bodineau,  and Dagallier~\cite{bauerschmidt2023kawasaki} made progress in a related direction for $G\sim \mathbf{G}_d(n)$ by showing that the restricted dynamics that moves in a fixed-magnetization slice $\Omega_\eta$ (done via a conservative variant of Glauber dynamics known as \textit{Kawasaki dynamics}) is fast mixing for all $\eta \in [-1,1]$ and $\beta<\frac{1}{4\sqrt{d-1}}$. Note that for large $d$ ($d\ge 66$ suffices) this is in the low-temperature $\beta>\beta_c\asymp \frac{2}{d-1}$ regime of the Ising model on $G\sim \bg_d(n)$. 

As a consequence of Theorem~\ref{thmFreeEnergy}, we are able to show that the sequence of free energies associated to the slices of magnetizations $\Zfix_{\bg_d(n),\beta}(\eta)$ for $\eta \in [0,1]$ is (approximately) unimodal. This is combined with the fast mixing of restricted dynamics to a slice via Markov chain decomposition techniques to show that the mixing time is significantly faster from spin-flip symmetric initializations than it is from worst-case.

\begin{theorem}
    \label{thmIsingMixing}
  Fix $d \ge 3$ and $\beta < \frac{1}{4 \sqrt{d-1}}$.  There exists $C> 0$ so that the following is true.  Let $\lambda_0$ be the uniform distribution on $\{ \pm 1 \}^n$ and let $\lambda_t$ be the distribution of $X_t$, after $t$ steps of the Ising Glauber dynamics starting from $X_0 \sim \lambda_0$.  With probability $1-o(1)$, $G\sim \mathbf G_d(n)$ is such that for all $\varepsilon>0$, 
  \begin{equation*}
      d_{\tv}(\lambda_t,\mu_{G, \beta}^{\mathsf{Ising}})  \le \eps 
  \end{equation*}
  for all $t \ge T_n(\eps)$ where $T_n(\eps)$ is a sub-exponential sequence $T_n(\eps) = e^{o(n)}\log(1/\eps)$. 
\end{theorem}

Note that the distribution $\lambda_t$ includes the randomness of the starting configuration $X_0$. We in fact prove the stronger statement that the mixing time from any initial distribution that is invariant under global spin flip has sub-exponential mixing: see Corollary~\ref{cor:fast-from-spin-flip-symmetric-initialization}. In particular, we show that with high probability over $G\sim \bg_d(n)$, the generator of the Ising Glauber dynamics only has one exponentially small, i.e., $\exp(-\Omega(n))$, eigenvalue. 
This formalizes that at least in the portion of the low-temperature regime up to the reconstruction threshold, the global choice of majority spin is the only $\exp(\Omega(n))$ bottleneck to fast mixing. 

The result of~\cite{bauerschmidt2023kawasaki} on the Kawasaki dynamics is expected to hold all the way up to the reconstruction threshold $\beta_r$. If this were the case, then our Theorem~\ref{thmIsingMixing} would also hold up to $\beta_r$ because its equilibrium inputs from Theorems~\ref{thmFreeEnergy} and~\ref{thmLocalWeak} both do. We mention as added evidence for this the results of~\cite{KoehlerLeePMLR}, which (among other things) approximated $\mu_{\bg,\beta}^{\mathsf{Ising}}$ by a mixture of two high-temperature distributions for $\beta<\beta_r$ and used this to devise a fast simulated tempering scheme on an extended state space for the model.

Notably one may not expect Theorem~\ref{thmIsingMixing} to also hold for all $\beta >\beta_r$, as at sufficiently low temperatures, the glassiness of near-zero magnetization configurations (see e.g.~\cite{DMS} for a connection between the lowest fixed-magnetization energy configurations and the Sherrington--Kirkpatrick model) may become dynamically relevant. The recent paper~\cite{gheissari2025rapid} demonstrated that if the initial magnetization is strictly bounded away from zero, say $\eta \gtrsim \frac{1}{\sqrt{d}}$, then in the low-temperature regime of $\beta \gtrsim \frac{1}{\sqrt{d}}$, mixing within the plus phase is fast, but it does not address escaping the set of near-zero magnetizations. The work~\cite{GS22} previously showed fast mixing to the plus phase measure specifically when initialized from the all-plus state.

Further, one would expect a polynomial (as opposed to sub-exponential) bound to hold in Theorem~\ref{thmIsingMixing}. The reason our approach does not deliver such a polynomial bound is because the fluctuations due to the randomness of the graph  $G\sim \bg_d(n)$ of the ratios $Z_\eta^{\mathsf{fix}}/Z_{\eta + \frac{1}{n}}^{\mathsf{fix}}$ from an explicit drift function $F(\eta)$, which we show has only one positive fixed point, are only shown to be $o(1)$, non-quantitatively. Improvements in this error estimation would translate to improvements in the mixing time bound of Theorem~\ref{thmIsingMixing}.

\subsection{Proof outline}
\label{secOutline}

Here we outline the proof ideas for the three main results: the free energy density of the fixed-magnetization Ising model,  local weak convergence of the measure, and the analysis of Ising Glauber dynamics from a uniformly random start. 

\subsubsection{Free energy density: background}

Our aim is to show that for every $\eta \in [-1,1]$, with high probability $\frac{1}{n}  \log \Zfix_{\mathbf G_d(n),\beta} (\eta) $ is close to $f_{d,\beta}(\eta) = \lim_{n \to \infty} \frac{1}{n}  \log \E \Zfix_{\mathbf G_d(n),\beta} (\eta)$, the annealed free energy density. The annealed free energy density measures the contribution (on an exponential scale) to the first moment of the Ising model partition function $\Zising_{\mathbf G_d(n),\beta}$ from configurations of magnetization $\eta$. This function is plotted in Figure~\ref{fig:FdbZero} for $d=10$ and $\beta=.32$, between the uniqueness threshold $\beta_c$ and the reconstruction threshold $\beta_r$.

Let us begin by reviewing for which values of $\eta$ such a result follows from the previous literature. 
The results on the full Ising model on random graphs from~\cite{GM07} (see also~\cite{dembo2010ising}) showed that its  free energy density is given by the maximum over $\eta$ of this function:
\begin{align*}
  \lim \frac{1}{n} \E \log \Zising_{\bg_d(n),\beta}  = \lim \frac{1}{n} \log \E \Zising_{\bg_d(n),\beta} = f_{d,\beta}(\eta_*)\qquad\text{where} \qquad \eta_* = \arg\max_{\!\!\!\!\!\!\!\!\!\!\!\eta'}f_{d,\beta}(\eta')\,.
\end{align*}
By using Markov's inequality for other values of $\eta$, this establishes the fixed-magnetization free energy density at the maximizer $\eta_*$.  

By adding an external field $h$, the annealed free energy density changes by an addition of $\eta h$, resulting in a family of functions, plotted in Figure~\ref{fig:Fdbh} for the same $d,\beta$ and some non-zero values of $h$.  For each value of $h$, the methods of~\cite{GM07,dembo2010ising} apply, showing the Ising model free energy density is given by the maximum of the annealed function over all choices of $\eta$, and again these results can be transferred back to the fixed-magnetization model  at $\eta$ equal to the maximizer $\eta_*(h)$.  As the absolute value of the maximizing $\eta$ increases from its value at $h=0$, as $h$ varies, this covers all values of $\eta$ with $|\eta| \ge |\eta_\ast(0)|$.  

\begin{figure}
    \centering
    \includegraphics[width=0.6\linewidth]{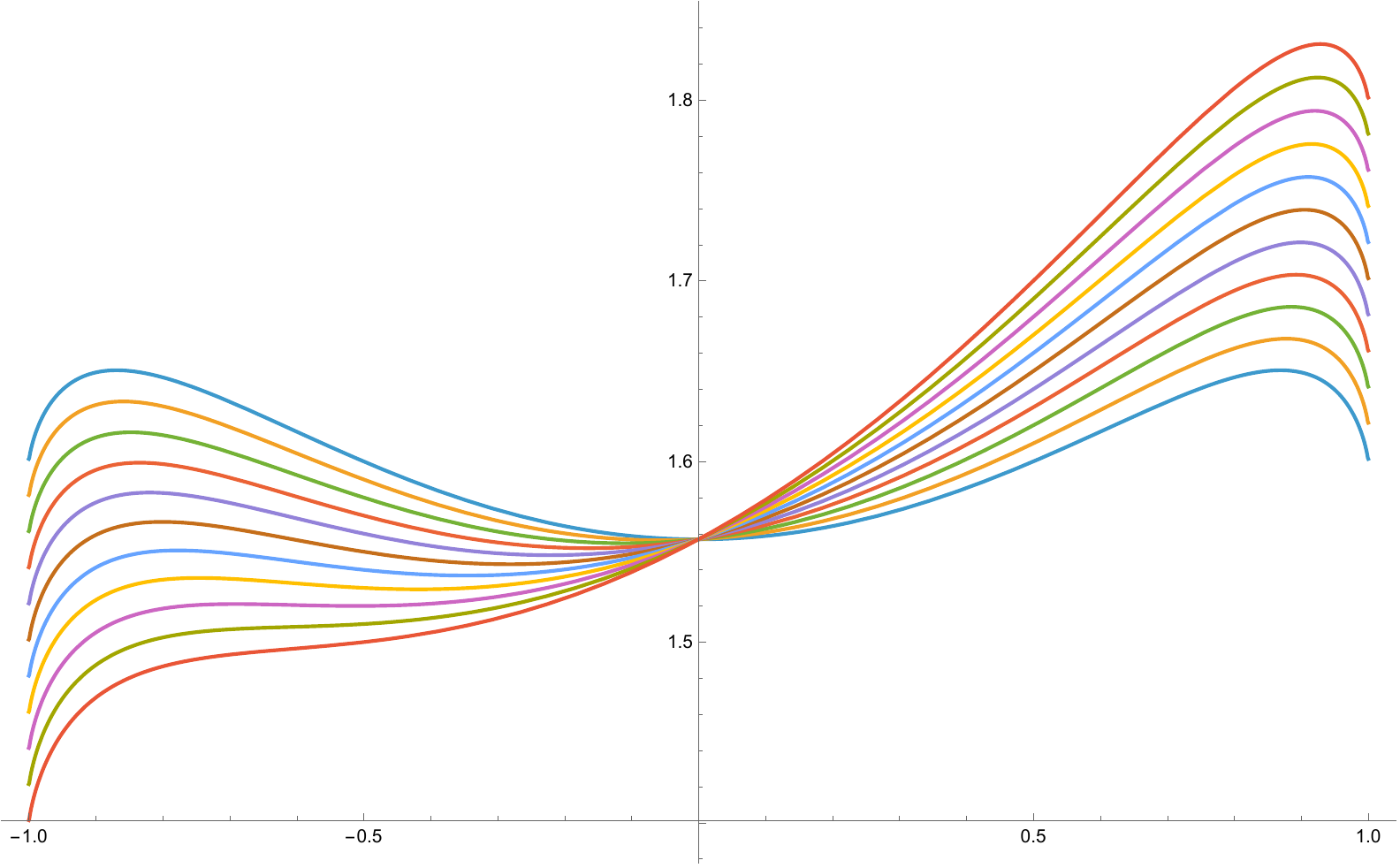}
    \caption{For $d=10$, $\beta=.32$, $f_{d,\beta}(\eta) + \eta h$ is plotted for different values of $\eta$ and $h$. This is the contribution to the annealed free energy of the Ising model with external field from different values of $\eta$. $\eta$ ranges from $-1$ to $1$ along the x-axis, while the different curves are for values of $h$ ranging from $0$ (blue) to  $.2$ (red).  Each value of $\eta$ appears as a stationary point for a unique value of $h$.}
    \label{fig:Fdbh}
\end{figure}

One can see the connection to translation-invariant measures on the tree in this discussion.  The stationary points of the annealed free energy density $f_{d,\beta}(\eta)$ correspond to fixed points of the Belief Propagation (BP) equations, which in turn correspond to translation-invariant tree measures (see Section~\ref{secTrees}).  The maximizer $\eta_*$ corresponds to the BP fixed point that maximizes the ``Bethe free energy,'' which in this case is equal to the annealed free energy density.  Following ~\cite{GM07,dembo2010ising,dembo2013factor,dembo2014replica} (and following the technical innovation in~\cite{galanis2015inapproximability}), Galanis, \u{S}tefankovi\u{c}, Vigoda, and Yang~\cite{GSVY16} established that all ferromagnetic models on random graphs follow the ``Bethe prediction'': the free energy density is given by the maximum of the Bethe free energy over all BP fixed points.

For $|\eta|< |\eta_*(0)|$, the fixed-magnetization free energy cannot be realized by a maximizer of the Bethe free energy with any external field.  However,~\cite{Coja-Oghlan22} showed that the ``metastability'' phenomenon in ferromagnetic models on random graphs can be investigated through a specialized second-moment method and the connection to \textit{local} maxima of the Bethe free energy. For those values of $\eta$ that are local maxima for some value of $h\in \mathbb R$, this shows (see~\cite{kuchukova2024fast} following the methods of~\cite{Coja-Oghlan22})  that the fixed-magnetization free energy density is given by the annealed free energy density. 
That set of $\eta$ is $|\eta|> \eta_s$ where $\pm \eta_s$ are the so-called spinodal points where the second derivative of $f_{d,\beta}(\eta)$ changes sign (see Figure~\ref{fig:FdbZero}). 

\subsubsection{Free energy density for unstable fixed points}
The remaining set of values of $\eta$, namely $\eta \in (-\eta_s,\eta_s)$ are neither global nor local maximizers of the Bethe free energy for any external field $h$, and so  never occur as typical or metastable states for the Ising model. 
The main contribution of this paper is dealing with these values of the magnetization. In particular, each of these values of $\eta$ is a critical point of the Bethe free energy for some value of $h$ (and thus corresponds to a BP fixed point and a translation-invariant Ising measure on the tree), but it is a local minimum instead of a local maximum; or, in other words, it corresponds to an unstable fixed point. Previous arguments that needed to allow for small fluctuations in the global number of plus spins will not work in this case, since given any freedom, the instability of the measure pushes the magnetization away from the desired $\eta$. 

At a high level, our proof still follows the framework of~\cite{Coja-Oghlan22}, which itself has antecedents in~\cite{achlioptas2008algorithmic,coja2018information,coja2018charting,coja2022ising}.  The main idea is to define a ``planted model'': first choose a configuration  $\sigma \in \Omega_{\eta}$ uniformly at random, then choose a $d$-regular graph $G$ with probability proportional to  $e^{\beta H_G(\sigma)}$.  This distribution of graph/configuration pairs is identical to that of first sampling a graph with probability proportional to the fixed-magnetization partition function and then sampling a configuration from the fixed-magnetization Ising model.  This planted model is in fact a regular, assortative stochastic block model, the one proposed in~\cite[Section 6]{mossel2015reconstruction} (and in fact one can deduce from our methods that below the reconstruction threshold the KL-divergence between this stochastic block model and the random regular graph is sublinear in $n$---see~\cite[Theorem 1.4]{coja2022ising} for a similar statement in the disassortative case).

We first show that in the planted model, we have local weak convergence to the tree measure $\nu_{\eta}$.  We then use the non-reconstruction property of this tree measure to show that with high probability over the planted model, two samples from the fixed-magnetization Ising model are uncorrelated, in the sense that their empirical measure converges to $\nu_\eta \otimes \nu_{\eta}$.  It is here that we must deal with the instability of this measure, as we will describe below. Once we have shown this, we return to the usual random graph and modify the partition function by multiplying it by an indicator random variable: $Y_{\bg_d(n),\beta}^{\mathsf{fix}}(\eta) = \Zfix_{\bg_d(n),\beta}(\eta) \ind{\mathcal E}$, where the event $\mathcal E$ captures the notion of two samples from the Gibbs measure being uncorrelated.  The fact that $\mathcal E$ holds with high probability in the planted model means we have not lost much in the first moment with the indicator: $\E Y_{\bg_d(n),\beta}^{\mathsf{fix}}(\eta) = (1+o(1)) \E \Zfix_{\bg_d(n),\beta}(\eta)$.  But we have now gained a lot in the second moment: immediately from the definition of $\mathcal E$ we get that the main contribution (on an exponential scale) to the second moment $\E [(Y_{\bg_d(n),\beta}^{\mathsf{fix}}(\eta))^2]$ comes from uncorrelated configurations, and their contribution matches the square of the first moment.  This suffices to carry out the second-moment method to establish the free energy density.

We deal with the instability of $\nu_{\eta}$ by working directly with the fixed-magnetization model, placing a hard constraint that prevents the magnetization from moving at all.  The price to pay for this hard constraint is that we lose direct application of the non-reconstruction property of the limiting tree measure.  Where this appears in the argument is in the key step of proving that the planted model exhibits overlap concentration (decorrelation of two samples).  Whereas in the previous approaches non-reconstruction on the tree is enough to deduce this, we need a stronger form of non-reconstruction that also involves conditioning on the magnetization of an Ising model on a tree with fixed, typical, boundary conditions. This argument (contained in Section~\ref{secDecayBoundary}) is delicate and requires several new ideas, along with inputs from a spectral gap bound of~\cite{martinelli2004glauber} on trees with boundary conditions, and local central limit theorems.

\subsubsection{Local weak convergence}

The above argument relating the planted model to the fixed-magnetization Ising model allows us to transfer  events that fail with exponentially small probability from the former to the latter.  In particular, if for some fixed $\eps>0$ the event $\mathcal E$ holds with probability at least $1- e^{- \eps n}$ in the planted model, then $\mathcal E$ holds with probability $1-o(1)$ in the non-planted model.  The argument above also tells us that the local distribution of a vertex in the planted model converges to that of the correct Gibbs measure on the infinite tree  (the measure $\nu_{\eta}$ guaranteed by Lemma~\ref{lem:existenceTreeMeasure}).  What remains is to show that the empirical distribution of local neighborhoods (over the population of all vertices of the graph) has exponentially small probability of deviating significantly from this tree measure in the planted model.  For Erd\H{o}s--R\'{e}nyi type planted models where edges are independent, this kind of a result is a straightforward application of concentration of Lipschitz functions for martingales.  For random regular graphs (either generated uniformly or from the configuration model), a combinatorial argument allows one to apply martingale concentration~\cite{wormald1999models}. For our planted model (which is an exponentially weighted configuration model, or equivalently a regular assortative stochastic block model), neither of these techniques apply directly. We prove a new result establishing concentration of Lipschitz functions of switchings in such weighted models that may be of independent interest.  See Lemma~\ref{planted-concentration} and Section~\ref{secConcentration} for a self-contained proof.   

This concentration lemma can be applied directly to other settings  in which a regular planted model is used to facilitate a second moment argument to establish a free energy density (e.g.~\cite{coja2022ising,Coja-Oghlan22}).  This adds a general method of proving local weak convergence to the techniques used in, e.g.,~\cite{bhatnagar2016decay,helmuth2023finite,shriver2023typical,basak2025potts,du2025characterizing,sly2018phase}.

\subsubsection{Analysis of Ising Glauber dynamics}
By spin-flip symmetry of the Glauber dynamics, for Theorem~\ref{thmIsingMixing} it will essentially suffice to show the desired mixing time upper bound for the Glauber dynamics restricted to $\Omega^+ = \{\sigma: \sum_v \sigma_v \ge 0\}$. 

There are two main ingredients to establishing this: the fast mixing result for Kawasaki dynamics from~\cite{bauerschmidt2023kawasaki} and our free energy density and local weak convergence results from Section~\ref{subsec:freeenergyIntro}. Our approach to obtaining mixing time bounds for Glauber dynamics in $\Omega^+$ is the projection-restriction method of Madras and Randall~\cite{madras2002markov}. To apply this method, given the fast mixing of the Kawasaki dynamics at every magnetization, we need to upper-bound the mixing time of a one-dimensional \emph{projection chain} $P_H$ moving on counts $k\in \{n/2,...,n\}$ of plus spins (equivalent to magnetizations $\eta \in [0,1]$ via $k = \frac{1+\eta}{2}n$) with weights at each $k$ given by the corresponding fixed-magnetization partition function: 
\begin{align*}
    P_{H}(k,k+1) = \frac{1}{2}\frac{Z_{k+1}}{Z_{k}+ Z_{k+1}} \qquad \text{where} \qquad Z_k = \Zfix_{G,\frac{2k}{n} -1}\ .
\end{align*}
Note that even though both the numerator and denominator are exponential in $n$, the transition probability is a sigmoid function applied to the ratio $Z_{k}/Z_{k+1}$ which is an order-one quantity. In particular, points where the ratio is $1$ are points where the projection chain has zero drift.

If we formally treat $\log (Z_k/Z_{k+1})$ as $f'_{d,\beta}(\eta)
$, the shape of $(f_{d,\beta}(\eta))_{\eta \in [0,1]}$ from Figure~\ref{fig:FdbZero} suggests the dynamics has an unstable fixed point at $\eta=0$ from which it diffuses away, and a stable fixed point at $\eta_*$ at which point mixing has occurred. By birth-death chain analysis, this could be shown to take $O(n\log n)$ steps, giving fast mixing of the overall Glauber dynamics on $\Omega^+$.  

However, the treatment above is not valid as it disregards the multiplicative $e^{o(n)}$ corrections to $e^{f(\eta)}$ in our identification of $Z_k$ from Theorem~\ref{thmFreeEnergy}. Indeed, such correction terms are real and come from the specific realization of $G$.  On the other hand, $\E[Z_{k}/Z_{k+1}]$ is a possibly different function which captures the correlations of $Z_k$ and $Z_{k+1}$ through the underlying graph, but would still be expected to have the same properties as $e^{f'(\eta)}$ to ensure fast mixing of the projection chain. 

Our approach is to rewrite the ratio of partition functions as the Gibbs expectation
\begin{align*}
    \frac{Z_{k+1}}{Z_{k}} = \frac{1}{k+1} \mathbb E_{\mufix_{G,\beta,\frac{2k}{n}-1}}\Big[\sum_v \ind{\sigma_v =-1} e^{ \beta \sum_{w\sim v} \sigma_w}\Big ]\,.
\end{align*}
The local weak convergence of $\mufix_{G,\beta,\eta}$ can then be used to say that with probability $1-o(1)$, $G\sim \bg_d(n)$ is such that this empirical average is close to a tree expectation of the local neighborhood of the root. Under the tree measure, the latter can be computed exactly as~\eqref{eq:F(eta)} and shown to have the same unimodal behavior on $\eta \in [0,1]$ that one expects. 

However, the fact that we go through local weak convergence, or alternatively exponential equivalence to the planted model, for this to be a calculable function means we incur errors that are $o(1)$ but not quantitative. These $o(1)$ errors could in principle conspire in the regions where $P_H$ has little drift (near the saddle at $0$ magnetization and the optimum at $\eta_*$ magnetization) to induce $\exp(o(n))$ slowdown in the mixing time. We note that any quantitative bound of $r_n=o(1)$ on the maximum (over $k$) deviation of the above expectation from its local weak tree limit would result in an $\exp(O(r_n^2 n))\cdot \text{poly}(n)$ upper bound on the mixing time from spin-flip symmetric initializations.

\subsection*{Organization}

In Section~\ref{secTrees} we recall background on translation-invariant Ising measures on $\mathbb T_d$ and connect these measures to the annealed free energy density $f_{d,\beta}(\eta)$ defined in~\eqref{eqn:firstmomentDef}. 
In Section~\ref{secFreeEnergy} we define the planted model and use it to carry out the proofs of Theorems~\ref{thmFreeEnergy} and Theorem~\ref{thmLocalWeak}.  In Section~\ref{secDynamics}, we use the main results on the energy landscape and local weak limit to prove Theorem~\ref{thmIsingMixing}. Finally in Section~\ref{secConcentration} we prove the result on concentration of Lipschitz functions in weighted configuration models that we use in proving Theorem~\ref{thmLocalWeak}.

\section{Ising model on trees and first-moment calculations}
\label{secTrees}

In this section we review background on Ising models on infinite trees, Belief Propagation fixed points, the reconstruction problem, and the connections of these tree measures to the first-moment $\E \Zfix_{\mathbf G_d(n),\beta}(\eta)$.  Much of this material appears in the textbook~\cite{georgii2011gibbs} and the general treatment of ferromagnetic spin models on random regular graphs in~\cite{GSVY16}.  The main goal of this section is to explain a three-way correspondence between translation-invariant Ising measures on $\mathbb T_d$, broadcast processes on $\mathbb T_d$, and critical points of the optimization problem determining the annealed free energy density.

We begin with some notational disclaimers. All logarithms are base $e$. All our results are for $d\ge 3$ and $\beta<\beta_r(d)$ fixed, while $n$ is sufficiently large. Unless otherwise specified by a subscript, all little-o and big-O notations are with respect to $n$. Because $d,\beta$ are fixed, we will often drop them from subscripts; we also may drop the graph $G\sim \bg_d(n)$ when understood from context.  Since the Ising measures  that we will predominantly work with are the fixed-magnetization ones, $\mu_{G,\beta,\eta}$ and $Z_{G,\beta}(\eta)$ will henceforth denote $\mufix_{G,\beta,\eta}$ and $\Zfix_{G,\beta}(\eta)$ respectively; the full Ising measure and partition function will retain their $\mathsf{Ising}$ superscripts.  The expectations and probabilities $\E,\mathbb{P}$ will be with respect to $G\sim \bg_d(n)$ except when indicated otherwise by a subscript, e.g., $\E_{\mu}$.

\subsection{Ising model on trees and BP fixed points}
\label{subsecIsingTrees}

We consider Ising model Gibbs measures on the infinite $d$-regular tree, denoted $\T_d$. It has long been known that the Ising model on $\T_d$ undergoes a phase transition between uniqueness and non-uniqueness of Gibbs measure.  That is, there exist $\beta_c$ and $h_c(d,\beta)$ such that there is a unique Gibbs measure on $\T_d$ if and only if (1) $\beta \leq \beta_c$ and $h$ is arbitrary, or (2) $\beta > \beta_c$ and $|h| > h_c(d,\beta)$. When $\beta, h$ satisfy either conditions (1) or (2), we say the model is in the {\em uniqueness regime} and otherwise is in the {\em non-uniqueness regime}.  

One can compute $\beta_c$ and $h_c(d,\beta)$ by restricting attention to translation-invariant Ising measures on $\T_d$.  Translation-invariant Gibbs measures are in one-to-one correspondence with solutions of the following fixed-point equation (a translation-invariant version of the more general BP equations):
\begin{equation}\label{eqn:tree-recursion}
R = \frac{e^{2h}  \left(\frac{R}{1+R}e^\beta + \frac{1}{1+R} \right)^{d-1}}{ \left (\frac{R}{1+R} + \frac{1}{1+R}e^\beta \right)^{d-1}}\ .
\end{equation}
Here the variable $R$ represents the ratio of the plus and minus spin probabilities of a single vertex on the infinite $(d-1)$-ary tree.  
In the Gibbs uniqueness regime there is only one fixed point, while in the non-uniqueness regime there are three (see, for example, \cite[Lemma 12.27]{georgii2011gibbs}).

For a given $d, \beta, h$, each fixed point of~\eqref{eqn:tree-recursion} can be used to define an infinite-volume Gibbs measure for the Ising specification defined by $\beta,h$ on $\T_d$. This is done via a {\em broadcast process}. Consider generating a spin configuration on a tree by first specifying the spin $\tau_{v_0}$ of the root $v_0$ according to some initial distribution, and then propagating spins down the levels of the tree with probabilities specified by a $2 \times 2$ transition matrix which we call the {\em broadcast matrix}.

The following lemma makes this connection explicit.  
\begin{lemma}[\hspace{1sp}{\cite[Proposition 12.24]{georgii2011gibbs}}]\label{lem:georgii-tree}
    For all $d \geq 1, \beta \geq 0, h \in \R $, there is a 1-to-1 correspondence $R \mapsto \mu_R$ between fixed points $R$ of \eqref{eqn:tree-recursion} and translation-invariant Ising measures $\mu_R$ (with specification given by $\beta,h$) on $\T_d$ which arise from the broadcast process with broadcast matrix
    \begin{equation}
    \label{eqBmatrix}
    M_R := \begin{pmatrix}
        \frac{e^{\beta}R}{e^{\beta}R+1} & \frac{R}{e^{\beta} + R} \\
        \frac{1}{e^{\beta}R + 1} & \frac{e^{\beta}}{e^{\beta} + R}
    \end{pmatrix}\end{equation}
and initial distribution for $\tau_{v_0}$  specified by
\begin{equation}\label{eqn:tree-marginal}
\E_{\mu_R}[\tau_{v_0}] = \frac{R^2 - 1}{R^2 + 2e^{-\beta}R + 1} \,.
\end{equation}
\end{lemma}
In words, the broadcast process defined by~\eqref{eqBmatrix} and~\eqref{eqn:tree-marginal} defines a translation-invariant Ising measure with parameters $\beta,r$ on $\T_d$ if and only if $R$ is a fixed point of~\eqref{eqn:tree-recursion}.

We can now prove Lemma~\ref{lem:existenceTreeMeasure}, which states that for each $d \ge 3$, $\beta \ge 0$, and $\eta \in (-1,1)$, there is a unique translation invariant Ising Gibbs measure on $\T_d$ with magnetization $\eta$.
\begin{proof}[Proof of Lemma~\ref{lem:existenceTreeMeasure}]
    The magnetization of the translation-invariant Ising measure $\mu_R$ is given by \eqref{eqn:tree-marginal}. This function is continuous with respect to $R$, strictly increasing, and approaches $-1$ as $R \to 0$ and $1$ as $R \to \infty$. Thus, given $\eta$ and $\beta$, we may find $R^*$ such that $\mathbb E_{\mu_{R^*}}[\tau_{v}] = \eta$.  Then given $R^*$, we claim there exists a unique choice of $h$ such that $R^*$ is a fixed point of \eqref{eqn:tree-recursion}. Indeed, we can solve \eqref{eqn:tree-recursion} for $h$. Note that while each choice of $R$ gives distinct $h$, the same $h$ corresponds to three distinct values of $R$ when $h$ is in the non-uniqueness regime.
\end{proof}
For fixed $\eta$, we henceforth define $\nu_\eta$ to be the Ising  measure guaranteed by Lemma~\ref{lem:existenceTreeMeasure}.  Each such measure comes with two statistics: the vertex marginal, i.e. the magnetization $\eta$, and the edge marginal, which we will parameterize by the single statistic $\rho_\eta$ representing the probability a given edge is monochromatic under $\nu_\eta$. Note that for a translation-invariant measure these statistics are constant over vertices and edges.  

We can take the defining equations~\eqref{eqBmatrix} and~\eqref{eqn:tree-marginal} and eliminate $R$, giving $\rho_\eta$ as a function of $\eta, \beta$ (it turns out not to depend on $d$).  

\begin{lemma}
    For each $\eta \in (-1,1)$, under $\nu_\eta$ the probability an edge is monochromatic is 
   \begin{align}\label{eq:rhoeta-explicit} 
    \rho_\eta = \frac{e^{2\beta} - \sqrt{e^{2\beta}(1 - \eta^2) + \eta^2}}{e^{2\beta} - 1}
\end{align}
\end{lemma}
\begin{proof}
    Expressing $R$ as a function of $\eta$ via~\eqref{eqn:tree-marginal} (and noting $R$ must be positive), we have
    \begin{equation*}
        R_\eta = \frac{-\sqrt{-e^{2 \beta} \eta^2+e^{2 \beta}+\eta^2}-\eta}{e^{\beta} \eta-e^{\beta}} \,.
    \end{equation*}
    The monochromatic edge probability is then the inner product of $((1+\eta)/2,(1-\eta)/2)$ with the diagonal entries of $M_R$.  This gives
    \begin{align*}
        \rho_{\eta}&= \frac{1+\eta}{2}  \frac{e^{\beta}R_\eta}{e^{\beta}R_\eta+1}  + \frac{1-\eta}{2} \frac{e^{\beta}}{e^{\beta} + R_\eta} \\
        &=\frac{e^{2\beta} - \sqrt{e^{2\beta}(1 - \eta^2) + \eta^2}}{e^{2\beta} - 1}
    \end{align*}
    after some algebraic simplification.
\end{proof}

\subsection{Reconstruction threshold}
The reconstruction problem for a broadcast process asks whether the distribution of spins at  level  $\ell$ of the tree contains any information about the distribution at the root in the $\ell \to \infty$ limit. More formally, let $T_\ell(d) = T_\ell$ be the truncation of the $d$-regular tree at depth $\ell$. Suppose $M$ is a broadcast matrix and let $P_{\ell}^i$ denote the distribution on $T_\ell$ induced by $M$, conditioned on the spin at the root $\sigma_{0} = i$.

We say the reconstruction problem for $\T_d$ and $M$ is {\em solvable} if $\lim_{n \to \infty} d_{\tv}(P_n^i, P_n^j) > 0$ for some $i, j$. Otherwise, the reconstruction problem is {\em unsolvable} and we say $M$ (or equivalently the measure on $\T_d$ arising from $M$) has {\em non-reconstruction}.

The proofs of our main results require restricting $\beta$ to the non-reconstruction regime for $\nu_{\eta}$, meaning non-reconstruction for the broadcast matrix $M$ corresponding to $\nu_{\eta}$ as described in the previous subsection.

In the zero-field case, reconstruction is characterized by the second-largest eigenvalue $\lam_2(M)$ of $M$. First \cite{kesten1966additional} and more explicitly \cite{higuchi1977remarks}
showed that $(d-1)|\lam_2(M)|^2 > 1$ is a necessary condition for reconstruction; later, \cite{bleher1995purity} proved this is also a sufficient condition, giving rise to the {\em reconstruction threshold} $\beta_r$ from~\eqref{eq:betardef}.

In the case of a non-zero external field, however, the precise value of the reconstruction threshold is only known in some cases. Indeed, it is known that there are choices of external field $h$ such that the corresponding Kesten-Steegum threshold (given by the eigenvalue condition) is provably not the reconstruction threshold for the Gibbs measure parametrized by $h$ \cite{mossel2001reconstruction}. 
Nevertheless, \cite{martinelli2004glauber, martinelli2007fast} show that for $\beta < \beta_r$ and all $h$, any Gibbs measure with respect to $\beta, h$ and arbitrary boundary conditions satisfies a ``variance-mixing'' condition which implies exponential decay of correlations and non-reconstruction.

\begin{lemma}[\hspace{1sp}\cite{martinelli2004glauber}]\label{lem:nonreconbd}
     For all $\beta < \beta_r$ and $\eta \in (-1,1)$, the reconstruction problem for the measure $\nu_{\eta}$ is unsolvable.
\end{lemma}

The following collects several equivalent concrete statements one deduces immediately from the non-reconstruction up to $\beta_r$. 

\begin{corollary}\label{cor:various-forms-of-nonreconstruction}
    For any $d\ge 3$, $\eta \in (-1,1)$ and $\beta<\beta_r$ one has the following equivalent properties. Let $\tau_r$ be the configuration on the vertices of $\mathbb T_d$ at depth $r$, so that $\tau_0$ is the spin at the root.    
    \begin{enumerate}[(a)]
    \item $\lim_{r\to\infty} d_{\tv}(\nu_\eta(\tau_r \mid \tau_0=+), \nu_\eta(\tau_r \mid \tau_0 = -)) =0$\,,
    \item  $\lim_{r\to\infty} \mathbb E_{\tau_r \sim \nu_\eta}\big[|\nu_\eta (\tau_0 \mid \tau_r)- \nu_\eta(\tau_0)|\big] =0$\,,
    \item There exists a set of configurations $A_r$ on the vertices of depth $r$ such that $\nu_\eta(\tau_r\in A_r) = 1-o_r(1)$ and for all $\tau_r \in A_r$ one has $|\nu_\eta (\tau_0 \mid \tau_r)- \nu_\eta(\tau_0)|=o_r(1)$. 
    \item There exists a set $A_r$ such that $\nu_\eta(\tau_r\in A_r) = 1-o_r(1)$ and for all $\tau_r \in A_r$, one has $$\nu_\eta(\tau_r \mid \tau_0 )= (1+o_r(1)) \nu_\eta(\tau_r)\,.$$
    \end{enumerate}
\end{corollary}

\begin{proof}
    The first two items are the standard equivalent definitions of non-reconstruction. The third item follows from Markov's inequality applied to the second item. 
    
    For the fourth item, suppose by way of contradiction that there were a set $A_r^c$ of probability $\eps$ having $\nu_\eta(\tau_r \mid \tau_0=+) >(1+\delta) \nu_\eta(\tau_r)$ for all $\tau_r \in A_r^c$. Then $\nu_\eta(A_r^c \mid \tau_0 =+) >(1+\delta) \nu_\eta(A_r^c)$ and $$\nu_\eta(A_r^c \mid \tau_0 =+) - \nu_\eta(A_r^c) > \delta \eps\,,$$ contradicting the first item. 
\end{proof}

\subsection{Annealed free energy density}
\label{secfirstmomentcalcs}

In this section we give the formula for the annealed free energy density of the fixed-magnetization Ising model.  The formula is a maximization of a one variable function $\rho$, and the maximizer coincides exactly with the edge statistics $\rho_{\eta}$ of the tree measure $\nu_{\eta}$ defined above.  The connection between Gibbs measures on trees, BP fixed points, and  optimization problems determining annealed free energy densities appears in~\cite{galanis2015inapproximability,GSVY16}; what is nice in our setting is that everything can be made explicit.  

Instead of working directly with a uniformly random $d$-regular graph, we work with the {\em configuration model}.  In this model, each of the $n$ vertices is represented by $d$ clones.  Let $\mathcal P_{d,n}$ denote the set of all $(dn-1)!!$  possible perfect matchings of these clones.   The configuration model is a uniformly chosen $G \in \mathcal P_{d,n}$, which defines a (possibly looped) multigraph (letting $G$ denote both the matching and the resulting graph).   For constant $d$, it is well known that events that hold  with probability $1-o(1)$ for the configuration model also hold with probability $1-o(1)$ for the uniform random $d$-regular graph~\cite{wormald1999models}. It therefore suffices to prove all our main theorems for this configuration model of random $d$-regular graphs. 
In a slight abuse of notation, due to the contiguity of the models, we will use $\bg = \bg_d(n)$ as notation for the configuration model from here on, and $\E,\Prob$ are with respect to this random graph distribution.  We also write $Z_G(\eta)$ for $Z_{\bg_d(n),\beta}(\eta)$ in this section.

We  give an explicit formula for the annealed free energy density.  The formula, given by a function $f_{d,\beta}(\eta)$, will be the maximum of another function $g_{d,\beta,\eta}(\rho)$ over $\rho \in [0,1]$.  We first define this function $g_{d,\beta,\eta}$,  show that the maximizer is exactly $\rho_{\eta}$ defined above in~\eqref{eq:rhoeta-explicit}, and use this to define $f_{d,\beta}(\eta)$ explicitly.  We then go through the combinatorial calculations to show that  $f_{d,\beta}(\eta)$ is indeed the annealed free energy density.  The main idea behind these calculations is that the number of monochromatic edges determines the contribution of a configuration to the partition function.  It is also straightforward to write a formula for the number of outcomes of the configuration model with a given number of monochromatic edges.  Since we are only interested in the exponential growth rate of the first moment, it is enough to maximize the contribution to the first moment over all possible numbers of monochromatic edges, and this is done by maximizing $g_{d,\beta,\eta}(\rho)$.

\begin{lemma}\label{lem:edge-statistics-max}
Given $d \ge 3$, $\beta \ge 0$, and $\eta \in (-1,1)$, define the following function of $\rho \in [0,1]$:
\begin{align}
    g_{d,\beta,\eta}(\rho) &= \frac{\beta \rho d}{2} + \log 2 + \frac{d-1}{2}\left((1+\eta) \log (1+\eta) + (1-\eta) \log (1-\eta) \right)   \label{eq:g-free-energy-function} \\
    & \quad  - \frac{d}{2}\Big((1-\rho)\log(1-\rho) +\frac12(\rho+\eta)\log(\rho+\eta) + \frac12(\rho-\eta)\log(\rho-\eta) + \log 2\Big)\,. \nonumber
\end{align}
The function $g_{d,\beta,\eta}$ has a unique maximizer at $\rho = \rho_{\eta}$, and the maximum is given by
\begin{align}
    f_{d,\beta}(\eta) &= \frac{\beta \rho_\eta d}{2} + \log 2 + \frac{d-1}{2}\left((1+\eta) \log (1+\eta) + (1-\eta) \log (1-\eta) \right)  \label{eq:f-closed-form} \\
    & \quad - \frac{d}{2}\Big((1-\rho_\eta)\log(1-\rho_\eta)   +\frac12(\rho_\eta+\eta)\log(\rho_\eta+\eta) + \frac12(\rho_\eta-\eta)\log(\rho_\eta-\eta) + \log 2\Big)\,. \nonumber
 \end{align}
\end{lemma}

\begin{proof} 
Differentiating $g_{d,\beta,\eta}(\rho)$, we have
$$g_{d,\beta,\eta}'(\rho) = \frac{\beta d}{2} - \frac{d}{2}\left( -\log(1-\rho) + \frac{1}{2}\log(\rho+\eta) + \frac{1}{2}\log(\rho-\eta) \right)\,,$$
and
$$g_{d,\beta,\eta}''(\rho) = -\frac{d}{2(1-\rho)} - \frac{d}{4(\rho+\eta)} - \frac{d}{4(\rho-\eta)}\ .$$
This shows $g''(\rho) < 0$ so $g(\rho)$ must have a unique maximum, and solving $g'_{d,\beta,\eta}(\rho) =0$ we see that the maximizer occurs at $\rho_{\eta}$. Plugging $\rho_{\eta}$ into $g_{d,\beta,\eta}(\rho)$ gives the formula for $f_{d,\beta}(\eta)$.
\end{proof}

We now show that  $f_{d,\beta}(\eta)$ is indeed the annealed free energy density.
\begin{lemma}\label{lem:partition-first-moment}
For every $d\ge 3$, $\beta\ge0$ and $\eta \in (-1,1)$,
    $$\E[Z_{G}(\eta)] = \exp\big(n f_{d,\beta}(\eta) + O(\log n)\big)\,.$$
\end{lemma}
\begin{proof}
In what follows, we omit floors and ceilings, (e.g., in $\frac{1+\eta}{2}n$); the effect of these omissions will always be $O(1/n)$ additive errors which are negligible. 
Suppose we fix a spin configuration $\sigma \in \Omega$ on $V$ where $n_+ = |\sigma^{-1}(+1)|$ and $n_- = |\sigma^{-1}(-1)|$. The number of ways to specify a perfect matching on $V$ with exactly $k$ bichromatic edges is given by 
\begin{equation}\label{eqn:bi-counting}
b(k) = {n_+ \choose k}{n_- \choose k}k!(n_+ - k - 1)!!(n_- - k - 1)!!\ .
\end{equation}
Indeed, there are ${n_+ \choose k}{n_- \choose k}$ ways to specify the endpoints of the bichromatic edges and $k!$ ways to pair up those endpoints. For the monochromatic edges, note that there are $(2m-1)!! = \frac{(2m)!}{2^m m!}$ perfect matchings between $2m$ vertices. In this way, we also see that the number of monochromatic $(+1, +1)$ edges and $(-1, -1)$ edges are determined by $k, n_+$, and $n_-$.

We may simplify \eqref{eqn:bi-counting} as
    $$b(k) = \frac{2^kn_+!n_-!}{k!2^{n_+/2}2^{n_-/2}(\frac{n_+-k}{2})!(\frac{n_--k}{2})!}\,.$$
Applying Stirling's formula in the form $\log n! = n \log n - n + \frac{1}{2} \log(2\pi n) + O(1/n)$, we get
    \begin{align}\label{eq:b(k)-approx}
    \log b(k)  &=  n_+ \log n_+ + n_- \log n_- - k \log k - \frac{n_+-k}{2}\log \frac{n_+-k}{2} - \frac{n_--k}{2}\log \frac{n_--k}{2}  \nonumber \\
    &\ \ \ \ + \frac12\left(\log \frac{2n_+n_-}{\pi k(n_+-k)(n_--k)}\right)  + \left(k - \frac{n}{2}\right)\log2 - \frac{n}{2} + O\left(\frac{1}{k}\right)\,.
    \end{align}

Now we write $\E[Z_{G}(\eta)]$ in terms of the contributions from graphs and configurations with a given number of bichromatic edges $B(G)$ in the following way. Fixing a magnetization $\eta \in (-1,1)$ and spin configuration $\sigma \in \Omega_\eta$, we have $n_+ = \frac{1+\eta}{2}dn$ and $n_- = \frac{1-\eta}{2}dn$. Parametrizing the number of monochromatic edges as $\rho\frac{d n}{2}$ and bichromatic edges as $(1-\rho)\frac{dn}{2}$, the probability that $G$ achieves these edge statistics is given by
\begin{align*}\Prob\left(B(G) = (1-\rho) \frac{dn}{2}\right) &= \frac{1}{(d n-1)!!} {\frac{1+\eta}{2} d n \choose (1-\rho)\frac{dn}{2}}{\frac{1-\eta}{2} d n \choose (1-\rho)\frac{dn}{2}}\left((1-\rho)\frac{dn}{2}\right)! \\
&\ \ \ \,\, \,\,\,\, \cdot  \left(\left((1+\eta) - (1-\rho)\right)\frac{d n}{4}-1\right)!!\left((1-\eta) - (1-\rho))\frac{d n}{4}-1\right)!!\ .
\end{align*}

Using the approximation for $b(k)$ from~\eqref{eq:b(k)-approx}, let
\begin{align*} \Psi(\eta, \rho) &= (1+\eta)\log (1+\eta) + (1-\eta)\log(1-\eta) - (1-\rho)\log(1-\rho)  \\
&\qquad  -\frac12(\rho + \eta)\log(\rho+\eta) - \frac12(\rho-\eta)\log(\rho-\eta) - \log 2\ .
\end{align*}
Then
\begin{align*}
    \log \Prob\left(B(G) = (1-\rho) \frac{dn}{2}\right) &= \frac{dn}{2}\Psi(\eta, \rho) - \frac12 \log dn  + O(1)\ .
\end{align*}
We can now write 
\begin{align}\label{eq:g-eta-rho-function}
g_{d, \beta,\eta}( \rho) = \mathsf{H}\left(\frac{1+\eta}{2}\right)\log 2 + \frac{d}{2}\Psi(\eta, \rho) + \frac{\beta \rho d}{2}\,,
\end{align}
using $\mathsf{H}$ for the binary entropy function, $\mathsf{H}(x) = -x\log_2 x - (1-x)\log_2(1-x)$.  After substituting in  the function $\Psi(\eta, \rho) $, this matches the formula for $g_{d, \beta,\eta}( \rho)$ given in Lemma~\ref{lem:edge-statistics-max}.

As a consequence, we get the following contribution to $\E Z_G(\eta) = \E \sum_{\rho} Z_{G}(\eta,\rho)$ from partition functions restricted to monochromatic edge fraction $\rho$, 
\begin{align}\label{eq:ZG-eta-rho}
    \E[Z_{G}(\eta,\rho)] & = {n \choose \frac{1+\eta}{2} n}  \Prob\Big(B(G) =\frac{(1-\rho) dn}{2}\Big) e^{\beta\rho dn/2} \nonumber  \\ 
    & = \exp\Big( n \cdot g_{d,\beta}(\eta,\rho) + O(\log n)\Big)
\end{align}
Since $ \E[Z_{G}(\eta)]$ is given by summing $ \E[Z_{G}(\eta,\rho)]$ over the $O(n)$ number of possible values of $\rho$, up to a polynomial factor the sum is given by the maximum. That is, $\E[Z_{G}(\eta)] = \exp\big(n f_{d,\beta}(\eta) + O(\log n)$.
\end{proof}

Given the formula~\eqref{eq:f-closed-form}, one can plot $f_{d,\beta}(\eta)$ as a function of $\eta$.  When $\beta < \beta_c$ (in the uniqueness regime), $f_{d,\beta}(\eta)$  has only one critical point, a global maximum at $\eta=0$.  When $\beta>\beta_c$, the function has three critical points (as seen in Figure~\ref{fig:FdbZero}): twin global maxima at $\pm \eta_*$ and a local minimum at $\eta=0$. These expected properties of $f$ can also be proven analytically with some calculus due to the explicit form of~\eqref{eq:f-closed-form}.

\section{Free energy of the fixed-magnetization Ising model}
\label{secFreeEnergy}

In this section we prove \Cref{thmFreeEnergy} by employing a  truncated second-moment argument facilitated by a comparison with a planted model. This section contains the bulk of the new technical work of the paper needed to handle fixed magnetization and convergence to an unstable BP fixed point.

\subsection{Truncated second moment}

At a high level, we use a technique from \cite{Coja-Oghlan22, coja2022ising}.  To apply the second-moment method beyond uniqueness, we find an event $\cE$ that gives control of the second moment, while intersecting with it does not change the first moment by much.  

We define $\cE$ in terms of the overlap statistics of two configurations. Given $\sig, \sig'$, we define a $4 \times 4$ matrix $R_{\sig,\sig'}$ where we label the rows with the edge types of $\sigma$ ($++, +-, -+, --$) and similarly the columns with edge types for $\sigma'$. Then the entry $R_{ij}$ is equal to the fraction of edges in $G$ which are type $i$ in $\sigma$ and type $j$ in $\sigma'$. 

Note that the entries of $R$ immediately determine the vertex overlap statistics. Indeed, the number of vertices that are $+$ in both $\sigma$ and $\sigma'$ can be computed as $\frac{n}{2}(2R_{++,++} + R_{+-,+-}+R_{++,+-}+R_{+-,++})$, and similarly for the number of vertices that are $-$ in both. 

Let $\rho_{\nu_{\eta}}^{\otimes 2}$ denote the $4 \times 4$ edge overlap matrix associated with two independent configurations sampled from the  tree measure $ \nu_{\eta}$. Let $\sig, \sig'$ be independent samples from $\mu = \mu_{G, \beta,\eta}$ for $G\sim \bg$. Then define the event
$$\cE := \{G : \E_{\mu_{G,\beta,\eta}^{\otimes 2}}[\|R_{\sig, \sig'} - \rho_{\nu_{\eta}}^{\otimes 2}\|_2] = o(1)\}\ .$$

\begin{proposition}\label{prop:truncated-moments}
    Fix $d\ge 3,\beta<\beta_r$ and $\eta \in (-1,1)$. Let $Z_G = Z_{G,\beta}(\eta)$ and let $Y_{G} = Z_{G} \ind{\cE}$. Then we have
    \begin{align*}\E[Y_{G}] & = (1+o(1)) \E[Z_{G}]\,, \quad  \text{ and } \\
    \E[Y_{G}^2] & \leq e^{o(n)}\E[Z_{G}]^2\ .\end{align*}
\end{proposition}

The bulk of our work in \Cref{secFreeEnergy} is in proving this proposition. We first show how the main theorem follows from this statement. To do this, we require a concentration result in the configuration model, for random variables  which have bounded differences under small perturbations of the random graph. To this end, define a {\em switching} of a graph $G \sim \bg_d(n)$ as a graph $G'$ obtained by replacing two edges $\{u,v\}, \{x,y\}$ in $G$ by the edges $\{u,y\},\{v,x\}$. Notice that $G'$ remains $d$-regular. 

\begin{lemma}[\hspace{1sp}{\cite{wormald1999models}}]
\label{lem:switching}
If $X_n$ is a random variable defined on $\bg_d(n)$ such that $|X_n(G) - X_n(G')| \leq c$ whenever $G$ and $G'$ differ by a single switch, then for all $t > 0$,
    $$\Prob(|X_n - \E[X_n]| \geq t) \leq 2\exp\Big(-\frac{t^2}{dnc^2}\Big)\ .$$
\end{lemma}

\begin{proof}[Proof of \Cref{thmFreeEnergy}]
    Fix $d \geq 3$, $\beta < \beta_r$, $\eta \in (-1,1)$, and let $Z = Z_{G}(\eta), Y = Z_G(\eta)\ind{\cE}$. First note that the random variable $\log Z$ is concentrated around its expectation. 
    Indeed, if $G, G' \sim \bg_d(n)$ differ by a single switch, then their partition functions differ by at most a factor of $e^{2\beta}$, so we may apply \Cref{lem:switching}.  Then, 
\begin{equation}\label{eqn:logZconcentration}
    \Prob(|\log Z - \E[\log Z]| \geq t) \leq 2\exp\left(-\frac{t^2}{4dn\beta^2}\right)\ .
    \end{equation}
The theorem will follow by applying~\eqref{eqn:logZconcentration} with $t = \varepsilon n$ for arbitrary fixed $\varepsilon$ and taking a union bound over the $n$ possible values of the fixed magnetization $\eta$, if we can establish 
\begin{equation}\label{eqn:expectation-log-Z}
\E[\log(Z)] = \log(\E[Z]) + o(n)\,.
\end{equation}
The upper bound follows from Jensen's Inequality. For the lower bound, it is enough (via~\eqref{eqn:logZconcentration} again)  to show that $\log Z =  \log(\E[Z]) - o(n)$ with subexponentially small probability. For $\theta \in (0,1)$, $\Prob(\log Z - \log \E[Z] \geq \log \theta) = \Prob(Z \geq \theta \E[Z])$. Trivially $Y \leq Z$, so by \Cref{prop:truncated-moments}, 
    $$\Prob(Z \geq \theta\E[Z]) \geq \Prob(Y \geq \theta(1+o(1))\E[Y])\ .$$
Applying the Paley-Zygmund inequality with $\theta = \frac12$, and using Proposition~\ref{prop:truncated-moments} again gives
\begin{equation}\label{eqn:PZbound}
\Prob(Z \geq \frac12\E[Z]) \geq \frac{1}{4}\frac{((1+o(1))\E[Y])^2}{\E[Y^2]}\geq \Omega(\exp[-o(n)])\ ,
\end{equation}
which establishes~\eqref{eqn:expectation-log-Z} and hence the theorem.
\end{proof}

To analyze the moments of the truncated partition function $Y_{G}$, we move to a planted model and prove that the event $\cE$ holds in this model with high probability. We then couple the planted model with a broadcasting process on $\T_d$ and use the non-reconstruction property, together with various novel arguments to handle the fixed magnetization aspect, to conclude.

\subsection{The planted model}
\label{secPlantedDef}

We now describe the \textit{planted model}, a distribution on pairs of graphs and configurations.  The planted model has been used (as we will use it) to compare to and analyze spin models on random graphs~\cite{achlioptas2008algorithmic,coja2018information,coja2022ising} as well as in its own right (as the stochastic block model) as a model for statistical inference~\cite{mossel2015reconstruction}.

Typically one defines the planted model by choosing a configuration at random and then choosing a graph consistent with this planted configuration.  In our setting, the planted model is the distribution $\hat \Prob$ over pairs $(\hat G,\hsig)$ given by:
\begin{enumerate}
    \item Sample $\hsig$ uniformly from $\Omega_\eta$.
    \item Given $\hsig$, sample $\hat G \in \mathcal P_{d,n}$ according to the distribution
    \begin{equation}\label{defn:plantedconfig-graph}
\hat\Prob(\hat G \mid  \hsig) = \frac{\Prob(G) e^{\beta H_G(\sig)}}{\E[e^{\beta H_{G}(\sig)}]} \,,
\end{equation}
\end{enumerate}
where $\hat \Prob, \hat \E $ denote probabilities and expectations with respect to the planted model while  $\Prob, \E$ denote probabilities and expectations with respect to the configuration model.
In words, the planted distribution reweights the configuration model by the weight the graph assigns $\hsig$ in the fixed-magnetization Ising model.

Alternatively, we can generate the same distribution $\hat \Prob$ by choosing the graph first and then the configuration, as follows:
\begin{enumerate}
    \item Sample $\hat G$ according to the distribution
    \begin{equation}\label{defn:plantedgraph}
\hat\Prob(\hat G =G) = \frac{Z_{G}(\eta)\Prob(G)}{\E[Z_{G}]}\,.
\end{equation}
\item Sample $\hsig$ from the fixed-magnetization Ising model on $\hat G$, i.e., from $\mu_{\hat G,\beta,\eta}$. 
\end{enumerate}
In words, in this description of the planted model we choose a graph with probability proportional to its fixed-magnetization Ising partition function and then sample a configuration from the Gibbs measure on the chosen graph.

The well-known Nishimori identity (see, e.g.,~\cite{coja2018information})  tell us that these two planting schemes give  the same distribution on graph/configuration pairs.

We now show that the edge statistics under the planted distribution at magnetization $\eta$ are concentrated around the ones that maximize $g_{\eta}$. Given $\eta \in (-1,1)$, let $\rho_{\eta}$ denote the  maximizer of $g_{\eta}(\rho)$~\eqref{eq:g-eta-rho-function} and let $\hat\rho = \hat\rho_{\hat G,\hsig}$ denote the random variable that is the fraction of monochromatic edges of $\hat G$ under $\hsig$ when $(\hat G, \hsig)\sim \hat \Prob$.

\begin{lemma}\label{lem:planted-edge-stats}
    Given $\eta \in (-1,1)$, there exist constants $c, t_0 > 0$ such that for all $t \in [0, t_0)$,
    $$\hat\Prob(|\hat\rho - \rho_\eta| > t) \leq \exp(-ct^2n + O(\log n))\ .$$
\end{lemma}

\begin{proof}    
    Using \eqref{defn:plantedgraph}, the joint distribution over graph-configuration pairs $(G, \sigma)$ is given by 
    $$\hat\Prob((G, \sigma)) = \frac{Z_G(\eta)\Prob(G)}{\E[Z_G(\eta)]} \cdot \frac{e^{\beta H_G(\sig)}}{Z_G(\eta)} = \frac{\Prob(G)e^{\beta H_G(\sig)}}{\E[Z_G(\eta)]}\ .$$
    For a given edge statistic $\rho$, we then have
    $$\hat\Prob(\hat{\rho} = \rho) = \sum_{(G, \sig) : \rho_{\sig,G} = \rho} \hat\Prob((G, \sig)) = \frac{1}{\E[Z_G(\eta)]} \sum_{(G, \sig) : \rho_{\sig,G} = \rho} \Prob(G)e^{\beta H_G(\sig)}\ .$$
    Observe that $\sum_{(G, \sig) : \rho_{\sig,G} = \rho} \Prob(G)e^{\beta H_G(\sig)}$ is exactly the unplanted expectation (over $G$) of the partition function $Z_G(\eta)$ restricted to configurations with monochromatic edge fraction $\rho$---denote this by $Z_G(\eta, \rho)$. Then,
    $$\hat\Prob(|\hat\rho - \rho_{\eta}| > t) = \sum_{\rho} \ind{|\rho -\rho_{\eta}| > t} \frac{\E[Z_G(\eta, \rho)]}{\E[Z_G(\eta)]}\ .$$
    Each expectation in the numerator was shown to be $\exp( n g(\eta,\rho) + O(\log n))$ in~\eqref{eq:ZG-eta-rho} while the denominator was shown there to be $\exp( n g(\eta,\rho_\eta) + O(\log n))$. 
    By definition of $\rho_{\eta}$ and the second derivative calculation of Lemma~\ref{lem:edge-statistics-max}, there exists $c > 0$ such that for all $\rho'$, we have $g_{\eta}(\rho') \leq g_{\eta}(\rho_{\eta}) - c|\rho' - \rho_\eta|^2$. Finally, there are at most $n$ choices of $\rho$ to sum over. Thus,
    \begin{align*}\hat\Prob(|\hat\rho - \rho_{\eta}| > t) &\leq \exp(n(\max_{\rho': |\rho' - \rho_\eta|>t} g(\eta, \rho') - g(\eta, \rho_{\eta})) + O(\log n))\\
    & \leq \exp(-ct^2n + O(\log n))\,. \qedhere
    \end{align*}
\end{proof}

\subsection{Local approximation of the planted model by the tree measure}
Equipped with an understanding of the typical edge statistics of the planted model, we  now compare the  distribution the planted model induces on bounded-depth neighborhoods with that induced by the tree measure $\nu_{\eta}$. 

For $v \in [n]$, let $B_r(v)$ denote the vertices of distance at most $r$ from $v$ (i.e. the closed $r$-neighborhood of $v$). Let $\partial_r(v) = \partial B_r(v)$ denote the vertices of distance exactly $r$ from $v$. Given $x$ the root of $\T_d$, let $\tau_{B_r(x)}$ denote the distribution on spin configurations on $B_r(x)$ induced by $\nu = \nu_{\eta}$, and let $\tau_r = \tau_{\partial_r(x)}$. Similarly let $\hsig_{B_r(v)}$ denote the distribution on spin configurations on $B_r(v)$ induced by $(\hsig, \hat{G}(\hsig))$ and $\hsig_r = \hsig_{\partial_r(v)}$.

\begin{lemma}\label{lem:local-dtv-planted-to-tree}
Fix $\beta < \beta_r$ and  $\eta \in (-1,1)$.  
For any $r = o(\log n)$, we have
    $$d_{\tv}(\hsig_{B_r(v)}, \tau_{B_r(x)}) = n^{-1/2 + o(1)}\,,$$
    where we use the shorthand of total-variation distance between random variables to mean the distance between their laws. 
\end{lemma}

\begin{proof}
    We proceed by induction on $r$ and construct a coupling between the two distributions.  The base case of $r = 0$ is clear: the marginal probabilities of plus spins for $\hsig(v)$ and $\nu(x)$ are each $\frac{1+\eta}{2}$ so we may couple exactly.

    Now suppose $r \geq 1$ and that $\hsig_{B_{r-1}(v)}$ and $\tau_{B_{r-1}(x)}$ coincide. Our goal is to couple $\hsig_r$ and $\tau_r$. Note first that the coupling fails if $B_r(v)$ contains a cycle, which occurs with probability $n^{-1+o(1)}$ (this is standard to see from exposing $B_r(v)$ in the configuration model).

    Now, if we condition on $\hsig$ and the edge statistics $\rho$, then $\hbg(\hsig)$ is  distributed uniformly on graphs with the given parameters. We may think of specifying $\hat G(\hsig)$ by (1) taking an ambient collection of half-edges with spins assigned and edges each consisting of a pair of half-edges
    where the number of each edge type is specified by $\hsig$, and (2) building the graph one edge at a time by sampling edges without replacement uniformly at random. 
    More precisely, given $\hsig_{B_{r-1}(v)}$ and $u \in \partial_{r-1}(v)$, the remaining $d-1$ neighbors---say $w_1, \dots, w_{d-1}$---of $u$ can be chosen uniformly at random from the collection of available edges where one half-edge has spin $\sigma(u)$. 
    If $\sigma(u) = -1$, for example, then $\hat\Prob(\hsig(w_1) = +1)$ is exactly the probability that when starting from a uniformly random $-1$ half-edge, its paired half-edge is a plus, which is $\frac{1-\rho}{1-\eta}$. Similarly, $\hat\Prob(\hsig(w_2) = +1\,|\,\hsig(w_1) = +1) =  \frac{(1-\rho)\frac{dn}{2}-1}{(1-\eta)\frac{dn}{2}-1}$. 
  
    In general, each edge probability has a $O\left(\frac{1}{n}\right)$ difference from that of the tree, and since $|\partial_{r}(v)| = d^r \leq n^{o(1)}$, the difference in probabilities overall is $O(n^{o(1)-1})$. 

    Lastly, we take into account the difference in the edge statistics themselves: given $\sigma(u) = -1$, the difference in probability of its first child being $+1$ is $\frac{|\hat\rho - \rho_{\eta}|}{1-\eta}$. 

    Taking a union bound over $\partial_r(v)$, the probability that the coupling fails in the extension to $\partial_r(v)$ is at most $n^{o(1)-1} + d^r\hat\E[\frac{|\hat\rho - \rho_{\eta}|}{1-\eta}]$. By \Cref{lem:planted-edge-stats}, $\hat\E[|\hat\rho - \rho_{\eta}|] = O(n^{-1/2})$, so the failure probability is at most $n^{o(1)-1/2}$, which gives the claim. 
\end{proof}

Our main goal is now to import non-reconstruction results for $\nu$ at $\beta<\beta_r$ to the local neighborhoods of the fixed-magnetization Ising model. 
The above planting showed that the laws induced on the boundary of a small ball $B_r(v)$ for $r=o(\log n)$ are close, as formalized in the following corollary. 

\begin{corollary}\label{cor:ball-bdy-close-to-tree-measure}
	Fix $x \in \T_d$, a configuration $\xi$ on the boundary $\partial_r(x)$, and magnetization $m_r$ sufficiently close to $\mathbb E_{\nu}[\tau_{B_r(x)}]$ such that  
	\begin{align}
		\nu\Big( \sum_{w\in B_r(x)} \tau(w)  = m_r , \tau_{r} = \xi \Big) \ge e^{ - 2 d^r}\,. \label{assumption-tree-magnetization}
	\end{align}  
	If $v$ is a vertex chosen uniformly at random from $[n]$, for any $r= o(\log \log n)$, 
	\begin{align*}
		d_\tv\Big(\hat \Prob\big(\hsig_{B_r(v)} \in \cdot  \mid m_r, \hsig_{r} = \xi\big), \nu\big(\tau_{B_r(x)}\in \cdot & \mid  m_r , \tau_{r} = \xi \big) \Big)  \le  n^{-1/2+o(1)}\,,
	\end{align*}
    where conditioning on $m_r$ means conditioning on the magnetization in $B_r(v)$ or $B_r(x)$ respectively, being equal to $m_r$.
\end{corollary}

\begin{proof}
	Note the general fact that for random variables $X, Y$ and event $A$,
	\begin{align*}
		d_{\tv}\big(\text{Law}( X \mid A), \text{Law}(Y \mid A)\big) \le  2\mathbb P(Y\in A)^{-1} d_{\tv} (X,Y) \,.
	\end{align*}
	The assumption gives a lower bound on the probability of the event being conditioned on under $\nu$; this along with $r=o(\log \log n)$ and \Cref{lem:local-dtv-planted-to-tree} immediately gives the bound. 
	\end{proof}

	We also make a simple observation that typical values of $m_r$ and $\xi$ under $(\hsig, \hat G(\hsig))$ will satisfy \eqref{assumption-tree-magnetization}, and so by restricting to the cases where the assumption is satisfied, we may replace the $\hsig$ conditional measure with the corresponding conditional measure on the tree and not lose much. 
	\begin{observation}
	Let $v$ be a uniformly chosen vertex, $r = o( \log n)$, and $m_r, \xi$ be those given by $\hsig_{B_r(v)}$. Then
		\begin{align*}
			\hat\Prob((\xi,m_r) \text{ satisfy \eqref{assumption-tree-magnetization}}) \ge 1-e^{ -  d^r} - n^{1/2-\delta}\,.
		\end{align*}
	\end{observation}
	
	\begin{proof}
		There are at most $2^{d^r}$ many possible configurations and at most $d^r$ many magnetizations available, so the total mass of $(\xi,m_r)$ pairs violating \eqref{assumption-tree-magnetization} under $\nu$ is at most $2^{d^r} d^r e^{ -2 d^r}$  The above follows from the total-variation bound of \Cref{cor:ball-bdy-close-to-tree-measure}.
	\end{proof}

    	As a result, when considering the law of $\hsig_{B_r(v)}$ \emph{given the configuration outside of $B_r(v)$} we can typically replace it with the distribution given by the tree measure $\nu$ over $B_r(x)$ with a frozen boundary condition $\tau_r$ on $\partial_r(x)$ and a magnetization constraint $m_r$ (by which we mean $m(\tau_{B_r(x)}) = \sum_{w\in B_r(x)} \tau_w = m_r$). Moreover, the total-variation bound on the spin configuration on $B_r(v)$ means that the set of typical pairs $(m_r, \tau_r)$ that have high probability under $\nu$ is also a high probability set under the distribution $\hsig_{B_r(v)}$ of the fixed-magnetization planted model on a ball around a uniformly random vertex.

\subsection{Decay of boundary influence even with fixed magnetization on the tree} 
\label{secDecayBoundary}
	
	In order to apply non-reconstruction results on the tree, we need to be able to drop the magnetization constraint, at least for the typical values of $(m_r,\tau_r)$ that will appear. In this subsection, let $v$ be the label of the root of $\mathbb T_d$.  

	\begin{lemma}\label{lem:drop-magnetization-constraint-tree}
		Fix $\beta < \beta_r$ and $\eta \in (-1,1)$. There exists a set $A_r$ of $({\mathfrak{m}},\zeta)= ({\mathfrak{m}}_r, \zeta_r)$ pairs which has probability $1-o_r(1)$ under $\nu$  such that 
		\begin{align*}
			\max_{({\mathfrak{m}}, \zeta)\in A_r} d_\tv\big(\nu ( \tau_v  \mid {\mathfrak{m}}, \zeta)),\nu ( \tau_v  \mid \zeta ) \big) = o_r(1)\,,
		\end{align*}
        where the shorthand $( \tau_v \mid {\mathfrak{m}})$ means conditioned on $m(\tau_{B_r(v)}) = \mathfrak m$ and $( \tau_v  \mid \zeta)$ means conditioned on $\tau_{r} \equiv \zeta$.
	\end{lemma}

In proving this, we require two additional lemmas. The first is the following  amplified version of Corollary~\ref{cor:various-forms-of-nonreconstruction} which says not only is the root spin not reconstructible, but also there is vanishing difference in the law on any $O(1)$-depth ball around the root. 

\begin{lemma}\label{lem:size-ell-non-reconstruction}
    For $\beta<\beta_r$, for every fixed $\ell < r$, one has that 
    \begin{align*}
        \mathbb E_{\tau_{r}\sim\nu } \big[ d_{\tv}\big(\nu(\tau_\ell \in \cdot \mid \tau_r) - \nu(\tau_\ell \in \cdot) \big)\big ]= o_r(1)\,.
    \end{align*}
\end{lemma}

The second is showing that for typical magnetizations given $\tau_{r}$ (which we recall is the configuration on $\partial_r(v)$), the conditioning on the spin at the root vertex does not shift the law of the magnetization significantly, as it can be corrected by local central limit theorem fluctuations in the magnetization between depth $\ell$ and $r$. 

\begin{lemma}\label{lem:sprinkling}
    In a tree of depth $r$ rooted at $v$, for every depth-$r$ configuration $\zeta= \zeta_r$, there exists a set $\mathcal M(\zeta)$ of magnetizations such that
    $\nu (\mathcal M \mid \zeta) = 1-o_r(1)$, and for all ${\mathfrak{m}} = {\mathfrak{m}}_r \in \mathcal M$ and $\tau_{\ell}$, 
    \begin{align*}
        \nu( {\mathfrak{m} }\mid  \tau_r = \zeta,\tau_{\ell} ) = (1+o_r(1)) \nu({\mathfrak m} \mid \zeta,\tau_\ell,\tau_v = +1)\,.
    \end{align*}
\end{lemma}

We defer the proofs of these two statements and first show how they imply \Cref{lem:drop-magnetization-constraint-tree}.

\begin{proof}[Proof of \Cref{lem:drop-magnetization-constraint-tree}]
Let $p=p_\eta$ be the probability that $\tau_v= +1$ under $\nu$ (so, $2p-1 =\eta$).  Our goal is to find a high probability set $({\mathfrak{m}},\zeta)$ so that
\[ \nu(\tau_v = +1 \mid {\mathfrak{m}},\zeta) = p + o_r(1)\,.\]
Note that this is sufficient for proving the claim; by the non-reconstruction statements of \Cref{cor:various-forms-of-nonreconstruction}, there is a high-probability set of $\zeta$ such that $\nu(\tau_v = + 1 \mid \zeta) = p+o_r(1)$.

We can rewrite the conditional probability as 
\begin{align*}
    \nu(\tau_v = + 1 \mid {\mathfrak{m}},\zeta)  &= \frac{\nu(\tau_v = + 1 ,{\mathfrak{m}},\zeta)}{\nu({\mathfrak{m}},\zeta)} = p \frac{ \nu(\zeta\mid \tau_v = +1)\nu({\mathfrak{ m}} \mid \tau_v = +1 ,\zeta)}{\nu(\zeta) \nu({\mathfrak{m}}\mid \zeta)}\,.
\end{align*}
By \Cref{cor:various-forms-of-nonreconstruction}, 
there is also a set $A'$ of $\zeta$ with probability $1-o_r(1)$ under $\nu$ such that  $\nu(\zeta \mid  \tau_v = +) = (1+o_r(1)) \nu(\zeta)$. 
Therefore, it is enough to show $\nu({\mathfrak{m}} \mid  \tau_v =+1,\zeta) = (1+o_r(1)) \nu({\mathfrak{m}} \mid \zeta)$, or equivalently,
\begin{align}\label{eq:NTS-drop-mag-constraint-1}
\nu({\mathfrak{m}} , \tau_v=+1 \mid \zeta) = (1+o_r(1)) \nu(\tau_v =+ 1 )\nu({\mathfrak{m}} \mid \zeta) \,,
\end{align}
since by non-reconstruction (Corollary~\ref{cor:various-forms-of-nonreconstruction}) for a high probability set of $\zeta$'s, we have $\nu(\tau_v =+ 1 \mid \zeta) = (1+o_r(1)) \nu(\tau_v =+ 1 )$. 

 Towards~\eqref{eq:NTS-drop-mag-constraint-1}, observe that for fixed $\ell < r$, we may write 
\begin{align*}
    \nu({\mathfrak{m}}, \tau_v=+1 \mid  \zeta) &=\sum_{\tau_\ell} \nu(\tau_\ell \mid \zeta) \nu({\mathfrak{m}},  \tau_v=+1 \mid \zeta ,\tau_\ell) \\
    &= (1+o_r(1)) \sum_{\tau_\ell} \nu(\tau_\ell)\nu({\mathfrak{m}}, \tau_v=+1 \mid  \zeta,\tau_\ell)\,,
    \end{align*}
    where the second line follows from \Cref{lem:size-ell-non-reconstruction}, which says that for a high-probability set of $\zeta$, the law on $B_\ell(v)$ for $\ell$ fixed relative to $r$ is within $(1+o_r(1))$ of $\nu(\tau_{B_\ell(v)}\in \cdot)$. Then,
    \begin{align*}
\nu({\mathfrak{m}}, \tau_v=+1 \mid  \zeta) &= (1+o_r(1))\sum_{\tau_\ell} \nu(\tau_\ell) \nu(\tau_v=+1 \mid \zeta, \tau_\ell) \nu ({\mathfrak{m}} \mid \zeta,\tau_\ell, \tau_v=+ 1 ) \\ 
&= (1+o_r(1))(1+o_\ell(1)) \sum_{\tau_\ell} \nu(\tau_\ell) \nu(\tau_v=+1 ) \nu( {\mathfrak{m}} \mid \zeta,\tau_\ell, \tau_v=+1) \,,
\end{align*}
where here $o_\ell(1)$ means first taking $r\to\infty$ with $\ell$ fixed, then subsequently sending $\ell \to\infty$. Note that this step used that for a high-probability set of configurations at level $\ell$, we can apply non-reconstruction in the form of Corollary~\ref{cor:various-forms-of-nonreconstruction} at that level. 

The last and key step is the application of \Cref{lem:sprinkling}, which shows that the probability of magnetization ${\mathfrak m}$ (which is in a set with high probability under $\nu(\cdot \mid \zeta)$) is not affected much by the configuration on $B_\ell(v)$ because $\ell$ is fixed as $r\to\infty$. This relies on a ``sprinkling" argument to show that around $\mathfrak{m}$ which are typical under $\nu(\cdot \mid \zeta)$, the probability mass function of the magnetization satisfies a local central limit theorem and therefore can compensate for different realizations of $\tau_\ell$ (which cause order-1 shifts in the magnetization compared to $r$). 
Namely, Lemma~\ref{lem:sprinkling} says that there exists $\mathcal M(\zeta)$ a high probability set of magnetizations such that for all $\mathfrak{m} \in \mathcal M$, 
\begin{align}\label{eq:sprinkling-step-in-proof}
    \nu( {\mathfrak m} \mid \zeta,\tau_\ell, \tau_v=+1)  = (1+o_r(1)) \nu( {\mathfrak m} \mid \zeta,\tau_\ell)\,. 
\end{align}
We thus arrive at a high probability set of $\zeta$ and their corresponding $\mathcal M(\zeta)$'s such that 
\begin{align*}
    \nu({\mathfrak m} , \tau_v=+1 \mid \zeta)&  = (1+o_\ell(1))\nu(\tau_v =+ 1) \sum_{\tau_\ell} \nu( \tau_\ell) \nu({\mathfrak m} \mid \zeta,\tau_\ell)(1+o_r(1)) \\ 
    & = (1+o_r(1))\nu(\tau_v = + 1) \nu({\mathfrak m} \mid \zeta)\,,
\end{align*}
where the last line follows from taking $\ell$ to be growing sufficiently slowly with $r$. 
\end{proof}

It remains to prove the two deferred lemmas. We start with the stronger form of non-reconstruction for constant-depth balls around the root.

\begin{proof}[Proof of \Cref{lem:size-ell-non-reconstruction}]
    We proceed by induction on $\ell$ and aim to show that 
    \begin{equation} \label{eqn:size-ell-nonreconstruction}
    \mathbb E_{\tau_{B_r(v)}\sim\nu } \big[ d_{\tv}\big(\nu(\tau_\ell \in \cdot \mid \tau_r) - \nu(\tau_\ell \in \cdot) \big)\big ]= C_{\ell} \cdot o_r(1)
    \end{equation}
    for a sequence $C_\ell$ (growing exponentially in $\ell$).
    
    For the base case, note that since $\beta < \beta_r$, by Corollary~\ref{cor:various-forms-of-nonreconstruction}, 
    there exists a set $A_r$ with high probability under $\nu$ such that for all $\tau \in A_r$, one has 
    $$|\nu(\tau_v = +1 \mid \tau_r) - \nu(\tau_v = +1)|=o_r(1)\,.$$ 
    
    Now suppose that \eqref{eqn:size-ell-nonreconstruction} holds for $\ell-1$. Let $u_{1,1},...,u_{1,d},...,u_{N,1},...,u_{N,d}$ be the vertices at level $\ell$ in the tree, so $N = d^{\ell-1}$. We can express the distribution of configurations on level $\ell$ in terms of those at level $\ell-1$ as follows. 
    \begin{align*}
        \nu(\tau_\ell = \varsigma \mid \tau_r)  &  = \nu(\tau(u_{1,1},...,u_{N,d}) =\varsigma \mid \tau_r) \\
        & = \mathbb E_{\tau_{\ell-1}\sim \nu} \Big[ \prod_{i}\nu(\tau(u_{i1},...,u_{i,d})= \varsigma_i \mid \tau_r ,\tau_{\ell-1}(i))\Big] \,,
    \end{align*}
    where $i$ is the parent (at level $\ell-1$) of vertices $u_{i1},...,u_{id}$, $\zeta_i$ is the vector configuration on the vertices $u_{i1},...,u_{id}$, and $\tau_{\ell-1}(i)$ is the spin at $i$.

    By the inductive assumption and the fact that the argument of the expectation is bounded by $1$, for most $\tau_r$ (say those in the high probability set $\tau_r \in A_{\ell-1}$) the outer expectation can be replaced by $\tau_{\ell-1}$ drawn independently from the infinite-volume distribution, up to error $C_{\ell-1} \cdot o_r(1)$. Next, we analyze the terms in the product.
    \begin{align}\label{eq:zeta-to-product}
        \nu(\tau(u_{11},...,u_{1d}) = \varsigma_1 \mid \tau_r, \tau_{\ell-1}(1)) = \prod_{j} \nu(\tau(u_{1j})= \varsigma_{1j} \mid \tau_{\ell-1}(1),\tau_r(T_{u_{1j}}))\,.
    \end{align}
    Here, $\varsigma_{1j}\in \{\pm 1\}$ and $T_{u_{1j}}$ is the subtree rooted at $u_{1j}$ with $\tau_r(T_{u_{1j}})$ being the restriction of $\tau_r$ to that subtree. We can remove the effect of the parent for each of these terms as follows. 
    For an Ising model on a tree with root $v$ and boundary at depth $\ell$, and adding a parent vertex above the root called $0$, algebra with ratios of partition functions shows that   
    \begin{align*}
        \nu_{T_v \cup \{0\}}(\sigma_v = + \mid \sigma_0=+, \tau_r)  = \frac{e^{2\beta} \nu_{T_v}(\sigma_v = + \mid \tau_r)}{1-(1-e^{2\beta})\nu_{T_v}(\sigma_v = + \mid \tau_r)} =: F_{++}(\nu(\sigma_v = + \mid \tau_r))\,.
    \end{align*}
   For every fixed $\beta$, the function $F_{++}$ is Lipschitz and bounded away from $0$ and $1$; likewise the same holds for $F_{+-}, F_{-+}$ and $F_{--}$. Thus, if $x$ is close to $x_*$ then $F_{++}(x)$ is close to $F_{++}(x_*)$. 
   
   Returning to~\eqref{eq:zeta-to-product}, we can apply this bound with $r$ replaced by $r-\ell$ and we get 
   \begin{align*}
       \nu(\tau(u_{11},...,u_{1d}) = \varsigma_1 \mid \tau_r, \tau_{\ell-1}(1)) = \prod_j F_{\tau_{\ell-1}(1),\varsigma_{1j}} (\nu_{T_{u_{1j}}}(\tau_v =\varsigma_{1j} \mid \tau_r(T_{u_{1j}}))\,.
   \end{align*}
   We now note that by non-reconstruction (and fact that $r-\ell$ goes to infinity with $r$), one has that as long as $\tau_r(T_{u_{1j}})\in A_{1}$, then the inputs $\nu_{T_{u_{1j}}}(\tau_v =\varsigma_{1j} \mid \tau_r(T_{u_{1j}}))$ are within $o_r(1)$ of $\nu(\tau_v = \varsigma_{1j})$. As we are taking a finite product (over $d$ many terms, and then another $d^{\ell-1}$ many terms), if $\tau_{\ell-1}\in A_{\ell-1}$ and each of its subtrees of the depth $\ell$-vertices are in $A_1$, then 
   \begin{align*}
       \prod_i\nu(\tau_\ell(u_{i1},...u_{id}) = \varsigma_i\mid \tau_r, \tau_{\ell-1}(i))  = \prod_i \prod_j F_{\tau_{\ell-1}(i),\zeta_{ij}} (\nu(\tau_v = \zeta_{1j})) +o_r(1)\,.
   \end{align*}
   The probability of not having the above is therefore bounded by $d^{\ell-1}$ times the $o_r(1)$ error probability inherited from the level up. 
   The right-hand side does not depend on $\tau_r$, and by the inductive assumption, the expectation over $\tau_{\ell-1}$ is taken with respect to $\nu$, up to a $o_r(1)$ error. This completes the claim.
\end{proof}

We now turn our attention to \Cref{lem:sprinkling}. To prove this, we require an intermediate lemma which bounds the variance of the magnetization conditional on $\zeta_r$. This will follow from the spectral gap bound of~\cite{martinelli2004glauber} up to $\beta_r$.

\begin{lemma}\label{lem:conditional-variance-bound}
    If $\beta<\beta_r$, there exists a constant $C(d,\beta)>0$ such that on a $d$-regular tree $T_r$ with boundary conditions $\zeta$ at depth $r$ and external field $h$, it has 
    \begin{align*}
        \mbox{\emph{Var}}_{\nu(\cdot \mid \tau_{r} = \zeta)}\Big( \sum_{v\in T_r} \tau_v \Big)  \le C |T_r|\,.
    \end{align*}
\end{lemma}

\begin{proof}
    In~\cite[Theorem 1.1]{martinelli2004glauber}, it was shown that on the $(d-1)$-ary tree, for $\beta<\beta_r$ and any external field $h\in \mathbb R$, there is a constant $c_{\mathsf{gap}}>0$ such that the (continuous-time) Glauber dynamics spectral gap is at least $c_{\mathsf{gap}}$ uniformly over boundary conditions and depth of the tree. By standard spectral gap comparison~\cite[Section 13.3]{levin2017markov} the spectral gap on a $d$-regular tree of depth $r$ is comparable, up to a constant $C(\beta)$, to the spectral gap of the $(d-1)$-ary tree of depth $r-1$, and thus there is an $\Omega(1)$ inverse spectral gap $c'_{\mathsf{gap}}$ lower bound on the $d$-regular tree's spectral gap (uniformly over the boundary conditions and depth).
    
    If we then apply the variational form of the spectral gap to the magnetization test function, and note that the continuous-time Dirichlet form for the magnetization test function is bounded by the number of vertices, which is simply $|T_r|$, we obtain the claim. 
\end{proof}

We also use the following local central limit theorem for sums of independent (but not necessarily identically distributed) binomial random variables; see, e.g.,~\cite[Theorem 1.1]{DaMc-LCLT}.  
This will translate to a local central limit theorem on integers of the same parity by the simple linear transformation from $k = $ number of $+1$ spins and $\mathfrak m = $ magnetization.

        \begin{lemma}\label{lem:local-CLT-rademachers}
            Let $(X_i)_{i=1}^{N}$ be independent random variables such that  $X_i \sim \text{Bin}(p_i)$ with $\bar p \le p_i \le 1-\bar p$, and let $S_N = \sum_{i=1}^N X_i$. Noting that $\Var(S_N) = \sum_i 4p_i (1-p_i)$, uniformly over $r= O(\sqrt{N})$ we have 
            \begin{align*}
                \mathbb P(|S_N - \mathbb E[S_N]|  = r)  = (1+o_N(1)) \frac{1}{\sqrt{2\pi \sum_i p_i(1-p_i)}} \exp\Big( - \frac{r^2}{2\sum_i p_i (1-p_i)}\Big)\,.
            \end{align*}
        \end{lemma}

We can now prove Lemma~\ref{lem:sprinkling}, which was used in equation~\eqref{eq:sprinkling-step-in-proof}.

\begin{proof}[Proof of \Cref{lem:sprinkling}]
    Recall the lemma statement: we want to show that for every $\zeta= \zeta_r$, there exists a set $\mathcal M(\zeta)$ of magnetizations such that $\nu (\mathcal M \mid \zeta) = 1-o_r(1)$, and for all ${\mathfrak{m}} = {\mathfrak{m}}_r \in \mathcal M$ and $\tau_{\ell}$, 
    \begin{align*}
        \nu( {\mathfrak{m} }\mid  \tau_r = \zeta,\tau_{\ell} ) = (1+o_r(1)) \nu({\mathfrak m} \mid \zeta,\tau_\ell,\tau_v = +1)\,.
    \end{align*}
    
    We begin by taking the left-hand side and additionally conditioning on all of $\tau_{\le \ell}= \tau_{B_\ell(v)}$. It suffices to show that for any two $\tau_{\le \ell},\tau_{\le \ell}'$ with the same values on $\tau_\ell = \tau'_\ell$, one has 
    \begin{align*}
        \nu({\mathfrak m} \mid \tau_{\le \ell},\tau_r = \zeta) = (1+o_r(1)) \nu({\mathfrak m} \mid \tau'_{\le \ell},\tau_r =\zeta)\,.
    \end{align*}
    Equivalently, by the Markov property,  we will show that for all $\zeta$, there exists  $\mathcal M$ such that 
    \begin{align*}
        \nu(\mathcal M \mid  \tau_r = \zeta) = 1-o_r(1)\,,
    \end{align*}
    such that for any $\tau_{\le \ell}$, one has uniformly over $2k\in \{-2d^\ell,...,2d^\ell\}$ (note the magnetization changes by two for every spin flip in $T_\ell$) that 
    \begin{align*}
        \nu({\mathfrak m} \mid \tau_{\le \ell}, \tau_r =\zeta)  = (1+o_r(1)) \nu({\mathfrak m} + 2k \mid \tau_{\le \ell}, \tau_r = \zeta)\,
    \end{align*}
    for all $\mathfrak m \in \mathcal M$. We define explicitly the set of magnetizations, 
    \begin{align*}
        \mathcal M_\eps =\mathcal M_\eps (\zeta)  = \big\{ {\mathfrak{m}}: \nu(m(\tau) = {\mathfrak{m}} \mid \tau_r = \zeta) > \eps d^{-r/2}\big\}\,.
    \end{align*}
    We will argue that for $\eps$ suitably chosen, $\mathcal M_\eps$ has high probability, and also that different values of ${\mathfrak m}$ which are at constant (relative to $r$) distance from one 
another have a $(1+o_r(1))$ factor difference of probability. 

We first note that by Chebyshev's inequality and Lemma~\ref{lem:conditional-variance-bound}, there exists $C  = C(\beta, d)$ (where $C$ absorbs the constant ratio between $|T_h|$ and $d^r$) such that for all $K$, 
\begin{align*}
    \nu\big( |m(\tau) - \mathbb E_\nu[m(\tau) \mid \tau_r =  \zeta] | > C K d^{r/2} \mid \tau_r = \zeta\big) \le \frac{1}{K^2}\,.
\end{align*}
Thus, for any $\delta>0$, there is a $K$ large enough and $\eps$ small enough such that 
\begin{align*}
    \nu(m(\tau) \notin \mathcal M_\eps \mid \tau_r = \zeta) & \le \sum_{{\mathfrak m}: |{\mathfrak m} -  \mathbb E_\nu [m (\tau)
    | \tau_r = \zeta]|> C K d^{r/2}}\nu( {\mathfrak m}\mid \tau_r = \zeta) \\
    &  \qquad + \sum_{{\mathfrak m}: |{\mathfrak m} -  \mathbb E_\nu [m (\tau)
    | \tau_r = \zeta]|< C K d^{r/2}, {\mathfrak m}\notin \mathcal M_\eps} \nu(  \mathfrak{m} \mid \tau_r = \zeta)  \\ 
    & \le \frac{1}{K^2}  + C \eps K < \delta\,.
\end{align*}
Now we use sprinkling (resampling on an independent set) to argue that for any ${\mathfrak{m}}\in \mathcal M_\eps$, the probability mass function is close to that on ${\mathfrak{m}}\pm 2k$ for $2k$ fixed relative to $r$. Let $U = \partial T_{r-1}$ be the set of vertices one level up from the root, and condition on the configuration $\tau(U^c)$. Then, 
\begin{align}\label{eq:nu(m)-conditioning}
    \nu({\mathfrak{m}} \mid \tau_r =\zeta,\tau_{\le \ell})  = \mathbb E_{\tau (U^c) \sim \nu(\cdot \mid \zeta, \tau_{\le \ell})} [ \nu(  {\mathfrak m}\mid \tau(U^c), \tau_r = \zeta, \tau_{\le \ell} )]\ .
\end{align}
    We claim first that the above expectation is dominated by the configurations $\tau(U^c)$ for which 
    \begin{align*}
        \nu({\mathfrak m} \mid \tau(U^c) , \tau_r = \zeta, \tau_{\le \ell}) \ge \eps^2 d^{-r/2}\,.
    \end{align*}
    Let $\mathcal A$ be the set of such configurations $\tau(U^c)$, and observe that by definition, for $\mathfrak m\in \mathcal M_{\eps}$,   
    \begin{align*}
        \mathbb E_\nu[\ind{\mathcal A^c} \nu({\mathfrak m} \mid \tau(U^c), \tau_r = \zeta,\tau_{\le \ell}) \mid \tau_r = \zeta, \tau_{\le \ell} ] \le \eps^2 d^{-r/2} \le \eps \nu({\mathfrak m}\mid \tau_r= \zeta,\tau_{\le \ell})\,.
    \end{align*}
    For the remainder of the contribution to~\eqref{eq:nu(m)-conditioning}, we get 
    \begin{align*}
        \nu(\mathcal A) \mathbb E_\nu [ \nu({\mathfrak m} \mid \tau(U^c) , \zeta_r,\zeta_{\le \ell} ) \mid \mathcal A, \tau_r = \zeta, \tau_{\le \ell}]\,.
    \end{align*}
    We now apply a local central limit theorem, Lemma~\ref{lem:local-CLT-rademachers} with $N = |U| = d^{ r}$ and $\bar p = e^{ - 2\beta d}$,  to see that  
    \begin{align*}
        \nu(\mathfrak m \mid \tau(U^c) , \tau_r= \zeta, \tau_{\le \ell}) = (1+o_N(1)) \frac{1}{\sqrt{2\pi \sigma_{\tau}^2}} \exp( - ({\mathfrak m}- \E{[\mathfrak{m}} \mid \tau(U^c)])^2/2\sigma_\tau^2)\,,
    \end{align*}
    where $\sigma_\tau^2$ is the conditional variance of $\mathfrak{m}$ given a realization $\tau(U^c)$ in $\mathcal A$. 
    
    If $\tau(U^c) \in \mathcal A$, then ${\mathfrak m}$ is in the moderate deviation regime in the above bound (meaning the Gaussian probability mass function is order $1/\sigma_\tau$). The same then holds for $\mathfrak{m}+ O(1)$, and so by the above, one has for all $2k\in \{- 2d^\ell, ...,2d^\ell\}$ that 
    \begin{align*}
        \nu({\mathfrak m} + 2k \mid \tau(U^c), \tau_r =\zeta, \tau_{\le \ell})  = (1+o_r(1)) \nu({\mathfrak m} \mid \tau(U^c), \tau_r =\zeta, \tau_{\le \ell})\,.
    \end{align*}
    If $\tau(U^c) \notin \mathcal A$, then the same local central limit theorem tells us that since ${\mathfrak m}$ has probability at most $\eps^2 d^{-r/2}$, we have 
    \begin{align*}
        \sup_{2k\in \{-2d^\ell,...,2d^\ell\}} \nu({\mathfrak m}+2k \mid \tau(U^c) , \zeta_r,\zeta_{\le \ell}) \le (1+o_r(1)) \eps^2d^{-r/2}\,.
    \end{align*}
    Combining these bounds, for any ${\mathfrak m} \in \mathcal M_\eps$ we obtain that ~\eqref{eq:nu(m)-conditioning} equals
    \begin{align*}
        \nu(\mathcal A^c) (1+o_r(1)) \eps^2d^{-r/2} &  + (1+o_r(1))  \mathbb E_\nu[ \ind{\mathcal A}  \nu({\mathfrak m} \mid \tau(U^c) , \tau_r = \zeta,\tau_{\le \ell}) ] \\ 
        & \le (1+o_r(1)) \nu({\mathfrak m}+ 2k \mid \tau_r = \zeta, \tau_{\le \ell}) + 2\eps^2 d^{-r/2} \\ 
        & \le (1+ 2\delta + o_r(1)) \nu({\mathfrak m} + 2k \mid \tau_r = \zeta, \tau_{\le \ell})\,,
    \end{align*}
    for $\eps$ sufficiently small depending on $\delta$. Sending $\delta$ to $0$ after $r\to \infty$ while keeping $\ell$ fixed yields the claimed bound.
\end{proof}

\subsection{Overlap statistics in the planted model}
We now analyze the statistics of pairs of configurations in the planted model. Recall that for a single configuration $\sigma$, we may represent the edge statistics as a $2 \times 2$ matrix parametrized in terms of magnetization $\eta$ and monochromatic edge density $\rho$ and that for a pair of configurations $\sigma$ and $\sigma'$, we may represent the edge overlap statistics as a $4 \times 4$ matrix $R_{\sig,\sig'}$ where 
the entry $R_{ij}$ is equal to the fraction of edges which are type $i$ in $\sigma$ and type $j$ in $\sigma'$.

The following lemma shows that a planted graph does not carry significant correlations between two independent samples from the fixed-magnetization Ising model on it. 
\begin{lemma}\label{prop:typical-overlap}
    Let $\sig, \sig'$ denote independent samples from $\mu_{G, \eta}$ for $G\sim \hbg$. Then 
    $$\hat\E[\|R_{\sig,\sig'} - \rho_{\nu_{\eta}}^{\otimes 2}\|_2] = o(1)\ .$$
\end{lemma}

\begin{proof}
    By the Nishimori identity, we may instead consider a uniformly random assignment $\hsig \in \Omega_{\eta}$ against $\sig \sim \mu_{\hbg_{\eta}(\hsig)}$ and show that
    $|\hsig \cdot \sig - \eta^2| =o(1)$ with high probability.
    
    To prove this, let $N_{s} = |\{v : \hsig_v = s, \sig_v = +\}|$ for $s \in \{+,-\}$, and let $Q_{s,t} = |\{uv \in E(\hat{G}) : \hsig_u = s, \hsig_v = t, \sig_u = \sig_v = + \}|$.
    By the previous discussion, the quantities $N_s$ and $Q_{s,t}$ determine the collection of vertex-pair and edge-pair statistics, respectively. Let $\rho_{\eta}^+$ denote the fraction of edges which are $++$ under $\nu_{\eta}$. 

    With the aim of using the second moment method, we consider the mean and variance of $N_{s}$ and $Q_{s,t}$. By the Law of Total Expectation, we may condition on $\hsig$. For the first moment, we have
    \begin{align*}\hat\E[N_{s}\mid \hsig] & = \hat\E\Big[\sum_{v} \ind{\hsig_v = s, \sig_v = +}\mid \hsig\Big] = \hat\E\Big[\sum_{v:\hsig_v=s}\ind{\sig_v = +}\mid \hsig\Big] \\
    &= |\{v: \hsig_v =s\}| \cdot \hat\Prob(\sig_v =+\ |\ \hsig)\,,
    \end{align*}
    and similarly 
    \begin{align*} \hat\E[Q_{s,t} \mid \hsig] &=  \sum_{u,v} \ind{\hsig_u = s, \hsig_v = t}\hat\Prob(uv \in E(\hat{G}) \mid \hsig) \cdot \hat\Prob(\sig_u = +, \sig_v = + \mid uv \in E(\hat{G}))\\
    &= \hat\E[|\{uv \in E(\hat{G}) : \hsig_u = s, \hsig_v = t\}|] \hat\Prob(\sig_u = +, \sig_v = + \mid uv \in E(\hat{G}))
    \end{align*}
    To show that the first moments are close to those given by the tree measure, let $r$ be a slowly diverging function of $n$ that is  $o(\log \log n)$, say $r=\log \log \log n$. Let $v$ be a uniformly random vertex with $\hsig_v= s$. Draw $\sig \sim \mu_{\hbg_{\eta}(\hsig), \eta}$ and generate a configuration $\sig'$ by resampling the colors of vertices in $B_r(v)$ conditionally on the configuration in $B_r(v)^c$. Note that we still have $\sig' \sim \mu_{\hbg_{\eta}(\hsig), \eta}$. 
    
    By Lemma~\ref{lem:local-dtv-planted-to-tree}, with  probability $1-o(1)$, the boundary values on $\partial_r(v)$ and the residual magnetization in the interior of $B_r(v)$ are in the good set $A_r$ of \Cref{lem:drop-magnetization-constraint-tree}. Intersecting that set of boundary values and residual magnetizations with the set of Corollary~\ref{cor:ball-bdy-close-to-tree-measure} still leaves a high probability configuration which by Corollary~\ref{cor:ball-bdy-close-to-tree-measure} is such that the subsequent resampling inside the ball $B_r(v)$ may equivalently be done under the tree measure $\nu$ conditioned on the boundary conditions and residual magnetization. 
    By Lemma~\ref{lem:drop-magnetization-constraint-tree}, one has non-reconstruction in this conditional resampling stage inside, and the distribution of $\sig'_v$ is exactly $\nu + o(1)$. Thus, using that $\hsig$ has exactly $\frac{1+\eta}{2}n$ vertices with $+$ spins, we get 
    \begin{align*}
        \hat\E[N_{s} \mid \hsig] &= \left(\frac{1+\eta}{2}\right)^2n+o(n)\,,\\
        \hat\E[Q_{s,t} \mid \hsig] & = (\rho_{\eta}^+)^2\frac{dn}{2} + o(n)\ .
    \end{align*}

    We proceed similarly for the second moment. Let $v, w$ be uniformly random vertices such that $\hsig_v = s = \hsig_w$. Draw $\sig \sim \mu_{\hbg_{\eta}(\hsig), \eta}$ and generate $\sig'$ by resampling both $B_r(v)$ and $B_r(w)$ conditioning on the boundaries, noting that with high probability the balls are disjoint. By non-reconstruction following from the same arguments as above, we have that $\Prob(\sig'(v) = s, \sig'(w) = t\ |\ \hsig) = \nu(s)\nu(t) + o(1)$. 

    This tells us that $\E[N_s^2] = \E[\E[N_s^2|\hsig]] = \E[N_s]^2 + o(n^2)$ and $\E[Q_{s,t}^2] = \E[Q_{s,t}]^2 + o(n^2)$. The claim then follows from Chebyshev's inequality.
\end{proof}

\subsection{Second-moment computation}

We now put everything from the previous subsections together to prove \Cref{prop:truncated-moments}, from which Theorem~\ref{thmFreeEnergy} was already shown to follow. 

\begin{proof}[Proof of \Cref{prop:truncated-moments}]

Recall we define $\cE := \{G : \E_{\mu_{G,\eta}^{\otimes 2}}[ \| R_{\sig, \sig'} - \rho_{\nu_{\eta}}^{\otimes 2} \|_2 ] = o(1)\}$, where we recall that $R_{\sigma,\sigma'}$ implicitly depends on $G$. By the planted model definition~\eqref{defn:plantedgraph} and \Cref{prop:typical-overlap}, we have
$$\frac{\E[Y_{G}]}{\E[Z_{G}]} = \hat\Prob(\hat G \in \cE) = 1-o(1)\ .$$

For the second moment, let $A_G({\sig,\sig'}) = \{\| R_{\sig, \sig'} - \rho_{\nu_{\eta}}^{\otimes 2} \|_2 = o(1)\}$. 
By definition
\begin{align*}
    \E[Y_{G}^2] & = \E[ \sum_{\sigma,\sigma'} \ind{\cE} e^{\beta (H(\sigma)  + H(\sigma'))}] \leq (1-o(1)) \E[ \sum_{\sig, \sig'} \ind{A_G(\sigma,\sigma')} e^{\beta (H(\sig) + H(\sig'))}]\,.
\end{align*}
As in the first-moment computation leading up to~\eqref{eq:g-eta-rho-function}, we may reparametrize the summation with respect to the vertex and edge overlap statistics of the two configurations. 
Given vertex overlaps $M = (M_{++}, M_{+-}, M_{-+}, M_{--})$ where $M_{ij}$ is the fraction of vertices with spin $i$ in $\sig$ and spin $j$ in $\sig'$, the number of configuration pairs $\sig, \sig'$ achieving a given $M$ is counted by a multinomial that we may approximate using Stirling's formula as $\exp(-n\sum_{i,j \in \{\pm\}} M_{ij} \log M_{ij} + O(\log n))$. If we further specify the edge overlap statistics in the matrix $R$, the number of ways to place edges achieving $R$ is similarly approximated by $\exp(-\frac{nd}{2} \sum_{i,j} R_{ij} \log R_{ij} + O(\log n))$. Thus, we have 
\begin{align*}\E[Y_{G}^2] \leq \sum_{M, R} \ind{M,R \in A_G(\sig, \sig')}\exp\Big(n  ( & -\sum_{i,j \in \{\pm\}} M_{ij} \log M_{ij}-\frac{d}{2}\sum_{i,j} R_{ij} \log R_{ij}) \\
& + \frac{\beta nd}{2}(\sum_{i,j} R_{++,ij} + R_{ij,++}) + o(n)\Big)\ .\end{align*}
Because of the indicator function and continuity of the entropy function, we may replace $M$ and $R$ with their corresponding values in $\rho_{\nu_{\eta}}^{\otimes 2}$ and only incur $o(1)$ error. This gives us exactly $-\sum_{i,j \in \{\pm\}} M_{ij} \log M_{ij} = 2\mathsf{H}\left(\frac{1+\eta}{2}\right)\log 2$ for the vertex overlap and $-\sum_{i,j} R_{ij} \log R_{ij} = 2\Psi(\eta, \rho)$ and $\sum_{i,j} R_{++,ij} + R_{ij,++} = 2\rho$ for the edge overlap. Thus, recalling~\eqref{eq:g-eta-rho-function}, 
\[\E[Y_{G}^2] \leq \exp(2n f_{d, \beta}(\eta) + o(n))= e^{o(n)} \E[Z_{G}]^2\,.  \qedhere\]
\end{proof}

\subsection{Proofs of consequences}\label{subsec:consequences}
We now deduce  corollaries of the free energy density result of Theorem~\ref{thmFreeEnergy}.

\begin{proof}[\textbf{\emph{Proof of Corollary~\ref{corZBconj}}}]
First note that setting $\eta=0$ and maximizing~\eqref{eq:g-eta-rho-function} over $\rho$ gives $\rho_{0}  = e^{\beta}/(1+e^{\beta})$.  Plugging this value into~\eqref{eq:g-free-energy-function} yields the formula~\eqref{eqfdb0} for the free energy density: $f_{d,\beta}(0) = \log 2 + \frac{d}{2} \log \left( \frac{1+ e^{\beta}}{2} \right )$.  Likewise, for $\beta<\beta_r$, ~\cite{coja2022ising} proved that the free energy density of the anti-ferromagnetic Ising model is $\log 2 + \frac{d}{2} \log \left( \frac{1+ e^{-\beta}}{2} \right ) $.  

We will now show that with high probability over $G \sim \bg_d(n)$,
\begin{align*}  \frac{1}{n} \E_{ \mufix_{G,\beta,0}} [H_G(\sigma)] &= \frac{\partial}{\partial \beta} \left [  \log 2 + \frac{d}{2} \log \left( \frac{1+ e^{\beta}}{2} \right )\right ] +o(1) =  \frac{d}{2} \cdot \frac{e^{\beta}}{1+ e^{\beta}} +o(1) 
     \intertext{and}
     \frac{1}{n} \E_{\muanti_{G,\beta}}  [H_G(\sigma)] &= - \frac{\partial}{\partial \beta} \left [  \log 2 + \frac{d}{2} \log \left( \frac{1+ e^{-\beta}}{2} \right )\right ] +o(1) =  \frac{d}{2}  \cdot \frac{1}{1+ e^{\beta}}  +o(1) \,.
\end{align*}
    Adding these expressions together gives $\frac{d}{2} + o(1)$, yielding Corollary~\ref{corZBconj}.  We prove the upper bound for the first equation since the proofs of other bounds are nearly identical.  The proof essentially just follows that of Griffith's Lemma (e.g.,~\cite{ruelle1969statistical}), but we take a little care to combine it with the high probability statement.  

    Fix $\eps>0$.  We  show that with probability $1-o(1)$, 
    \begin{equation}
        \label{eqFderv1}
    \frac{1}{n} \E_{ \mufix_{G,\beta,0}} [H_G(\sigma)] < \frac{d}{2} \cdot \frac{e^{\beta}}{1+ e^{\beta}}  + \eps \,.
\end{equation} 
Choose $\delta >0$ so that 
\begin{align}
\label{eqderivdef}
  \left |  \frac{f_{d,\beta+\delta}(0)-f_{d,\beta}(0) }{\delta }  - \frac{d}{2} \cdot \frac{e^{\beta}}{1+ e^{\beta}}  \right|  < \eps /2 \,,
\end{align}
 and $\beta +\delta < \beta_r$, which is possible by the definition of a derivative and the assumption that $\beta<\beta_r$. 
A standard computation shows that for any $G$, 
\begin{align*}
    \frac{1}{n} \E_{ \mufix_{G,\beta,0}} [H_G(\sigma)] &= \frac{1}{n} \frac{\partial}{\partial \beta} \log Z_{G,\beta}(0) \,,
\end{align*}
and moreover, $\log Z_{G,\beta}(0)$ is a  convex function of $\beta$ (which one can see  by taking the second derivative of $\log Z_{G,\beta}(0)$).  This means that for any $\delta>0$,
\begin{align}
\label{eqFEconvex}
     \frac{1}{n} \E_{ \mufix_{G,\beta,0}} [H_G(\sigma)] &\le \frac{\frac{1}{n} \log Z_{G,\beta+\delta}(0) - \frac{1}{n} \log Z_{G,\beta}(0)}{\delta} \,.
\end{align}
Since $\beta+\delta<\beta_r$, we can apply 
Theorem~\ref{thmFreeEnergy} (and a union bound) to say that with high probability over $G \sim \bg_d(n)$, we have
\begin{align}
\label{eqwhoZ1}
     \frac{1}{n} \log Z_{G,\beta+\delta}(0) &< f_{d,\beta+\delta}(0) + \varepsilon \delta/4 \intertext{and}
     \nonumber
   - \frac{1}{n} \log Z_{G,\beta}(0) &< - f_{d,\beta}(0) + \varepsilon \delta/4 \,.   
\end{align}
Combining the inequalities~\eqref{eqderivdef},\eqref{eqFEconvex},\eqref{eqwhoZ1} together yields~\eqref{eqFderv1}. The upper bound for $\E_{{\mu}_{G,\beta}^{\mathsf{anti}}}[H_G(\sigma)]$ is identical. 

    The proof for the lower bound (using convexity to now lower bound $\frac{1}{n} \E_{ \mufix_{G,\beta,0}} [H_G(\sigma)]$ by the left approximation $\frac{\frac{1}{n} \log Z_{G,\beta}(0) - \frac{1}{n} \log Z_{G,\beta-\delta}(0)}{\delta}$) and its analogue for $ \frac{1}{n} \E_{\muanti_{G,\beta}}  [H_G(\sigma)]$ follow symmetrically.
\end{proof}

\begin{proof}[\textbf{\emph{Proof of Corollary~\ref{CorLDrate}}}]
By Lemma~\ref{lem:switching}, since both $\frac{1}{n} \log Z$ and $\frac{1}{n} \log Z(\eta)$ are $C(\beta,d)/n$ Lipschitz, one has for some $c(\beta,d)>0$, 
\begin{align*}
    \mathbb P\left(\left|\frac{1}{n} \log Z_G^{\mathsf{Ising}} - \E\left[\frac{1}{n} \log Z_G^{\mathsf{Ising}}\right]\right| >\eps\right)  \le 2e^{ - c\eps^2 n}\,, \end{align*}
    and similarly, 
    \begin{align*}\max_{\eta}\mathbb P\left( \left|\frac{1}{n} \log \Zfix_G(\eta) - \E\left[\frac{1}{n} \log \Zfix_G(\eta)\right]\right| >\eps\right)\le 2e^{ - c\eps^2 n}\,.
\end{align*}
Thus, it suffices to show that for every $\eta$ fixed, as $n\to\infty$ 
\begin{align*}
    |\mathbb E[\frac{1}{n} \log \Zfix(\eta)] - \mathbb E[\frac{1}{n} \log \Zising]| \to f_{d,\beta}(\eta) - \max_{\eta'\in [-1,1]}f_{d,\beta}(\eta')\,.
\end{align*}
This in turn follows immediately from Theorem~\ref{thmFreeEnergy}. 
\end{proof}

\medskip
\noindent \textbf{\emph{Proof of Theorem~\ref{thmLocalWeak}.}}

To prove Theorem~\ref{thmLocalWeak}, we need two ingredients. The first is to show that events with exponentially small probability in the planted model have exponentially small probability in the configuration model. 
This is a  consequence of the success of the second-moment method and the definition of the planted model, utilized in, e.g.,~\cite{achlioptas2008algorithmic}. 
\begin{lemma}\label{lem:exponential-equivalence}
Fix $d \ge 3$, $\beta< \beta_r$, and $\eta \in [-1,1]$.  Let $(G,\sig)$ be a graph $G\sim \Prob$ drawn from the configuration model and a sample $\sigma\sim \mu_{G,\beta,\eta}$, and let $(\hat G,\hsig)$ be drawn from the planted model $\hat \Prob$.  Then for any event $\mathcal A$, if $\hat \Prob((\hat G,\hsig)\in \cA) \leq \exp(-\Omega(n))$ then we have  $\Prob((G,\sig)\in  \cA) \leq \exp(-\Omega(n))$ .
\end{lemma}
\begin{proof}
   By assumption there exists some $\eps>0$ so that $\hat \Prob (\cA) \le e^{-\eps n}$. Let $X_G = \frac{Z_G}{\E Z_G}$ and let $\mathcal B = \{G:  X_G < e^{-\eps n/2} \}$.  In~\eqref{eqn:expectation-log-Z}, we showed that $\E[\log Z] = \log (\E [Z])+o(n)$ and thus by Lemma~\ref{lem:switching}, we get $\Prob(\mathcal B) \le \exp (-c(\eps) n) $.

By the definition of the planted model, we have 
\begin{align*}
    \hat \Prob ((\hat G,\hsig) \in \cA  ) &= \sum_{(G,\sig) \in \cA} \Prob( G) \frac{Z_G}{\E Z_G} \mu_G(\sig) = \E  [ X \ind{\mathcal A}] \,.
\end{align*}
Using this and $\ind{\mathcal B^c} \le e^{\eps n/2} X$, we can write 
   \begin{align*}
       \Prob( (G,\sig) \in \cA) &\le  \E[ \ind{\cA}  \ind{ \mathcal B^c}]+ \Prob(\mathcal B) \le  e^{\eps n /2 } \hat\Prob (\cA   )  + e^{-c n}\,,
   \end{align*}
   which is at most $e^{-\eps n /2} + e^{-c n}$ as desired. 
\end{proof}

The next essential ingredient is to show that the probability that the empirical distribution of local neighborhoods in the planted model deviates significantly  from that of the tree measure $\nu_\eta$ is exponentially small.  This uses the following new concentration result (Lemma~\ref{planted-concentration}) for Lipschitz functions in the planted model (a weighted configuration model), analogous to \Cref{lem:switching} for the usual configuration model. 
The proof uses local central limit theorems and coupling arguments to correct discrepancies between planted graphs as monochromatic and bichromatic edges get different weights under the planted model. It is  deferred in a self-contained manner to \Cref{secConcentration} and may be of independent interest. 

\begin{lemma}\label{planted-concentration}
Fix any $d\ge 3$, $\beta>0$, and $\eta \in (-1,1)$. Suppose for some $L$ that the function $X$ defined on pairs of graphs and spin configurations $(G,\sigma)$ is such that for each $\sigma \in \Omega_\eta$, if $G,G'$ differ in exactly one switch, then $|X(G,\sigma) - X(G',\sigma)|\le L$.  Let $\hat \Prob_{\hbg(\hsig)}, \hat \E _{\hbg(\hsig)}$ denote the  distribution~\eqref{defn:plantedconfig-graph} of the planted model conditioned on the spin configuration $\hsig$.

Then there exists $C(d,\beta)>0$ such that  for all $\hsig \in \Omega_{\eta}$ and all $t>0$, 
    $$\hat\Prob_{\hbg(\hsig)}(|X - \hat\E_{\hbg(\hsig)} X | \geq t ) \leq 2\exp\Big(-\frac{t^2}{ C n L^2}\Big)\,.$$
\end{lemma}

We can now combine these two ingredients to prove Theorem~\ref{thmLocalWeak}.
\begin{proof}[Proof of Theorem~\ref{thmLocalWeak}]
It is sufficient to show that if $A(\sigma(B_r(v))$ is an event of the configuration on an $r$-neighborhood of a vertex $v\in V$, then for all such $A$, one has that the empirical frequency of $A$ is close to its expectation on the tree under $\nu_{\eta}$: 
\begin{align}\label{eq:local-weak-wts}
    \lim_{n\to\infty}\E\Big[\mu_{G,\eta}\Big(\Big|\frac{1}{n} \sum_{v} \ind{\sigma(B_r(v))\in A} - \nu_\eta(A)\Big|>\eps \Big)\Big] =0\,.
\end{align}
Markov's inequality would imply that with  probability $1-o(1)$, $G\sim \bg$ is such that $$\mu_{G,\eta}\Big(\Big|\frac{1}{n} \sum_{v} \ind{\sigma(B_r(v))\in A} - \nu(A)\Big|>\eps\Big)\to 0\,.$$ 

By Lemma~\ref{lem:exponential-equivalence}, it is sufficient to show that the probability in~\eqref{eq:local-weak-wts} is exponentially small in $n$: namely if
\begin{align*}
    \hat \Prob\Big(\Big|\frac{1}{n} \sum_{v} \ind{\hsig(B_r(v))\in A} - \nu(A)\Big|>\eps \Big) \le e^{ - \Omega(n)}\,.
\end{align*}
and it will necessarily be exponentially unlikely also under $\mathbb E[ \mu_{G,\eta}[\cdot]]$

Note that the law of $X(\hat \sigma,\hat G)= \frac{1}{n} \sum_{v\in V} \ind{\hsig(B_r(v)) \in A}$ is invariant under choice of $\hat \sigma \in \Omega_\eta$. 
By fixing $\hsig \in \Omega_{\eta}$, drawing $\hat G \sim \hbg(\hsig)$, and applying Lemma~\ref{planted-concentration}  to $X$,  which is $C_r/n$-Lipschitz with respect to switches of $\hat G$, it concentrates exponentially about $\hat \Prob(\hsig(B_r(v))\in A)$. By Lemma~\ref{lem:local-dtv-planted-to-tree}, that probability is within $O(n^{-1/2+\delta})$ of its tree probability $\nu(A)$. Therefore, for any fixed $\eps>0$, for $n$ large, the error in expectations is at most $\eps/2$ and the exponential concentration then gives the desired $e^{ -\Omega(n)}$ upper tail bound. 
\end{proof}

\section{Sub-exponential mixing of dynamics from symmetric initializations}
\label{secDynamics}

Our main aim in this section is to show that Glauber dynamics restricted to $\Omega^+ = \{\sigma: \sum_v \sigma_v\ge 0\}$, meaning rejecting any update that takes it out of $\Omega^+$, mixes in sub-exponential time when $\beta \in (\beta_c,\frac{1}{4\sqrt{d-1}})$.  

\begin{theorem}\label{thm:positive-magnetization-mixing-time}
    Fix $d\ge 3$ and $\beta\in (\beta_c,\frac{1}{4\sqrt{d-1}})$ and consider Ising Glauber dynamics restricted to $\Omega^+$ on $G\sim \bg_d(n)$. With probability $1-o(1)$, $G$ is such that the mixing time is $e^{o(n)}$. 
\end{theorem}

After proving Theorem~\ref{thm:positive-magnetization-mixing-time}, we will then use the spin-flip symmetry of the Ising model, its dynamics, and the uniform-at-random initialization to deduce Theorem~\ref{thmIsingMixing}. 

We prove Theorem~\ref{thm:positive-magnetization-mixing-time} by decomposing the Glauber dynamics into two Markov chains using the projection--restriction tools introduced in \cite{madras2002markov}; one will essentially move within one slice of configurations of a given magnetization using Kawasaki dynamics, for which fast mixing up to $1/(4\sqrt{d-1})$ was recently shown in~\cite{bauerschmidt2023kawasaki} (though a bit more care is needed to do the comparison arguments), and the other will be a one-dimensional chain moving on the free energy landscape given by the partition functions of the fixed-magnetization Ising model for different magnetizations $\eta n \in \{0,...,n\}$. 

This section will be predominantly concerned with the full (not fixed-magnetization) Ising model. To avoid writing $\mu^{\mathsf{Ising}}$ everywhere, we now switch to the common Markov chain notation of $\pi = \mu^{\mathsf{Ising}}$ for stationary distributions. 

\subsection{From Kawasaki to Glauber-Kawasaki}

Let $P^+$ denote the Glauber dynamics transition matrix restricted to $\Omega^+$. Define a new hybrid chain $P_{\GK}^+$ on $\Omega^+$ where at each step, with probability $\frac12$ the chain performs a step of the Glauber dynamics (restricted to $\Omega^+$) and with probability $\frac12$, the chain performs a step of the Kawasaki dynamics.

\begin{definition}\label{def:Kawasaki-dynamics}
    The Kawasaki dynamics on $\Omega_\eta$ (corresponding to $k=\frac{1+\eta}{2} n$ pluses and $n-k$ minuses) is the discrete-time Markov chain which at each time picks a uniform at random pair of vertices $x,y\in V$ such that $\sigma_x \ne \sigma_y$ and swaps the spins on those two sites with one another to get the new configuration $\sigma^{x\leftrightarrow y}$ with probability $\pi(\sigma^{x\leftrightarrow y})/(\pi(\sigma) + \pi(\sigma^{x\leftrightarrow y}))$. 
\end{definition}

It is easy to verify that for each fixed $k$, the Kawasaki dynamics is reversible with respect to $\pi_k:=\pi(\cdot \mid \Omega_k)$ where we recall that $\Omega_k = \binom{[n]}{k}$ is the set of configurations with $k$ plus spins. Therefore, the Glauber--Kawasaki dynamics $P^+_{\GK}$ is reversible with respect to the full positive-magnetization measure $\pi^+:=\pi (\cdot \mid \Omega^+)$. 

In what follows, for an ergodic reversible Markov chain with transition matrix $P$, we use $\gap(P)$ to denote its (absolute) spectral gap. 

\begin{proposition}\label{prop:augmented-chain}
For any $\beta$ and any graph, the hybrid Glauber-Kawasaki chain satisfies
    $$\gap(P^+) \ge (3n e^{\beta d})^{-1} \gap(P_{\GK}^+) \,.$$
\end{proposition}

    We will prove Proposition~\ref{prop:augmented-chain} by comparison via canonical paths. Recall the setup: let $P$ and $\tP$ be reversible transition matrices with stationary distributions $\pi$ and $\tilde{\pi}$, respectively. Let $E$ denote the edges $e =xy\in \Omega^2$ with $P(x,y)>0$, and define $\tilde E$ analagously. Supposing that for each $uv \in \tilde E$, there is an $E$-path from $u$ to $v$, choose one and denote it by $\Gamma_{uv}$, and let $B$ denote the congestion ratio of that set of paths $(\Gamma_{uv})_{u,v}$:
$$B(\Gamma) := \max_{xy \in E} \Big(\frac{1}{\pi(x)P(x,y)} \sum_{u,v : \Gamma_{uv} \ni xy}\tilde{\pi}(u)\tP(u,v)|\Gamma_{uv}|\Big)\,.$$

\begin{lemma}[\hspace{1sp}{\cite[Theorem 13.20]{levin2017markov}}]\label{lem:canonical-paths}
     Suppose for each edge $uv \in \tilde{E}$, there is a canonical path $\Gamma_{uv}$ from $u$ to $v$ via edges of $E$. Then $\gap(\tP) \leq \max_{x \in \Omega} \frac{\pi(x)}{\tilde{\pi}(x)} \cdot B(\Gamma) \cdot \gap(P)$. 
\end{lemma}
    
\begin{proof}[Proof of \Cref{prop:augmented-chain}]
    We will apply \Cref{lem:canonical-paths} with $P^+$ playing the role of $P$ and $P_{\GK}^+$ playing the role of $\tP$. Note that they have the same stationary distribution $\pi^+$. 
    
    For each step in the hybrid chain, we can simulate it using at most two Glauber steps: if $P_{\GK}^+$ performs a Glauber transition, then $P^+$ does the same, and if $P_{\GK}^+$ performs a Kawasaki transition---say swapping a $+1$ spin at $u$ with a $-1$ spin at $v$---then $P^+$ performs two Glauber swaps, first at $u$ and then at $v$. Thus, $|\Gamma_{\sigma \tau}| \leq 2$ for all $\sigma, \tau \in \Omega$ such that $\sigma\tau$ is an $P_{\GK}^+$ edge.

    To get an upper bound on $B$, we consider the summation in two parts. Let $\sigma, \tau$ be adjacent via a Glauber step and $\Gamma_{\sigma'\tau'} \ni \sigma\tau$. If $\sigma'\tau'$ are also separated by a Glauber transition, then we must have $\sigma = \sigma', \tau = \tau'$, and  $\frac{\pi^+(\sigma') P_{\GK}^+(\sigma', \tau')}{\pi^+(\sigma)P^+(\sigma, \tau)} = \frac12$. 
    If $\sigma'\tau'$ are adjacent via a Kawasaki step, we simply bound $P_{\GK}^+(\sigma', \tau') \leq \frac{1}{2k(n-k)}$ and $P^+(\sigma, \tau) \geq \frac{1}{n}\frac{1}{2e^{\beta d}}$. Lastly, there are at most $n-1$ choices of such $\sigma'\tau'$, so we have
    $$B \leq \frac12 + n \cdot (n-1)\frac{2e^{\beta d}}{n-1} \leq 3n e^{\beta d}\ .$$
    This gives the claimed bound via Lemma~\ref{lem:canonical-paths}. 
\end{proof}

\subsection{Projection-restriction decomposition of Markov chains}

We now restrict our attention to $P_{\GK}^+$, and apply the projection--restriction decomposition of Markov chains from~\cite{madras2002markov}, which we now recall. 
Given a Markov chain $P$ on state space $\Omega$ with stationary distribution $\pi$, suppose we may cover $\Omega$ by (not necessarily disjoint) slices $A_1 \cup A_2 \cup \dots \cup A_m$. Let $P_i$ denote the Markov chain restricted to slice $A_i$: that is, for $x, y \in A_i, x \neq y$, set $P_i(x,y) = P(x,y)$, and set $P_i(x,x) = 1-\sum_{y \in A_i, y\neq x} P_i(x,y)$. 

Define the projection Markov chain $P_H$ with state space $\{1, \dots, m\}$ as follows. Let $\Theta = \max_{x \in \Omega}|\{i : x \in A_i\}|$. For $i \neq j$, set $P_H(i,j) = \frac{\pi(A_i \cap A_j)}{\Theta \pi(A_i)}$, and set $P_H(i,i) = 1 - \sum_{j\neq i} P_H(i,j)$.

\begin{lemma}[\hspace{1sp}{\cite{madras2002markov}}]\label{lem:main-projection-restriction}
In the above projection--restriction context, one has 
    $$\gap(P) \geq \frac{1}{\Theta^2}\gap(P_H)\cdot \min_{1\le i\le m} \gap(P_i)\,.$$
\end{lemma}

In order to show Theorem~\ref{thm:positive-magnetization-mixing-time} that the Glauber dynamics restricted to $\Omega^+$ are fast mixing, we consider a decomposition $$A_i = \Omega_i \cup \Omega_{i+1}\,, \qquad \text{where} \qquad \Omega_i = \{\sigma: |\{v:\sigma_v = +1\}| = i\}\,,$$ for $\frac{n}{2} \leq i \leq n$. The transition probabilities for the projection chain are then defined to be 
\begin{align*}
P_H(k, k+1) = \frac{1}{2} \frac{z_k}{z_k + z_{k+1}}\,, \qquad \text{where} \qquad z_k = \sum_{\sigma \in \Omega_k} e^{\beta H(\sigma)}
\end{align*}
is the partition function associated to configurations with $k$ pluses. 

Combining the results of~\cite{bauerschmidt2023kawasaki} on the Kawasaki dynamics that the restricted chains $P_i$ are each individually fast mixing with the characterization of the local weak convergence of the fixed-magnetization Ising model for all magnetizations, we show (approximate) unimodality of the sequence of $(z_k)_{k\in [n/2,n]}$ and from that deduce sub-exponential mixing of the one-dimensional projection chain $P_H$. 

\subsection{Fast mixing within a slice}
We begin with deducing the fast mixing for the restriction chain $P_i$ for all $i\in [0,n]$. The result of~\cite{bauerschmidt2023kawasaki} showed that with high probability $G \sim \bg_d(n)$ is such that for all $\beta<\frac{1}{4\sqrt{d-1}}$ and all $i \in [0,n]$ the mixing time of Kawasaki dynamics on $\Omega_i$ is $O(n \log n)$. For this step, we use a slightly different form of projection-restriction coming from~\cite{JeSoTeVi04}, where the cover is a partition (the slices are disjoint). Because our end result is not a polynomial bound, we are not careful about optimizing polynomial losses in these steps. 

\begin{lemma}\label{lem:fast-mixing-of-slice-chain}
    Fix $\beta \in [0, \frac{1}{4\sqrt{d-1}})$. There exists a constant $C(\beta,d)>0$ such that for all $i \in [0,n]$, the Glauber--Kawasaki on $A_i = \Omega_i \cup \Omega_{i+1}$ has spectral gap $\gap(P_i) \ge 1/(Cn^2)$. 
\end{lemma}
\begin{proof}
For any $i$, we split $A_i$ into the disjoint union $\Omega_i \sqcup \Omega_{i+1}$. We can write $P_{i,-}$ as the restriction of $P_i$ to $\Omega_i$ (which is exactly the Kawasaki dynamics on $\Omega_i$), $P_{i,+}$ as the restriction of $P_i$ to $\Omega_{i+1}$, and $P_{i,H}$ as the two-state Markov chain on $\{i,i+1\}$ with transition probabilities given by 
\begin{align*}
    P_{i,H}(i,i+1) = \frac{1}{\pi(\Omega_i)}\sum_{\sigma\in \Omega_i, \tau\in \Omega_{i+1}} \pi(\sigma) P^+(\sigma,\tau)
    \end{align*}
    and \begin{align*}
        P_{i,H}(i+1,i) = \frac{1}{\pi(\Omega_{i+1})}\sum_{\sigma\in \Omega_{i+1}, \tau\in \Omega_{i}} \pi(\sigma) P^+(\sigma,\tau)\,.
\end{align*}
Then Theorem~1 of~\cite{JeSoTeVi04} says that 
\begin{align*}
    \gap(P_i) \ge \min\Big\{  \frac{\gap(P_{i,H})}{3}, \frac{\gap(P_{i,H})\min\{P_{i,-},P_{i,+}\}}{3 +  \gap(P_{i,H})} \Big\}\,.
\end{align*}
The spectral gaps of $P_{i,-}, P_{i,+}$ are exactly the spectral gaps of Kawasaki dynamics on the slice $\Omega_i$ or $\Omega_{i+1}$ respectively. Theorem 1.1 of~\cite{bauerschmidt2023kawasaki} showed that there is a constant $C(\beta)$ independent of $i$ such that with high probability, $G\sim \bg_d(n)$ is such that the spectral gap of Kawasaki dynamics on $G$ with number of $+$ spins equal to $i\in [0,n]$ is at least $C n^{-1}$. 

To bound the spectral gap of the two-state projection chain $P_{i,H}$ we lower-bound the transition rates $P_{i,H}(i,i+1)$ and $P_{i,H}(i+1,i)$ because the spectral gap is given by their sum. Note that $P^+(\sigma,\tau)$ is bounded below by $\frac{1}{n} e^{-2\beta d}$, so 
\begin{align*}
    P_{i,H}(i,i+1) \ge \min P^+(\sigma,\tau) \ge \frac{1}{n} e^{ - 2\beta d}\,.
\end{align*} 
Putting the above together,we get the claimed bound. 
\end{proof}

\subsection{The magnetization profile}
The main work is therefore to show that the projection chain, which is a birth-death process with transition probabilities governed by the partition functions associated to each magnetization, mixes in sub-exponential time. Each of these partition functions is exponentially large, but we will write transition rates of the projection chain as certain empirical averages under the Gibbs measure, which we will be able to identify (up to smaller order fluctuations) using the local weak convergence result of Theorem~\ref{thmLocalWeak}.

\begin{lemma}\label{lem:ratio-of-partition-functions-expectation}
    We have that for every $\eps>0$, there exists $c>0$ such that for all $k\in \{n/2,...,n\}$, 
    \begin{align*}        
    \mathbb P\Big( \Big|\frac{z_{k+1}}{z_k} - \frac{n}{k}\E_\nu\Big[\ind{\sigma_v=-1} e^{\beta \sum_{w\sim v}\sigma_w} \Big]\Big| >\eps \Big) \le e^{ - cn}\,.
    \end{align*}
\end{lemma}
\begin{proof}
We begin by writing for every $G$ and any $k$,  
    \begin{align*}\frac{z_{k+1}}{z_k}  & = \frac{1}{z_k}\sum_{\sig' \in \Omega_{k+1}}e^{\beta H(\sig')} = \frac{1}{z_k} \frac{1}{k+1} \sum_{\sig \in \Omega_k} e^{\beta H(\sig)}\sum_{v : \sigma_v = -1} e^{\beta \sum_{w\sim v} \sigma_w} \\
& = \frac{1}{k+1} \E_{\sig \sim \pi_k}\Big[ \sum_v \ind{\sigma_v =-1} e^{\beta \sum_{w\sim v} \sigma_w}\Big] \\ 
& = \frac{n}{k+1} \E_{\sigma \sim \pi_k} \Big[\frac{1}{n} \sum_v \ind{\sigma_v =-1} e^{\beta \sum_{w\sim v} \sigma_w}\Big]\,.
\end{align*}
This is evidently a $C(\beta,d)/n$-Lipschitz function of the switching operation on $G$, so by Lemma~\ref{lem:switching}, it satisfies 
\begin{align*}
    \mathbb P\left(\left|\frac{z_{k+1}}{z_k} - \mathbb E\left[ \frac{z_{k+1}}{z_k}\right]\right| >\frac{\eps}{2}\right) \le e^{ - \Omega(n)}\,.
\end{align*}
On the other hand, because the expectation here can be written as a double expectation 
\begin{align*}
    \E_{\bg}\left[\E_{\pi_k}\left[ \frac{n}{k} \frac{1}{n} \sum_v\ind{\sigma_v =-1} e^{\beta \sum_{w\sim v} \sigma_w}\right]\right]
\end{align*}by the local weak convergence of Theorem~\ref{thmLocalWeak}, the expectation is within $o(1)$ of the expectation under the limiting tree measure $\nu_\eta$ where $k = \frac{1+\eta}{2} n$. In particular, for $n$ large, one will have the deterministic statement that for all $k$, 
\begin{align*}
    \left|\mathbb E\left[  \frac{z_{k+1}}{z_k}\right]- \frac{n}{k}\mathbb E_\nu \left[ \ind{\sigma_v =-1} e^{ \beta \sum_{w\sim v} \sigma_w}\right] \right| <\frac{\eps}{2}\,,
\end{align*}
where $v$ is the root of the $d$-regular tree. Combining the above two displays gives the claim. 
\end{proof}

\begin{lemma}\label{lem:properties-of-ratio-of-part-functions}
    If $\beta<\beta_r(d)$, the expectation $\frac{n}{k}\mathbb E_\nu [\ind{\sigma_v =-1} e^{ \beta \sum_{w\sim v} \sigma_w}]$ is within $o(1)$ of 
    \begin{align}\label{eq:F(eta)}
        F_{d,\beta}(\eta) = \frac{1-\eta}{1+\eta} \Big( \frac{\rho_{\eta}-\eta}{1-\eta} e^{-\beta} + \frac{(1-\rho_{\eta})}{1-\eta} e^{\beta}\Big)^d\,,
    \end{align} 
    where $\rho_\eta$ is from~\eqref{eq:rhoeta-explicit}. 
    Moreover, for all $\beta \in (\beta_c, \beta_r)$, this function $F$ satisfies the following properties: 
    \begin{itemize}
        \item $F$ is $1$ at exactly two points, one at $0$ and one at $m_*(d,\beta) \in (0,1)$,
        \item $F$ is increasing on $(0,m_*)$ and decreasing on $(m_*,1)$,
        \item For every $\delta>0$, $F', F''$ are uniformly bounded on $[0,1-\delta]$, and $\lim_{x\to 1} F(x)= 0$. 
    \end{itemize}
\end{lemma}    
\begin{proof}
We consider the expectation under $\nu_\eta$ with $\frac{1+\eta}{2} n = k$, 
\begin{align*}
    \frac{n}{k} \E_\nu \Big[\ind{\sigma_v =-1} e^{ \beta \sum_{w\sim v} \sigma_w} \Big] = \frac{2}{1+\eta} \frac{1-\eta}{2} \E_{\nu}\Big[ e^{\beta \sum_{w\sim v} \sigma_w} \mid \sigma_v =-1\Big]\,.
\end{align*}
Now notice that by the total-variation bound of Lemma~\ref{lem:local-dtv-planted-to-tree}, this expectation can be approximated, up to $n^{-1/2+\delta}$ error, by its expectation under the planted model $\hat{\mathbb E}$ where it is easier to compute explicitly. 

Following the proof of Lemma~\ref{lem:local-dtv-planted-to-tree}, conditional on $\sigma_w$ being $-1$, its neighboring matched half-edges are chosen uniformly among all available half-edges with spin $-1$ on their other endpoint, which implies in particular that their joint law is within $o(1)$ in total variation distance to $d$ independent Rademacher random variables with parameter $(1-\rho_{\eta})/(1-\eta)$ for $\rho_\eta$ from~\eqref{eq:rhoeta-explicit}. Using the moment generating function of the Rademacher random variables, we get that the expectation under consideration is within $o(1)$ of 
\begin{align*}
    F_{d,\beta}(\eta) = \frac{1-\eta}{1+\eta} \Big( \frac{\rho_{\eta}-\eta}{1-\eta} e^{-\beta} + \frac{(1-\rho_{\eta})}{1-\eta} e^{\beta}\Big)^d\,.
\end{align*}

Observe that since $\rho_{\eta=0} = \frac{e^{\beta} - 1}{e^{\beta}-e^{-\beta}}$, one has $F(0)=1$ as claimed. Next, computing the derivative of $\rho_\eta$ we find that 
\begin{align*}
    \frac{d\rho_{\eta}}{d\eta} = \frac{\eta}{\sqrt{e^{2\beta}(1 - \eta^2) + \eta^2}}\,.
\end{align*}
Since $\frac{d}{d\eta}\rho_\eta \vert_{\eta=0}=0$, one calculates $$F'(0) = -2 + d(1-e^{-\beta})$$ which is straightforwardly seen to undergo a phase transition between being negative to positive exactly at $- \log (1-\frac{2}{d}) =\beta_c$, as expected. At the same time, $\lim_{x\to 1} F(x) =0$ is evident from L'Hopital's rule. Therefore, the intermediate value theorem implies that when $\beta \in (\beta_c,\beta_r)$ there must be (at least) another point $m_*(\beta)\in (0,1)$ at which $F$ equals $1$. For the first two bullet points, it remains to show that there is no other point where $F$ equals $1$. 

We will show that $F'(\eta) =0$ has at most one solution by showing that $\phi(\eta) = \log F$ has only one critical point (sufficient because $\phi' = F'/F$ and $F$ is bounded positive on $(0,1)$. If we let 
\begin{align*}
    S_\eta = \sqrt{e^{2\beta}(1-\eta^2) + \eta^2} \qquad \text{so that} \qquad \rho_\eta = \frac{e^{2\beta} - S(\eta)}{e^{2\beta} -1}\,,
\end{align*}
then 
\begin{align*}
    \phi(\eta) = - \beta d + d \log (S_\eta - \eta) + (1-d)\log (1-\eta) - \log(1+\eta)\,.
\end{align*}
Computing $S'_\eta= -\eta(e^{2\beta} -1)/S_\eta>0$, we differentiate to get 
\begin{align*}
    \phi'(\eta) 
    = d\frac{ \frac{-\eta(e^{2\beta}-1)}{S_\eta}-1}{S_\eta - \eta}+ \frac{d-1}{1-\eta} - \frac{1}{1+\eta}
\end{align*}
For $\eta \in (0,1)$ we can multiply through by $(1-\eta)(1+\eta)(S_\eta - \eta) S_\eta$ to get equivalently 
\begin{align*}
    0 
    & = - d(e^{2\beta}-1)\eta (1-\eta)(1+\eta)- dS_\eta (1-\eta)(1+\eta) + (d-1) (S_\eta - \eta) (1+\eta) S_\eta \\ 
    &  \qquad - (S_\eta - \eta) (1-\eta)S_\eta \,.
\end{align*}
We can thus write a pair of polynomial equations satisfied by $S_\eta$ (in terms of $\eta$):  
\begin{align*}
    0 & =S^2 - e^{2\beta} + (e^{2\beta}-1) \eta^2 \,, \\ 
    0& = (-d \eta  - d +2) S_\eta^2 + (d\eta + d -2\eta) S_\eta + d\eta(1-\eta)(1+\eta)(e^{2\beta}-1)\,.
\end{align*}
To determine which values of $\eta$ can give solutions, we use the polynomial resultant of the pair of equations, which says that any $\eta$ solving the above must satisfy 
\begin{align*}
    \text{Res}_S(\eta) & = e^{2\beta} (\eta-1)(1-\eta) P(\eta) =0 \qquad \text{where}  \\ 
    P(\eta) &  =  \eta^2 ( d^2 (e^{2\beta}-1)-4d (e^{2\beta}-1)+4(e^{2\beta}-1))- d^2 (e^{2\beta} -1)+ 4 (d -1)e^{2\beta}\,. 
\end{align*}
Noting that the possible zeros are at $\eta \in \{\pm 1\}$ and at roots of $P(\eta)$, and that $P(\eta)$ is a linear expression in $\eta^2$ and thus only has at most one solution for $\eta \in (0,1)$, we get the claim. Indeed, with some further algebra, one verifies that $P(\eta)$ has a root in $\eta \in (0,1)$ if and only if $\beta>\beta_c$ as expected.   

Finally, the third bullet point is established by differentiating $F$ by hand, and noticing that its derivatives are smooth on $(0,1)$ and the only places any term can blow up is at $\eta =\pm 1$. 
\end{proof}

Putting Lemmas~\ref{lem:ratio-of-partition-functions-expectation}--\ref{lem:properties-of-ratio-of-part-functions} together with a union bound over $k$, and recalling the projection chain's transition probabilities, we arrive at the following approximation for the one-dimensional projection chain's transition rates. In words, it says the projection chain behaves like a random walk in an approximately unimodal (on $k=\{n/2,...,n\}$, bimodal on all of $\{0,...,n\}$) energy landscape.  

\begin{corollary}\label{cor:all-1d-transition-probability-properties}
    Fix $\beta \in (\beta_c,\beta_r)$. There exists a sequence $r_n = o(1)$ such that with high probability over $G\sim \bg$, the transition probabilities $P_{H}(i,i+1)$ satisfy the following for all $i\in \{n/2,...,n\}$: 
    \begin{enumerate}
        \item $|P_H(i,i+1) - \frac{1}{2} \frac{1}{1+F(\frac{2i}{n}-1)^{-1}}| \le  r_n$,
        \item $|P_H(i,i-1) - \frac{1}{2} \frac{1}{1+F(\frac{2i}{n}-1)}| \le r_n$,
        \item $P_H(i,i) = 1- P_H(i,i+1) - P_H(i,i-1)$,
    \end{enumerate}
    for $F$ given by~\eqref{eq:F(eta)} satisfying the three properties outlined in Lemma~\ref{lem:properties-of-ratio-of-part-functions}. 
\end{corollary}

\subsection{Mixing time of the 1D projection chain} 
In this subsection, we show that any one-dimensional birth-and-death chain whose transition probabilities satisfy the set of properties outlined in Corollary~\ref{cor:all-1d-transition-probability-properties} will have a mixing time that is at most $\exp( O(r_n \cdot n))$. 

\begin{remark}
Note that if we had access to optimal concentration estimates on $z_k/z_{k+1}$ as we expect (i.e., those we have for the planted model, without having to pass to the unplanted model), one would take $r_n  = n^{-1/2}\sqrt{\log n}$ and with use of slightly refined estimates in the proof below, improve the mixing time bound to $\exp( O(r_n^2 \cdot n))$ to get polynomial mixing. However, as mentioned in the introduction, getting such a precise error for the deviation of the unplanted expectation of $z_k/z_{k+1}$ from the properties outlined in Lemma~\ref{lem:properties-of-ratio-of-part-functions} seems to pose a serious challenge. 
\end{remark}

\begin{lemma}\label{lem:conditions-guaranteeing-fast-birth-death-mixing}
    Consider any 1D birth-and-death chain on $\{0,...,n\}$ with transition probabilities $p_{i,i-1},p_{i,i+1}$ and $p_{ii} = 1-p_{i,i-1} - p_{i,i+1}$. Suppose there exist $m_*\in (0,1)$, $\delta>0$ small, sequence $r_n = o(1)$ with $r_n \gtrsim n^{-1/2}\sqrt{\log n}$, and constants $(c_i)_{0\le i\le 5}>0$ such that   
    \begin{enumerate}
        \item For $i\in [0,\delta n]$, one has $p_{i,i+1} - p_{i,i-1} \ge c_0 (i/n) - r_n$,
        \item For $i\in [\delta n, m_*n - \delta n]$, one has  $p_{i,i+1} - p_{i,i-1} \ge c_1$,
        \item For $i\in [m_* - \delta n, m_*n + \delta n]$ one has $$p_{i,i+1} - p_{i,i-1} \in  [c_2(m_*-i/n)- r_n, c_3(m_* -i/n) - r_n]\,,$$
        \item For $i \in [m_*n+ \delta n, n]$, one has $p_{i,i+1} - p_{i,i-1}\le - c_4$. 
        \item For all $i\in [0,n]$, one has $p_{ii}\in [c_5, 1-c_5]$. 
    \end{enumerate}
    Then there exists $C$ such that $\tmix\le C e^{Cr_n \cdot  n}$. 
\end{lemma}

 \begin{proof}
     All birth-death chains are monotone so it is sufficient to couple the chains initialized from $X_0 =0$ and from $X_0 = n$. Call the former chain $X_t^-$ and the latter $X_t^+$. 
     
     The argument will proceed in three stages. 
     \begin{itemize}
         \item Stage 1: $X_t^-$ reaches $\delta n$ in $C_1e^{ C_1 r_n n} + C_1 n \log n$ steps; 
         \item Stage 2: $X_t^-$ and $X_t^+$ reach $[m_* n-\tfrac{\delta}{2} n, m_* n + \tfrac{\delta}{2} n]$ in a further $C_2 n$ steps; 
         \item Stage 3: $X_t^-$ and $X_t^+$ couple in a further $C_3e^{ C_3 r_n n} + C_3 n \log n$ steps. 
     \end{itemize}
     
     \medskip
     \noindent \emph{Stage 1}: We first argue that with high probability, $X_t^-$ is able to escape the (approximately) repulsive fixed point of the dynamics at $0$. We split the analysis in two, first bounding the hitting time $\tau_1$ of $X_t^-$ to $C_0  r_n n$ for $C_0$ large (as in this region one isn't even guaranteed that $X_t^-$ has a positive drift). Notice that by forcing the chain to move to the right in every step that it moves, in every $C_0 r_n n $ steps, there is a fresh attempt of probability $(c_5(\frac{1}{2} - r_n))^{C_0 r_n n}$ of hitting $C_0 r_n n$, giving a high probability upper bound on $\tau_1$ of $C_1 e^{ C_1 r_n n}$ for an appropriately chosen $C_1$.

    Next define $\tau_2$ to be the hitting time for $X_t^-$ of $\delta n$. 
    By the strong Markov property, conditioning on $\mathcal F_{\tau_1}$, it suffices to consider $X_t^-$ initialized from $X_0^- = C_0 r_n n$. As long as $C_0$ is a sufficiently large constant, by Assumption (1) again, that Markov chain has a drift which is lower bounded by $c_0' i/n$ for all $i\in [\frac{C_0}{2} r_n n ,\delta n]$ where $c_0'$ is used as opposed to $c_0$ to absorb the $c_5$ constant for moving, and to compensate for the error $r_n$ in Assumption (1). In this region, we have 
    \begin{align*}
        X_t^- \ge  C_0 r_n n + \sum_{s=1}^t \frac{c_0' X_s^-}{n} + M_t\,,
    \end{align*}
    where $M_t$ is the Doob martingale, and has increments bounded by $1$. By the Azuma--Hoeffding inequality, one has 
    $$\mathbb P\Big(\sup_{t\le T} |M_t| \ge \frac{C_0}{2} r_n n \Big) \le \exp( - \Omega( r_n^2 n^2/T))\,.$$ 
    As long as $T \le C_1 n \log n$ where $C_1$ can be taken large as long as $C_0$ initially was large enough, this probability is at most $1/100$. On that event, in time $T = C_1 n \log n$, the process does not exit the interval $[\frac{C_0}{2} r_n n, \delta n]$ to the left, and by the discrete Gronwall inequality, one has 
    \begin{align*}
        X_t^- \ge \frac{C_0}{2} r_n n \cdot e^{c_0' t/n}\,\qquad \text{for all $t\le T$.}
    \end{align*}
    In particular, as long as $C_1$ was sufficiently large (depending on $c_0'$), this process hits $\delta n$, upper-bounding $\tau_2 - \tau_1$ by $C_1 \log n$, and thus taking the larger of the $C_1$'s, we get with probability $99/100 -o(1)$ that $\tau_2 \le C_1 e^{C_1 r_n n} + C_1 n \log n$. 

    \medskip
\noindent     \emph{Stage 2}: This portion of the analysis is the easiest as there is a bounded-away-from-zero drift towards a $\delta n$ neighborhood of $m_* n$. Again by the strong Markov property, suppose $X_t^-$ is initialized from $\delta n$. We explain the argument for $X_t^-$ reaching $m_* n - \delta n$, with the bound on the hitting time of $X_t^+$ reaching $m_* n + \delta n$ being symmetrical, with Assumption (4) playing the role of Assumption (2). As long as $X_t^-$  is at least $\delta n/2$, by Assumptions (1)--(2),  its drift is at least $c_1'$, and therefore, 
    \begin{align*}
        X_t^- \ge \delta n + tc_1'  + M_t\,,
    \end{align*}
    where $M_t$ is the Doob martingale. Again using Azuma's inequality, this time with $T = C_2 n$, one gets 
$
        \mathbb P(\sup_{t\le T}|M_t|\ge \tfrac{\delta }{4} n ) \le \exp ( - \Omega(n))$. 
    On the complement of that event, 
    we get that after $T = C_2 n$ steps, $X_t^-$ has hit $m_* n  - \delta n$. 
    Moreover, by the same reasoning and a union bound, once $X_t^-$ hits $m_* n - \delta n$, it stays at least $m_* n - 2\delta n$ for $\exp(\Omega(n))$ many steps except with probability $\exp(- \Omega(n))$. 
    
    \medskip
    \noindent \emph{Stage 3}: In the last section of the coupling, we can assume that $X_t^- = m_* n - \delta n$ and $X_t^+ = m_* n + \delta n$. We first establish that after $C_3 n \log n$ steps, $X_t^-$ attains $m_* n - C_0' r_n n$ and $X_t^+$ attains $m_* n  + C_0' r_n n$ except with probability $\frac{2}{100}$. This argument goes symmetrically to the second part of stage 1 where $X_t^-$ went from $C_0 r_n n$ to $\delta n$ with Assumption~(3) replacing Assumption (1) (together with a union bound). 
    
    At this point, we have established that if we let $T_0 = C_1 e^{C_1 r_n n} + C_1 n\log n + C_2 n$, one has with probability $1 - \frac{3}{100} - o(1)$,  
    $$X_{T_0}^- \ge m_* n - C_0' r_n n\,, \qquad \text{and} \qquad X_{T_0}^+ \le m_* n + C_0' r_n n\,.$$
    By monotonicity, it suffices to consider two chains $X_{T_0+t}^-$ and $X_{T_0+t}^+$ started from these two extreme points, and establish that there is at least probability $e^{- C_3 r_n n}$ of them being equal after a further $n$ steps. Indeed, one can force that in the next $C_0 ' r_n n$ many updates,$X_{T_0+t}^-$ always moves to the right, and $X_{T_0}^+$ always moves to the left, in which case by the end of that period they will have coincided, and this forcing costs $(c_5(\frac{1}{2} - r_n))^{2 C_0' r_n n}$. 
 \end{proof}

 The next corollary applies Lemma~\ref{lem:conditions-guaranteeing-fast-birth-death-mixing} to the projection chain. 

 \begin{corollary}\label{cor:projection-chain-subexponential-mixing}
     For every $d\ge 3$ and $\beta \in (\beta_c, \beta_r)$, there exists $C(\beta,d)$ such that with probability $1-o(1)$, $G\sim \bg$ is such that the mixing time and inverse spectral gap of the projection chain $P_H$ are $e^{o(n)}$. 
 \end{corollary}
 \begin{proof}
     We show that Corollary~\ref{cor:all-1d-transition-probability-properties} guarantees the properties of Lemma~\ref{lem:conditions-guaranteeing-fast-birth-death-mixing} (up to identification of $\{n/2,...,n\}$ with $\{0,...,n\}$) for some sequence $r_n = o(1)$. Note that at the points where $F(\eta)=1$, one has zero drift, and we can Taylor expand 
     \begin{align*}
         P_H(i,i+1) - P_H(i,i-1) \ge \frac{1}{2} \frac{F(\frac{2i}{n} -1)-1}{F(\frac{2i}{n}-1)+1}- r_n \,.
     \end{align*}
     By Taylor expanding $F$ about $0$ (corresponding to $i/n= 1/2$), 
     \begin{align*}
        | F(2i/n-1) - 1 + F'(0)(2i/n -1) | \le C (2i/n -1)^2\,,
     \end{align*}
     for all $2i/n-1 \le 1-\delta$ by the uniform upper bound on $F''$. Plugging this into $(x-1)/(x+1)$ and using that $F'(0)>0$ per Corollary~\ref{cor:all-1d-transition-probability-properties}, we get 
     \begin{align*}
         P_H(i,i+1) - P_H(i,i-1) \ge \frac{1}{2} F'(0) (\tfrac{2i}{n} -1) + O((\tfrac{2i}{n}-1)^2) -r_n
     \end{align*}
     which is at least $c_0 (2i/n-1) -r_n$ for $c_0 = F'(0)/2$ as long as $2i /n -1 <\delta$ for suitable~$\delta$.

     For Assumption (2), by continuity, there exists $\eps_\delta$ such that for all $i \in [\delta n, m_* n - \delta n]$, one has  $F\ge 1+\eps_\delta$. Thus 
     \begin{align*}
         \min_{i\in [\delta n,m_* n -\delta n]} P_H(i,i+1) - P_H(i,i-1) &  \ge \min_i \frac{1}{2} \frac{F(\frac{2i}{n} -1)-1}{F(\frac{2i}{n}-1)+1} - o(1)\\ & \ge \frac{1}{2} \frac{ \eps_\delta}{\eps_\delta +1}-o(1)\,,
     \end{align*}
     because $(x-1)/(x+1)$ is increasing for $x>1$. The right-hand side is strictly positive. 

     The proof that Assumption (3) holds is analagous to the proof for Assumption (1), with the Taylor expansion now being about $m_*$. The proof that Assumption (4) holds is analagous to the proof of Assumption (2). 
     
     Finally, to see that Assumption (5) holds, consider  
     \begin{align*}
         |P_H(i,i+1) + P_H(i,i-1) - \frac{1}{2}| \le  r_n\,.
     \end{align*}
     It follows that for $n$ sufficiently large, this is at least, say, $1/3$.
 \end{proof}

 We can combine the above ingredients to conclude sub-exponential mixing of the positive-magnetization Glauber dynamics. 

 \begin{proof}[\textbf{\emph{Proof of Theorem~\ref{thm:positive-magnetization-mixing-time}}}]
    By Lemma~\ref{lem:main-projection-restriction}, the inverse spectral gap of the Glauber dynamics is upper bounded by  
    \begin{align*} 
        \gap(P^+)^{-1}\le 2 \gap(P_H)^{-1} \max_{n/2\le i \le n} \gap(P_i)^{-1}\,.
    \end{align*}
    By Lemma~\ref{lem:fast-mixing-of-slice-chain}, as long as $\beta< \frac{1}{4\sqrt{d-1}}$ this is at most $2 \gap(P_H)^{-1} \cdot Cn ^2$. By Corollary~\ref{cor:projection-chain-subexponential-mixing}, for all $\beta \in (\beta_c, \beta_r)$, the birth-death chain $P_H$ has inverse spectral gap that is sub-exponential $e^{o(n)}$. Multiplying these bounds together, we get the desired claim. 
 \end{proof}

 \subsection{Mixing from spin-flip symmetric initializations}

We now use a spectral argument, leveraging the spin-flip symmetry, to go from the positive-magnetization fast mixing of Theorem~\ref{thm:positive-magnetization-mixing-time} to the fast mixing from uniform-at-random initialization of Theorem~\ref{thmIsingMixing}. That theorem is a special case of the following.  

  \begin{corollary}\label{cor:fast-from-spin-flip-symmetric-initialization}
            Suppose $\nu$ is any probability distribution on $\Omega = \{-1,+1\}^n$ that is invariant under global spin flip, i.e., $\nu(\sigma) = \nu(-\sigma)$. Then if $\beta\in (\beta_c,  \frac{1}{4 \sqrt{d-1}})$, the mixing time of Ising Glauber dynamics from a $\nu$-initialization is at most $e^{o(n)}$. 
  \end{corollary}

  \begin{proof}
      We begin by bounding the inverse spectral gap of Glauber dynamics on all of $\Omega$ but restricted to the subspace $L^2_{sym}(\pi)$ which consists only of spin-flip symmetric functions. For any configuration $\sigma$, let $\bar \sigma \in \{-\sigma,\sigma\}$ be its positive magnetization version. Let $P$ be the Glauber dynamics on all of $\Omega$, and recall that $P^+$ is the Glauber dynamics restricted to $\Omega^+$ and $\pi^+ = \pi^+(\cdot \mid \Omega^+)$. Note that $\pi^+(\bar \sigma) = \pi(\sigma)$ for all $\sigma$.\footnote{For $n$ even, a little extra care is needed to handle the $m(\sigma)=0$ configurations. The same proof works mutatis mutandis if the weight assigned to all zero-magnetization configurations under $\pi^+$ is divided by two.} 
     
      For any spin-flip symmetric test function $f$, since $f(\sigma) = f(\bar \sigma)$, 
      \begin{align*}
          \text{Var}_{\pi}(f) = \sum_{\sigma,\sigma'} \pi(\sigma) \pi(\sigma') (f(\sigma) - f(\sigma'))^2 = 4 \sum_{\bar \sigma,\bar \sigma'}\pi^+(\bar \sigma)\pi^+(\bar \sigma')(f(\bar \sigma) - f(\bar\sigma'))^2 = 4\text{Var}_{\pi^+}(f)\,.
      \end{align*}
      On the other hand, its Dirichlet form can be lower-bounded by only considering transitions that preserve the magnetization:  
      \begin{align*}
          \mathcal E(f,f) & = \sum_{\sigma} \sum_{\sigma'}\pi(\sigma) P(\sigma,\sigma') (f(\sigma) - f(\sigma'))^2  \\& \ge \sum_{\sigma}\sum_{\sigma': \text{sign}(m(\sigma')) = \text{sign}(m(\sigma))} \pi(\sigma) P(\sigma,\sigma')(f(\sigma) - f(\sigma'))^2 \\ 
          & \ge 2 \sum_{\bar \sigma} \sum_{\bar \sigma'} \pi(\bar\sigma) P(\bar \sigma,\bar \sigma')(f(\bar \sigma) - f(\bar \sigma'))^2 \,,
      \end{align*}
      where the last inequality used that if $\sigma,\sigma'$ have the same magnetization sign, then $P(\sigma,\sigma') = P(\bar \sigma,\bar \sigma')$. In fact, for $\bar \sigma \ne \bar \sigma'$ in $\Omega^+$ this transition matrix will be exactly the same as $P^+$ and therefore the right-hand side here is exactly ${\mathcal E}^+(f,f)$, the Dirichlet form for the $\Omega^+$-restricted Glauber dynamics. In particular, we obtain 
      \begin{align*}
          \gap_{sym}(P):= \inf_{f\in L^2_{sym}(\Omega): \E_\pi[f] = 0} \frac{\mathcal E(f,f)}{\text{Var}_\pi(f)} \ge \frac{1}{2} \inf_{f\in L^2(\Omega^+): \E_{\pi^+}[f] =0} \frac{{\mathcal E}^+(f,f)}{\text{Var}_{\pi^+}(f)}\,,
      \end{align*}
      and the right-hand side is the spectral gap of the Glauber dynamics restricted to $\Omega^+$. In particular, by Theorem~\ref{thm:positive-magnetization-mixing-time}, we deduce that the spectral gap of Glauber dynamics restricted to the subspace $L^2_{sym}(\pi)$ is at least $\gap_{sym}(P) \ge e^{-o(n)}$. 

      It remains to conclude from here that the total variation distance started from spin-flip symmetric initializations becomes small after time that is sub-exponential in $n$. Since the two subspaces $L^2_{sym}(\pi)$ and $L^2_{antisym}(\pi)$ are orthogonal subspaces, and the operator $P$ keeps each of them fixed (due to spin-flip symmetry of the dynamics), for spin-flip symmetric initial distribution $\nu$, the distribution of $P^t\nu$ is spin-flip symmetric and the function $\frac{P^t\nu}{\pi}$ is in $L^2_{sym}(\pi)$. In particular, using its orthogonal decomposition, we get
      \begin{align*}
          \Big\|\frac{P^t \nu}{\pi} - 1\Big\|_{L^2(\pi)} \le \exp( - 2 \gap_{sym}\cdot t) \Big\|\frac{\nu}{\pi} -1\Big\|_{L^2(\pi)}\,.
      \end{align*}
    Using the sub-exponential lower bound on $\gap_{sym}$ and the exponential lower bound on $\pi(\cdot)$, we deduce up to a change in the constant $C$ that the total variation distance is 
    \begin{align*}
        d_{\tv}\big(P^t\nu(\sigma) ,\pi(\sigma)\big) \le e^{ 2d \beta n} \exp( - t/e^{o(n)})\,,
    \end{align*}
    which can be made at most any $\eps$ if $t \ge e^{o(n)}$ up to a change in the $o(n)$ sequence. 
  \end{proof}

\section{Concentration of Lipschitz functions in the planted model}
\label{secConcentration}

For the planted model introduced in \eqref{defn:plantedconfig-graph}, we may represent it as the following weighted configuration model: fix $\beta > 0$, and fix $\sigma : [n] \to \{-1, +1\}$. Sample a random perfect matching on a vertex set consisting of $d$ copies of $[n]$ where each monochromatic edge is weighted by $e^{\beta}$ and each bichromatic edge is weighted by $1$. 

Note that specifying $\sigma$ is equivalent to specifying $n_+$ and $n_-$ as the number of endpoints (or half-edges) in $[dn]$ with $+1$ and $-1$ spins, respectively (so long as these numbers are divisible by $d$) and letting $\{1, 2, \dots, n_+\}$ be the vertices with $+1$ spins and the remainder with $-1$ spins. For the next several results, we work in this more general context. Let $\cP = \cP_{n_+, n_-}$
be the collection of pairings of vertices where $n_+$ vertices have $+1$ spins and $n_-$ have $-1$ spins. When needed, we specify the underlying spin configuration $\sigma : [n_+ + n_-] \to \{-1, +1\}$.

Our goal is to prove concentration of Lipschitz random variables in this model in the style of results known for concentration of Lipschitz functions of random regular graphs, as in \Cref{lem:switching}. We recall the statement for the reader's convenience. 

\begin{customlem}{\bf \ref{planted-concentration}}
Fix any $d\ge 3$, $\beta>0$, and $\eta \in (-1,1)$. Suppose for some $L$ that the function $X$ defined on pairs of graphs and spin configurations $(G,\sigma)$ is such that for each $\sigma \in \Omega_\eta$, if $G,G'$ differ in exactly one switch, then $|X(G,\sigma) - X(G',\sigma)|\le L$.  Let $\hat \Prob_{\hbg(\hsig)}, \hat \E _{\hbg(\hsig)}$ denote the  distribution~\eqref{defn:plantedconfig-graph} of the planted model conditioned on the spin configuration $\hsig$.

Then there exists $C(d,\beta)>0$ such that  for all $\hsig \in \Omega_{\eta}$ and all $t>0$, 
    $$\hat\Prob_{\hbg(\hsig)}(|X - \hat\E_{\hbg(\hsig)} X | \geq t ) \leq 2\exp\Big(-\frac{t^2}{ C n L^2}\Big)\,.$$
\end{customlem}

To do that, we show that an edge exposure martingale has bounded differences by showing that the weighted configuration measure conditioned on its edge statistics---say, a fixed number of bichromatic edges---can be coupled with a measure where the edge statistics are perturbed slightly. We start by characterizing the distribution of bichromatic edges. We deal separately with the edge cases where one of $n_+, n_-$ is $O(1)$ relative to their sum, and the cases where the number of bichromatic edges is far from the mode. Outside of these cases, the distribution of bichromatic edges behaves nicely in that it satisfies a local central limit theorem with a quantitative error bound. 

In what follows in this section, we abuse notation to now write $\Prob,\E$ for probability and expectation with respect to the $\beta$-\emph{weighted} configuration model on $n_+$ plus vertices and $n_-$ minus vertices, which is exactly the distribution of $\hat \Prob$ given assignment $\hat \sigma\in \Omega_\eta$.

\begin{lemma}\label{lem:localCLT-planted}
    Given $\beta > 0$, let $n_+, n_- \in \mathbb N$. Let $B$ be the number of bichromatic edges in the $\beta$-weighted configuration model. Then, there exist $\mu, \varsigma^2 \to \infty$ as $n_+, n_- \to \infty$ such that for all $k \in [\mu - \varsigma \log \varsigma, \mu + \varsigma \log \varsigma]$,
    $$\Prob(B = k) = \frac{1}{\sqrt{2\pi}\varsigma}e^{-(k-\mu)^2 / 2\varsigma^2}\left(1 + O\Big(\frac{(k-\mu)^3}{\varsigma^4} \vee \frac{1}{\varsigma^2}\Big)\right)\ .$$
\end{lemma}

\begin{proof}
    Let $N = n_+ + n_-$.
    Recall our formula~\eqref{eqn:bi-counting} for $b(k)$, the number of pairings in $\cP = \cP_{n_+,n_-}$ with exactly $k$ bichromatic edges. Let $Z$ be the partition function, meaning the sum of weights over all $P \in \cP$. 
    Consider the probability mass function for $B$, given by
    $$p(k):= \Prob(B=k) = \frac{e^{\beta(N/2-k)}}{Z}b(k) = \frac{e^{\beta(N/2-k)}}{Z}{n_+ \choose k}{n_- \choose k}k!(n_+ - k - 1)!!(n_--k-1)!!\,.$$
    As we did with $b(k)$ after~\eqref{eqn:bi-counting}, this simplifies to
    \begin{align}\label{eq:derived-expression-p(k)}p(k) = \frac{e^{\beta(N/2-k)}}{Z} \cdot \frac{2^kn_+!n_-!}{k!2^{n_+/2}2^{n_-/2}(\frac{n_+-k}{2})!(\frac{n_--k}{2})!}\,.
    \end{align}

To determine the value of $k$ which maximizes $p(k)$, we approximate $\log p(k)$ by using Stirling's formula $\log n! = n\log n - n + \frac12\log(2\pi n) + O\left(\frac{1}{n}\right)$ and taking a smooth continuous approximation. More precisely, analagous to~\eqref{eq:b(k)-approx}, let $\log p(k) = \ell(k) + O\left(\frac{1}{k}\right)$ where $\ell(k)$ is the function on $\mathbb R$ defined as 
     \begin{align*}\ell(k)  &:= \beta\left(\frac{N}{2}-k\right) - \log Z + n_+ \log n_+ + n_- \log n_- - k \log k - \frac{n_+-k}{2}\log \frac{n_+-k}{2}  \\
    &\ \ \ \ -\frac{n_--k}{2}\log \frac{n_--k}{2} + \frac12\left(\log \frac{2n_+n_-}{\pi k(n_+-k)(n_--k)}\right)  + (k - \frac{N}{2}) \log 2 - \frac{N}{2}\ . 
    \end{align*}
    
    The first and second derivatives with respect to $k$ are thus
    \begin{align*} \ell'(k) &= -\beta - \log k + \frac12 \log \frac{n_+ - k}{2} + \frac12\log \frac{n_- - k}{2} - \frac{1}{2} \left(\frac{1}{k} - \frac{1}{n_+-k} - \frac{1}{n_--k} \right) + \log 2 + O\left(\frac{1}{k^2}\right)\ ,
    \end{align*} 
    \begin{align*}
        \ell''(k) &= - \frac{1}{k} - \frac{1}{2(n_+-k)} - \frac{1}{2(n_- - k)} + \left(\frac{1}{k^2} + \frac{1}{(n_+-k)^2} + \frac{1}{(n_--k)^2}\right) + O\left(\frac{1}{k^3}\right)\ .
    \end{align*}
    The second derivative is strictly negative for $n_+, n_-$ larger than an absolute constant $C_0$, telling us that $\ell(k)$ is strictly concave and thus has a unique global maximum. Call 
    \begin{align*}
        \mu:= {\arg\max}_k \ell(k) \qquad \text{and} \qquad \varsigma^2 :=  - \frac{1}{\ell''(\mu)}\,.
    \end{align*}
    (For intuition, it is straightforward to verify that if $n_+,n_-$ go to infinity together proportionately, then $\mu$ and $\varsigma^2$ both scale linearly with $n_+,n_-$, as everything can be rewritten in terms of proportions $\frac{n_+}{n_- + n_+}, \frac{n_-}{n_+ + n_-}$ and $\frac{k}{n_+ + n_-}$ as done in Lemma~\ref{lem:partition-first-moment}.)

To obtain the claim, we compute the Taylor series expansion of $\ell(k)$ about $k = \mu$: 
$$\ell(k) = \ell(\mu) - (k-\mu)\ell'(\mu) + (k-\mu)^2 \frac{\ell''(\mu)}{2}+R(k)$$
where $|R(k)| \leq \frac{\max_{x \in [k \pm \mu]} \ell^{'''}(k)}{6}(k-\mu)^3$.
Note that
$$\ell'''(k) = \left(\frac{1}{k^2} - \frac{1}{2(n_+-k)^2} - \frac{1}{2(n_--k)^2}\right) + \left(\frac{-2}{k^3} + \frac{2}{(n_+-k)^3} + \frac{2}{(n_--k)^3}\right) + O\left(\frac{1}{k^4}\right)\ ,$$
so $|R(k)| \leq O(\frac{(k-\mu)^3}{\varsigma^4})$.

The probability mass function of $B$ is then 
$$\Prob(B = k) = e^{\ell(\mu)}e^{-(k-\mu)^2 / 2\varsigma^2}\left(1 + O\left(\frac{(k-\mu)^3}{\varsigma^4}\right)\right)$$
using the estimate $e^x = 1+O(x)$ for $x$ small.
Lastly, we claim that we have $e^{\ell(\mu)} = \frac{1}{\sqrt{2\pi}\varsigma}\left(1 + O\left(\frac{1}{\varsigma^2}\right)\right)$. This follows from noting that $1 = \sum_k \Prob(B = k)$ which can be approximated by an integral $\int_{-\infty}^{\infty} e^{\ell(\mu)}e^{-(x-\mu)^2/2\varsigma^2}dx$ with $O(1/\varsigma)$ additive error. This Gaussian integral gives the the claimed normalizing constant.
\end{proof}

\begin{lemma}\label{bounded-coupling}
    Let $\beta > 0$, $N, n_+, n_- \in \mathbb N$ such that $n_+ + n_- = N$. 
    Let $\sigma_0:[N] \to \{+1, -1\}$ such that $n_+ = |\sigma_0^{-1}(+1)|$ and $n_- = |\sigma_0^{-1}(-1)|$, and let $\sigma_1 : [N] \to \{+1, -1\}$ such that $n_+ +1 = |\sigma_0^{-1}(+1)|$ and $n_- - 1= |\sigma_0^{-1}(-1)|$.
    Let $\pi_0, \pi_1$ be the respective $\beta$-weighted configuration models. 
    Let $B_0, B_1$ be the random variables counting the number of bichromatic edges under $\pi_0, \pi_1$, respectively.

    Then there exists $C = C(d,\beta) > 0$ and a coupling of $B_0, B_1$ such that $\Prob(|B_0 - B_1| \leq C)=1$. 
\end{lemma}

\begin{proof}
Let $C_0$ be a sufficiently large, fixed, constant. First note that if either $n_+$ or $n_-$ is bounded by $C_0$, 
the claim trivially follows by taking $C>C_0$ because the number of bichromatic edges is bounded by the smaller part's vertex count. 
Thus, we may assume that $n_+, n_-$ are sufficiently large, and in particular, in that case Lemma~\ref{lem:localCLT-planted} still applies with the hidden constants in the big-O being between some fixed $\delta_0$ and $\delta_0^{-1}$.

By Strassen's criterion for stochastic domination, it is enough to show that there exists $C>0$ such that for every $k$, we have
\begin{align}\label{eq:nts-Strassen}
\Prob(B_0 \leq k-C) \leq \Prob(B_1 \leq k) \leq \Prob(B_0 \leq k+C)\ .\end{align}
Let $\mu_0, \varsigma^2_0$ and $\mu_1, \varsigma^2_1$ be the $\mu$ and $\varsigma$ given by \Cref{lem:localCLT-planted} for $B_0$ and $B_1$, respectively. 

\medskip
\noindent\emph{Eq.~\eqref{eq:nts-Strassen} for $k$ in the tail}. 
First consider $k$ in the lower tail, $\mu_0 - k > \varsigma_0\log(\varsigma_0)$. Let $\gamma = \frac{c\varsigma_0}{\log \varsigma_0}$ for some $c > 0$ (to be specified). We will show that 
\begin{align}\label{eq:wts-tails}\Prob(B_0 \leq k - C) \leq \gamma\Prob(B_0 = k-C) \leq \Prob(B_1 = k)\,.
\end{align}
To show the first inequality in~\eqref{eq:wts-tails}, let $p_0(k) = \Prob(B_0 = k)$ and $r_0(k) = \frac{p_0(k-1)}{p_0(k)}$. By log-concavity of $p_0(k)$, the function $r_0(k)$ is increasing. Then
$$\Prob(B_0 \leq k) = \sum_{i = 0}^k p_0(i) \leq p_0(k)\sum_{i=0}^k r_0(k)^{i} \leq \frac{p_0(k)}{1-r_0(k)}$$
The ratio $\frac{1}{1-r_0(k)}$ is maximized at the boundary, where $|k - \mu_0| = \varsigma_0\log(\varsigma_0)$. 
By approximating $r_0(k)$ by $\exp(-\ell'(k))$ and evaluating $\ell'(k)$ at $k = \mu_0 - \varsigma_0 \log (\varsigma_0)$, we get that 
\begin{align*}
    r_0(k) = \exp( \ell(k-1) - \ell(k)) = e^{-(k-\mu)\ell''(\mu)} (1+O(1/\varsigma^2))
\end{align*}
and plugging in $\ell''(\mu) = -\frac{1}{\varsigma_0^2}$, we get 
$\frac{1}{1-r_0(k)} = \frac{\varsigma_0}{\log(\varsigma_0)} (1+O(1/\varsigma^2))$. This gives the first inequality of~\eqref{eq:wts-tails} with $\gamma = \frac{c\varsigma_0}{\log \varsigma_0}$ for $c$ a large constant. 

For later reference, also note that this argument allows us to show that the sum of the tail probabilities is super-polynomially small in $\varsigma$. Using \Cref{lem:localCLT-planted}, we can bound the error at $k = \mu_0 - \varsigma_0\log(\varsigma_0)$ by $R(k) = O(\frac{\log^3(\varsigma_0)}{\varsigma_0})$. Thus, applying the local CLT statement gives \begin{equation}\label{tailbound}
\Prob(B_0 \leq \mu_0 - \varsigma_0 \log(\varsigma_0)) \leq \frac{c\varsigma_0}{\log \varsigma_0} \cdot \frac{1}{\sqrt{2\pi} \varsigma_0}e^{-\log^2(\varsigma_0)/2}\left(1 + O\left(\frac{\log^3(\varsigma_0)}{\varsigma_0}\right)\right)
\end{equation}which is dominated by the $\exp(-\Omega(\log^2 \varsigma_0))$ term as $\varsigma_0 \to \infty$.

To show the second inequality in~\eqref{eq:wts-tails}, we want to compare the point probabilities of $B_0$ and $B_1$. Using the previously derived expression for $p(k)$~\eqref{eq:derived-expression-p(k)}, we have

\begin{align}\frac{\Prob(B_0 = k-C)}{\Prob(B_1 = k)} &= e^{2C\beta}\frac{Z_1}{Z_0}\frac{n_-(k)\cdots (k-C+1)}{2^C(n_++1)}\frac{(\frac{n_+-k+1}{2})!}{(\frac{n_+-k+C}{2})!} \frac{(\frac{n_--k-1}{2})!}{(\frac{n_--k+C}{2})!} \nonumber \\ 
&\leq \frac{Z_1}{Z_0}\frac{n_-}{n_+}\left(\frac{e^{2\beta}k}{2(n_+-k) (n_--k)}\right)^C \,.\label{eqn:pointprobs}
\end{align}
 We first bound the ratio $\frac{Z_1}{Z_0}$.     Recall $Z_0 = \sum_{P \in \cP_{n_+, n_-}} w_P(\sigma_0)$ where $w_P(\sigma_0)$ is the weight of $P$ on spin configuration $\sigma_0$. Notice that we may define a bijection between $\cP_{n_+, n_-}$ equipped with spins $\sigma_0$ and $\cP_{n_++1, n_--1}$ equipped with $\sigma_1$ by preserving the matching and flipping the spin of a distinguished vertex $v$ from $-1$ to $+1$.

    Let $\cP^+_{n_+, n_-}$ be the set of matchings where $v$ is matched to a vertex with a $+1$ spin and analogously define $\cP^-_{n_+, n_-}$. Then
    \begin{align*}
        Z_0 &= \sum_{P \in \cP^+_{n_+, n_-}}w_P(\sigma_0) + \sum_{P \in \cP^-_{n_+, n_-}}w_P(\sigma_0)\\
        &= \sum_{P \in \cP^+_{n_+, n_-}}w_{P \setminus v}(\sigma_0) + \sum_{P \in \cP^-_{n_+, n_-}}w_{P \setminus v}(\sigma_0)e^{\beta}\\
        &= \sum_{P \in \cP^+_{n_++1, n_--1}}w_{P \setminus v}(\sigma_1) + \sum_{P \in \cP^-_{n_++1, n_--1}}w_{P \setminus v}(\sigma_1)e^{\beta}\\
        &= \sum_{P \in \cP^+_{n_++1, n_--1}}w_{P}(\sigma_1)e^{-\beta} + \sum_{P \in \cP^-_{n_++1, n_--1}}w_{P}(\sigma_1)e^{\beta}\,.
    \end{align*}
    Then calling $N(v)$ the (random) vertex to which $v$ is matched, 
    $$\frac{Z_0}{Z_1} = e^{-\beta}\Prob(N(v) \in \sigma_1^{-1}(+1)) + e^{\beta}\Prob(N(v) \in \sigma_1^{-1}(-1)) \geq e^{-\beta}\ .$$

    If $n_-/n_+ \le C_0$, then since $k$ is in the lower tail, we use that $k/(n_+ - k)(n_- -k)\le \mu/(\varsigma^2 \log^2\varsigma) =o(1)$ and the inequality holds; if $n_-/n_+ \ge C_0$, then $n_- - k\ge n_- - n_+$ so as long as the exponent $C>2$, then $k/(n_- -k)$ can be used to kill both the $n_-/n_+$ and any other constants, leaving the inequality true again.

The argument for $k$ in the upper tail is analogous. In this case, $\Prob(B_0 \leq k-C) \leq \Prob(B_1 \leq k)$ follows from showing $\Prob(B_0 = k-C+1) \geq \gamma'\Prob(B_1= k) \geq  \Prob(B_1 > k)$ for a specified $\gamma'$ in terms of $\varsigma_0$. As before, we first find $\gamma'$ by writing $\Prob(B_1 > k)$ as a product of $p_1(k)$ with a geometric series. We then compare the point probabilities $\frac{\Prob(B_1 = k)}{\Prob(B_0 = k-C+1)}$ which is of the form of the reciprocal of the expression of~\eqref{eqn:pointprobs}. We can see that by the assumption of $k$ being in the upper tail and taking $C_0$ large enough and $C>2$, the desired inequalities follow.

\medskip
\noindent\emph{Eq.~\eqref{eq:nts-Strassen} for $k$ in the moderate deviations}. 
For $|k-\mu_0| \leq \varsigma_0 \log (\varsigma_0)$, we aim to apply \Cref{lem:localCLT-planted} to compare $B_0$ and $B_1$. First, note that perturbing $n_+, n_-$ by $\pm 1$ changes $\mu$ and $\varsigma^2$ by at most a uniform additive constant. Indeed, let $\ell_0(k) = \ell_{n_+, n_-}(k)$ and $\ell_1(k) = \ell_{n_++1, n_--1}(k)$. Let $K > 0$ be a constant. 
By the Mean Value Theorem, $\ell_0'(\mu_0+K) = K\ell_0''(\xi)$ for some $\xi \in (\mu_0, \mu_0 + K)$. By evaluating $\ell_1'(k)$, we also have $\ell_1'(\mu_0+K) - \ell_0'(\mu_0+K) \approx \frac{1}{2(n_+-(\mu_0+K))} - \frac{1}{2(n_--(\mu_0+K))}$. Then we have
$$\ell_1'(\mu_0+K) = K \ell_0''(\xi) + \frac{1}{2(n_+-(\mu_0+K))} - \frac{1}{2(n_--(\mu_0+K))}\ .$$

By analyzing $\ell_0''(k)$, we also see that $\ell''(\xi) = (1+o(1))\ell''(\mu_0)$ (as $n_+, n_- \to \infty$) for $\xi \in (\mu_0, \mu_0+K)$, and $|\ell_0''(\mu_0)| > |\ell_1'(\mu_0+K) - \ell_0'(\mu_0+K)|$. However, we also know that $\ell_0''(k) < 0$ for all $k$. This implies that for $K$ sufficiently large (e.g. $K = 2$), we have $\ell_1'(\mu_0 + K) < 0$ whereas $\ell_1'(\mu_0 - K) > 0$, so the unique root $\mu_1$ of $\ell_1$ must lie in the interval $(\mu_0 - K, \mu_0 + K)$.

We perform a similar analysis to bound $|\varsigma_1^2 - \varsigma_0^2|$ by a constant.
Note that this implies $|\varsigma_1 - \varsigma_0| \leq \frac{K'}{\varsigma_0}$ for some constant $K' > 0$.

Now to show Strassen's condition holds for $k$ in this case, we use \eqref{tailbound} for $k$ in the tail and for larger $k$ add up the probability mass function estimates given by \Cref{lem:localCLT-planted} to get an estimate on the cumulative distribution function as
\begin{align*} \Prob(B_0 \leq k-C) &= \Prob(B_0 \leq \mu_0 - \varsigma_0\log\varsigma_0) + \\
&\qquad  + \sum_{j=\mu_0 - \varsigma_0\log(\varsigma_0)}^{k-C} \frac{1}{\sqrt{2\pi}\varsigma_0}e^{-(j-\mu_0)^2 / 2\varsigma_0^2}\left(1 + O\left(\frac{(j-\mu_0)^3}{\varsigma_0^4}\right)\right)
\end{align*}
The first term was already bounded by $e^{ - \Omega(\log^2 \varsigma_0)}$. 
To bound the error from the big-$O$ term, we approximate the sum by an integral, up to an error of $O(1/\varsigma_0)$, to get 
$$\int_{-\log(\varsigma_0)}^{(k-C-\mu_0)/\varsigma_0} \frac{1}{\sqrt{2\pi}}e^{-x^2/2}\left(\frac{C_3x^3}{\varsigma_0}\right)dx$$ where $x = \frac{j-\mu_0}{\varsigma_0}$ to obtain an additive error of $O(1/\varsigma_0)$.  Thus, again switching the sum to an integral, we end up with 
$$\Prob(B_0 \leq k-C) = \Phi\left(\frac{k-C-\mu_0}{\varsigma_0}\right)+ O\left(\frac{1}{\varsigma_0}\right)\,,$$
where $\Phi$ is the Gaussian CDF. 
We obtain a similar expression for $B_1$ as 
$$\Prob(B_1 \leq k) = \Phi\left(\frac{k - \mu_1}{\varsigma_1}\right) + O\left(\frac{1}{\varsigma_1}\right)\ .$$
To conclude that $\Prob(B_0 \leq k - C) \leq \Prob(B_1 \leq k)$, let $E_0$ be the error term from the left-hand side and $E_1$ be that of the right-hand side. We want $E_0 - E_1 \leq \left(\frac{k-\mu_1}{\varsigma_1} - \frac{k-C-\mu_0}{\varsigma_0}\right)\Phi'(\xi)$ where $\frac{k-C-\mu_0}{\varsigma_0} < \xi < \frac{k-\mu_1}{\varsigma_1}$. By bounding $E_0 - E_1 \leq \frac{c_1}{\varsigma_0}$ and $\Phi'(\xi) \geq c_2 > 0$, and recalling that $\mu_1 \leq \mu_0 + K$, we may take $C$ large enough in terms of $c_1, c_2, K$ so that the inequality holds.
\end{proof}

\begin{lemma}\label{lem:Doob-martingale-is-bounded-difference-planted}
    Let $f$ be a function on spin configuration and $\beta$-weighted configuration models  $(\sigma,P)$ and suppose there exists $L$ such that for all $\sigma\in \Omega_\eta$, one has $|f(P) - f(P')|\le L$ if $P,P'\in \mathcal P$ differ by a single switch.

    Fix $\sigma \in \Omega_\eta$ and draw $\hat G \sim \hbg(\sigma)$. Let $X_0, X_1, \dots$ be the corresponding edge exposure martingale for this planted configuration model, i.e. $X_{m} = \E[ f(\sigma,\hat G)  \mid H_m \subset \hat G]$ where $H_i$ denotes the first $i$ pairings (under an arbitrary advanced ordering) of $\hat G$. Then there exists $c(L,d,\beta)$ such that this martingale has $c$-bounded differences, $|X_{m+1} - X_m| < c$. 
\end{lemma}

\begin{proof}
Suppose $ij$ is the $(m+1)$'th pairing exposed. Then we can write 
\begin{align*}X_{m+1} &  = \E[f \mid (H_{m} \cup ij)\subset \hat G] \\ 
X_m &  = \sum_k \E[f \mid  (H_m \cup ik) \subset \hat G] \cdot \Prob(ik \in \hat G \mid  H_m \subset \hat G)
\end{align*}
so that 
\begin{align*} |X_{m+1} - X_m| &\leq \sum_k \Prob(ik \mid  H_m ) |\E[f \mid  H_m \cup ij] - \E[f \mid  H_m \cup ik]|\,.
    \end{align*}
    Considering now the difference in expectations, we have 
    \begin{align*}
        &\E[f\ |\ H_m \cup ij] - \E[f\ |\ H_m \cup ik]= \sum_{P' \in \Omega_j} \pi_j(P')f(P') - \sum_{P'' \in \Omega_k} \pi_k(P'')f(P'')
    \end{align*}
    where for all $\ell$ we define $\Omega_{\ell} = \{P' : (H_m \cup i\ell) \subset P'\}$ and $\pi_{\ell}(P') = \Prob(P' \mid \Omega_\ell)$.

Our goal is to construct a coupling $\pi$ on $\Omega_j \times \Omega_k$ so that we may rewrite this as
$$\sum_{(P', P'') \in \Omega_j \times \Omega_k} \pi(P', P'')(f(P') - f(P''))\ .$$
Observe that the laws $\pi_j$ and $\pi_k$ are simply uniformly random matchings on the vertices not matched under $H_m \cup ij$ and $H_m\cup ik$ respectively. If $j$ and $k$ have the same spin under $\sigma$, then the $n_+,n_-$ on the remainder are the same for $\pi_j,\pi_k$ and there is a trivial coupling which sets $j$'s partner under $P''$ to be $k$'s partner under $P'$ and uses the identity coupling on the remaining vertices, leading to a switch distance of two.

Now suppose $j,k$ have different spin under $\sigma$. Let $B_j, B_k$ be the number of bichromatic edges under $\pi_j, \pi_k$, respectively, on the complement of $H_m$. 
By \Cref{bounded-coupling} applied to the sets of un-matched vertices, we know there exists a coupling such that with probability 1, $|B_j - B_k| \leq C$ for some constant $C > 0$. Now we claim that for each value $C' \leq C$, we may couple the conditional measures to have a $2C'$ difference in their bichromatic edges.

\begin{claim}\label{claim:lipschitz-coupling}
    Conditioned on $B_j = b_j, B_k = b_k$ such that $|b_j-b_k| \leq C$, there exists a coupling of the conditional measures $\pi_j(\cdot \mid B_j = b_j), \pi_k(\cdot \mid B_k = b_k)$ such that with probability $1$, the resulting pairings $P',P''$ differ in at most $2C$ switches. 
\end{claim}

Let us defer the proof of the claim and first conclude the proof of Lemma~\ref{lem:Doob-martingale-is-bounded-difference-planted}.

Combining \Cref{bounded-coupling} and \Cref{claim:lipschitz-coupling}, there is a coupling $\pi$ of $P' \sim \pi_j$ and $P'' \sim \pi_k$ such that $P', P''$ differ in at most $2C$ switches. By the Lipschitz assumption and the triangle inequality, $|f(P') - f(P'')| \leq 2 L C$. Then, 
\begin{align*}
    |\E[f\ |\ H_m \cup ij] - \E[f\ |\ H_m \cup ik]| &\leq \sum_{(P', P'') \in \Omega_j \times \Omega_k} \pi(P', P'')|f(P') - f(P'')| \leq 2 LC\ . \qedhere
\end{align*}
\end{proof}

By applying Azuma's Inequality, we immediately obtain the desired \Cref{planted-concentration} as a consequence of Lemma~\ref{lem:Doob-martingale-is-bounded-difference-planted}. It remains to prove the deferred claim.

\begin{proof}[Proof of Claim~\ref{claim:lipschitz-coupling}]
Recall our distributions $\pi_j$ and $\pi_k$ are over matchings conditioned on some fixed submatching $H_m$ (which we thus ignore) and the vertex $i$ being matched to some vertex $j$ in the former and $k$ in the latter. We construct a coupling between $\pi^0 = \pi_j( \cdot \mid  B_j = b_j)$ and $\pi^1 = \pi_k( \cdot \mid  B_k = b_k)$ as follows. 

We will sample pairings $P_0 \sim \pi^0, P_1 \sim \pi^1$ one edge at a time. Observe that one way to generate a pairing $P$ with $b$ bichromatic edges is the following: write a tuple of length $\frac{nd}{2} =: \frac{n_+ + n_-}{2}$ where each entry is either $-1, 0$, or $1$ and where the total number of $0$'s is equal to $b$ and the total number of $1$'s is $m_+ := \frac{n_+ - b}{2}$ (and thus, the total number of $-1$'s is $m_- := \frac{n_--b}{2}$). To generate the $i$th edge of $P$, look at the $i$th entry of the tuple. If it is $0$, choose a uniformly random pair of $+1$ and $-1$ vertices from the remaining unmatched vertices and add this pair to $P$. Similarly, if the $i$th entry is $1$, choose a uniformly random pair of unmatched $(+1, +1)$ vertices, and if $-1$, choose a uniformly random pair of unmatched $(-1, -1)$ vertices.

We can verify that this sampling procedure results in the uniform distribution over pairings with $b$ bichromatic edges by showing that the probability it gives to a pairing is invariant under adjacent transpositions.  
Given $n_+, n_-$ available vertices with $+1, -1$ spins, respectively, choosing a bichromatic edge followed by a $(+1, +1)$ edge has probability $$\frac{b}{n_+n_-} \cdot \frac{m_+}{{n_+ - 1 \choose 2}} = \frac{2bm_+}{n_+n_-(n_+-1)(n_+-2)}$$ 
whereas choosing the $(+1,+1)$ edge followed by the bichromatic edge has probability $$\frac{m_+}{{n^+ \choose 2}} \cdot \frac{b}{(n_+ - 2)n_-} = \frac{2bm_+}{n_+(n_+-1)(n_+-2)n_-}\ .$$
The computations for transposing a bichromatic and $(-1,-1)$ edge as well as two monochromatic edges are analogous.

Without loss of generality, suppose $b_j \le  b_k$. 

First, note that $k$ is available to match in $P_0$ and $j$ is available to match in $P_1$. We may generate $P_0, P_1$ by first specifying the partners of $k, j$, respectively. 
We first sample the partner $u_0$ of $k$ in $P_0$ according to the marginal of $k$, and similarly $u_1$ the partner of $j$ in $P_1$. Note that $u_0$ is a random unpaired vertex of $P_1$ and $u_1$ is a random unpaired vertex of $P_0$. 

In the case that $0 < b_j \le b_k < \frac{n_+ + n_-}{2}$, do the following. If $u_0 \neq u_1$, then sample a random vertex $u_2$ such that $\sigma(u_2) = \sigma(u_0)$ and add the edges $u_0u_2$ to $P_1$ and $u_1u_2$ to $P_0$. If $\sigma(u_0) = \sigma(u_1)$, both edges are monochromatic whereas if $\sigma(u_0) \neq \sigma(u_1)$, this will add one monochromatic edge to $P_1$ and one bichromatic edge to $P_0$.
In the case that $b_k = \frac{n_+ + n_-}{2}$ or $b_j = 0$, repeat the above but sample $u_2$ such that $\sigma(u_2) = \sigma(u_1)$.

We are now left with the same remaining set of vertices to match in $P_0$ and $P_1$, with $B_0 \in \{b_j, b_j \pm 1, b_j \pm 2\}$ remaining bichromatic edges to place in $P_0$ and $B_1 \in \{b_k, b_k \pm 1, b_k \pm 2\}$ remaining bichromatic edges to place in $P_1$. 

If $B_1 = B_0$, the identity coupling concludes the proof. Without loss of generality assume $B_1 - B_0 = D > 0$. Choose a random set of vertices $x_1, x_2, \dots, x_D$ with $+1$ spins and $y_1, y_2, \dots, y_D$ with $-1$ spins. Add the edges $\{x_iy_i : 1 \leq i \leq D\}$ to $P_1$ and the edges $\{x_{2i-1}x_{2i} : 1 \leq i \leq D/2\} \cup \{y_{2i-1}y_{2i} : 1 \leq i \leq D/2\}$ to $P_0$. 
On the remaining vertices, we couple the pairings exactly by using the same set of edges to complete the pairings, as described in the sampling procedure above. 
\end{proof}

\subsection*{Acknowledgments}
RG supported in part by NSF CAREER grant 2440509 and NSF DMS grant 2246780.  WP supported in part by NSF grant CCF-2309708. CY supported in part by NSF grant DMS-1928930 while in residence at the Simons--Laufer Mathematical Sciences Institute during the Spring 2025 semester.

\bibliographystyle{plain}
\bibliography{references}

\end{document}